\documentclass[11pt,reqno,a4paper]{amsart}

\usepackage{amsmath}
\usepackage{enumerate}
\usepackage{amssymb}
\usepackage{amscd}
\usepackage{amsthm}
\usepackage{amsfonts}
\usepackage{graphicx}
\usepackage{hyperref}
\usepackage[matrix, arrow, curve]{xy}
\usepackage{etoolbox}

\usepackage{fouriernc}
\usepackage[T1]{fontenc}

\numberwithin{equation}{section}

\patchcmd{\subsection}{-.5em}{.5em}{}{}
\patchcmd{\subsubsection}{-.5em}{.5em}{}{}

\newcommand{\cA}{\mathcal{A}}

\newcommand{\cC}{\mathcal{C}}

\newcommand{\cH}{\mathcal{H}}

\newcommand{\cS}{\mathcal{S}}

\newcommand{\F}{\mathbb{F}}

\newcommand{\bN}{\mathbb{N}}

\newcommand{\bR}{\mathbb{R}}

\newcommand{\bZ}{\mathbb{Z}}
\newcommand{\R}{\mathbb{R}}
\newcommand{\C}{\mathbb{C}}

\newcommand{\ra}{\rightarrow}

\newcommand{\qor}{\quad \textrm{or} \quad}
\newcommand{\qand}{\quad \textrm{and} \quad}

\newcommand\subsetsim{\mathrel{%
\ooalign{\raise0.2ex\hbox{$\subset$}\cr\hidewidth\raise-0.8ex\hbox{\scalebox{0.9}{$\sim$}}\hidewidth\cr}}}
\newcommand{\eps}{\varepsilon}

\DeclareMathOperator{\supp}{supp}

\newcommand{\Q}{\mathbb Q}
\newcommand{\Z}{\mathbb Z}
\renewcommand{\L}{\mathfrak}
\renewcommand{\epsilon}{\varepsilon}

\theoremstyle{theorem}
\newtheorem{theorem}{Theorem}[section]
\newtheorem{corollary}[theorem]{Corollary}
\newtheorem{proposition}[theorem]{Proposition}
\newtheorem{lemma}[theorem]{Lemma}

\theoremstyle{definition}
\newtheorem{definition}[theorem]{Definition}
\newtheorem{convention}[theorem]{Convention}
\newtheorem{remark}[theorem]{Remark}
\newtheorem*{example}{Example}
\newtheorem*{problem}{Problem}

\renewcommand{\phi}{\varphi}

\begin{document}

\title{Spectral theory of approximate lattices in nilpotent Lie groups}

%  Author I information
\author{Michael Bj\"orklund}
\address{Department of Mathematics, Chalmers, Gothenburg, Sweden}
\email{micbjo@chalmers.se}
\thanks{}

%    Author II information
\author{Tobias Hartnick}
\address{Mathematics Department, Technion, Haifa 32000, Israel}
\curraddr{}
\email{hartnick@tx.technion.ac.il}
\thanks{}

%    Author II information
%\author{Felix Pogorzelski}
%\address{Mathematics Department, Technion, Haifa 32000, Israel}
%\curraddr{}
%\email{felixp@tx.technion.ac.il}
%\thanks{}


\subjclass[2010]{Primary: ; Secondary: }

\date{}

\dedicatory{}

\maketitle

\begin{abstract} We show that an approximate lattice in a nilpotent Lie group admits a relatively dense subset of central $(1-\epsilon)$-Bragg peaks for every $\epsilon > 0$. For the Heisenberg group we deduce that the union of horizontal and vertical $(1-\epsilon)$-Bragg peaks is relatively dense in the unitary dual. More generally we study uniform approximate lattices in extensions of lcsc groups. We obtain necesary and sufficient conditions for the existence of a continuous horizontal factor of the associated hull-dynamical system, and study the spectral theory of the hull-dynamical system relative to this horizontal factor.
\end{abstract}

\section{Introduction}

\subsection{Goals and scope}
This article is concerned with the spectral theory of hulls of uniform approximate lattices in locally compact second countable (lcsc) groups. Uniform approximate lattices, as introduced in \cite{BH1}, are a common generalization of uniform lattices in lcsc groups and Meyer sets \cite{MeyerBook, Moody} in abelian lcsc groups. With a uniform approximate lattice $\Lambda$ in a lcsc group $G$ one associates a topological dynamical system $G \curvearrowright \Omega_\Lambda$, where the so-called \emph{hull} $\Omega_\Lambda$ is defined as the orbit closure of $\Lambda$ in the Chabauty space of $G$. 

If $\Lambda$ is a uniform lattice in $G$, then $\Omega_\Lambda$ is just the homogeneous space $G/\Lambda$, hence carries a unique $G$-invariant measure. In general, $\Omega_\Lambda$ need not admit an invariant measure, nor does such a measure have to be unique if it exists. However, in many cases of interest, in particular if $G$ is amenable or $\Lambda$ is a ``model set'' \cite{BHP1}, there does exist an invariant measure $\nu_\Lambda$ on $\Omega_\Lambda$, and in this case one is interested in the associated Koopman representation $G \curvearrowright L^2(\Omega_\Lambda, \nu_\Lambda)$ and its (irreducible) subrepresentations.

In the case where $G$ is abelian there is a well-developed spectral theory of hull-dynamical systems. Since Meyer sets in lcsc abelian groups are mathematical models of quasi-crystals, this theory can be applied to study diffraction of quasi-crystals; in this context one is particularly interested in understanding the experimentally observable phenomenon of sharp Bragg peaks in the diffraction picture. We refer the reader to the extensive bibliography of \cite{BG13} for references on this classical theory.

One of our initial motivations was to establish the existence of a relatively dense set of large Bragg peaks for uniform approximate lattices in the Heisenberg group, which we expect to be relevant for modelling quasi-crystals in Euclidean spaces in the presence of a magnetic field. To achieve this we had to develop a relative spectral theory comparing the hull of an approximate lattice in the Heisenberg group to the hull of its projection in Euclidean space.

We then discovered that our method applies much more widely. Namely, we were able to establish a relative version of spectral theory which applies to hulls of large classes of uniform approximate lattices in quite general extensions of lcsc groups. This relative theory is specific to the non-abelian case and has no counterpart in the classical theory.

\subsection{General setting: Uniform approximate lattices aligned to an extension}

While we will obtain our strongest results in the setting of nilpotent Lie groups, most of our basic results work in the following general setting. Let $G$ be a lcsc group which admits a non-trivial closed normal subgroup $N$, set $Q: = G/N$ and denote by $\pi: G \to Q$ the canonical projection. Given a uniform approximate lattice $\Lambda \subset G$ we consider the projection $\Delta := \pi(\Lambda)$ and the fibers of $\pi|_{\Lambda}:\Lambda \to \Delta$ which we identify with subsets of $N$ by a suitable translation.
\begin{proposition}[Alignment with respect to a projection]\label{IntroAligned} Let $\Lambda\subset G$ be a uniform approximate lattice. Then the following conditions on $\Lambda$ are equivalent:
\begin{enumerate}[(i)]
\item $\Delta = \pi(\Lambda)$ is a uniform approximate lattice in $Q$
\item Some fiber of $\pi|_\Lambda$ is a Meyerian subset of $N$.
\item All fibers of $\pi|_{\Lambda}$ over a relatively dense subset of $\pi(\Lambda)$ are Meyerian subsets of $N$.
\end{enumerate}
\end{proposition}
Here, a \emph{Meyerian subset} of a lcsc group is a subset which satisfies all the axioms of a uniform approximate lattice except for symmetry under inversion; in the abelian case Meyerian subsets are just the \emph{harmonious sets} of Yves Meyer (cf.\ Definition \ref{Meyerian} below). If $\Lambda$ satisfies the equivalent conditions of Proposition \ref{IntroAligned} then we say that $\Lambda$ is \emph{aligned} with $N$. 

Given an approximate lattice $\Lambda$ in $G$ which is aligned with $N$ with projection $\Delta := \pi(\Lambda)$, it is natural to ask whether the associated hull-dynamical systems $G \curvearrowright \Omega_{\Lambda}$ and $Q \curvearrowright \Omega_{\Delta}$ can also be related. Note that if $\Lambda$ is a subgroup, then $\Omega_\Delta$ is always a continuous factor of $\Omega_\Lambda$. In our case this is not the case in full generality; however, we have:
\begin{theorem}[Existence of continuous horizontal factors]\label{IntroHorizontalFactor} Assume that $\pi: G \to Q$ admits a continuous section. Let $\Lambda$ be an $N$-aligned uniform approximate lattice and $\Delta := \pi(\Lambda)$.
\begin{enumerate}[(i)]
\item The projection $\pi$ induces a continuous factor map $\pi_*: \Omega_\Lambda \to \Omega_\Delta$ if and only if $\Lambda$ has uniformly large fibers in the sense that there exists a uniform constant $R>0$ such that all fibers of $\Lambda$ are $R$-relatively dense in $N$.
\item The condition that $\Lambda$ have uniformly large fibers can always be arranged by passing to a commensurable uniform approximate lattice, which can be chosen to be contained in $\Lambda^2$.
\end{enumerate}
\end{theorem}

\subsection{Examples of aligned uniform approximate lattices from nilpotent Lie groups}
Natural examples of aligned uniform approximate lattices arise from nilpotent Lie groups:
\begin{theorem}[Universally aligned characteristic subgroups]\label{IntroTower} Let $G$ be a nilpotent $1$-connected Lie group. If $G$ is non-abelian, then there exists a non-trivial characteristic abelian subgroup $N\lhd G$ such that every uniform approximate lattice $G$ is aligned with $N$. In particular, every uniform approximate lattice in $G$ gives rise to a uniform approximate lattice in $G/N$.
\end{theorem}
Our original proof established that $N$ can always be chosen to be $Z(C_G([G,G]))$, the center of the centralizer of the commutator. After finishing the work on this article we learned from S. Machado that he has established that all uniform approximate lattices in nilpotent Lie groups are subsets of model sets \cite{Machado}. By standard arguments concerning lattices in nilpotent Lie groups (as explained below) this implies:
\begin{theorem}[S. Machado] The group $N$ in Theorem \ref{IntroTower} can always be chosen to be the center $Z(G)$ of $G$ (or the commutator subgroup $[G,G]$ of $G$).
\end{theorem}
In view of this result will mostly focus on \emph{central} extensions (rather than general abelian extensions) below. This allows us to avoid some major technical difficulties related to non-central extensions.

\subsection{The vertical part of the spectrum}
From now on let $G$ be a connected lcsc group, $\Lambda \subset G$ a uniform approximate lattice and let $Z\lhd G$ be a central normal subgroup. Assume that $Z$ is $1$-connected (so that $\pi: G \to Q := G/Z$ admits a continuous section, see \cite[Thm.\ 2]{Shtern}), that $\Lambda$ is $Z$-aligned with uniformly large fibers and that $\Omega_\Lambda$ admits a $G$-invariant ergodic probability measure $\nu_\Lambda$. By Theorem \ref{IntroHorizontalFactor} we then have a continuous $G$-factor map $\pi_*: \Omega_\Lambda \to \Omega_\Delta$, and we denote by $\nu_\Delta$ the push-forward of $\nu_\Lambda$. We then find a subrepresentation
\[
 L^2(\Omega_\Lambda, \nu_\Lambda)_{\rm hor} \subset L^2(\Omega_\Lambda, \nu_\Lambda)
\]
on which $Z$ acts trivially and which is isomorphic to $L^2(\Omega_\Delta, \nu_\Delta)$ as a representation of $Q = G/Z$. We refer to this subrepresentation as the \emph{horizontal part} of $L^2(\Omega_\Lambda, \nu_\Lambda)$.

To further analyze $L^2(\Omega_\Lambda, \nu_\Lambda)$ one would like to understand also the \emph{vertical part}, i.e.\ subrepresentation on which $Z$ acts non-trivially. More specifically, we would like to determine the \emph{central pure-point spectrum} of $L^2(\Omega_\Lambda, \nu_\Lambda)$ as given by  
\[{\rm spec}^Z_{\rm pp}(L^2(\Omega_\Lambda, \nu_\Lambda)) := \{\xi \in \widehat{Z} \mid L^2(\Omega_\Lambda, \nu_\Lambda)_\xi \neq \{0\}\},\]
where $\widehat{Z}$ denotes the Pontryagin dual of $Z$ and given $\xi \in \widehat{Z}$ we denote
\[
L^2(\Omega_\Lambda, \nu_\Lambda)_\xi := \{f \in L^2(\Omega_\Lambda, \nu_\Lambda)\mid \forall z \in Z: f(z^{-1}g) = \xi(z)f(g)\}.
\]
Given $\xi \in {\rm spec}^Z_{\rm pp}(L^2(\Omega_\Lambda, \nu_\Lambda))$ we would also like to construct explicit eigenfunctions in $L^2(\Omega_\Lambda, \nu_\Lambda)_\xi$. 

\subsection{A relative version of the Bombieri-Taylor conjecture}
Our construction of explicit eigenfunction in $L^2(\Omega_\Lambda, \nu_\Lambda)_\xi$ relies on the existence of twisted fiber densities of $\Lambda$. This existence is guaranteed by the following result, which can be seen as a relative version of Hof's solution to the Bombieri-Taylor conjecture \cite{Hof95}. Here we assume for simplicity that $Z \cong \R^n$; we denote by $B_T$ Euclidean balls of radius $T$ around $0$ and by $m_Z$ Lebesgue measure on $Z$.
\begin{theorem}[Relative Bombieri-Taylor conjecture]\label{IntroBT} Let $\Lambda \subset G$ be a $Z$-aligned uniform approximate lattice with uniformly large fibers and assume that $\Omega_\Lambda$ admits a $G$-invariant measure $\nu_\Lambda$. Then for every $\xi \in \widehat{Z}$ there exists a $G$-invariant $\nu_\Lambda$-conull subset $\Omega_\xi \subset \Omega_\Lambda$ such that for all $\Lambda' \in \Omega_\xi$ and $\delta' \in \pi(\Lambda')$ the following limit exists:
\[
D_\xi(\Lambda', \delta') := \lim_{T \to \infty} \frac{1}{m_Z(B_T)}\sum_{z \in \Lambda'_{\delta'} \cap B_T} \overline{\xi(z)}
\]
\end{theorem}
The function $D_\xi$ is called the \emph{twisted fiber density function} associated with $\xi$. 

\subsection{Twisted periodization operators and central Bragg peaks}
Using the twisted fiber density functions one can define \emph{twisted periodization operators}
\[
\mathcal P_\xi: C_c(Q) \to L^2(\Omega_\Lambda, \nu_\Lambda), \quad (\mathcal P_\xi \phi)(\Lambda') := \sum_{\delta' \in \pi(\Lambda')} \phi(\delta') D_\xi(\Lambda', \delta').
\]
While this operator may look unfamiliar at first sight, it is actually related to classical objects, at least in the case of lattices in the Heisenberg group. Indeed, write the Heisenberg group as a central extension $H_3(\R) = \R \oplus_\omega \R^2$, with underlying cocycle $\omega: \R^2 \times \R^2 \to \R$ given by a symplectic form. Let $\Delta \subset \R^2$ and $\Xi \subset \R$ be lattices with $\omega(\Delta, \Delta) \subset \Xi$ so that $\Lambda := \Xi \oplus_\omega\Delta$ is a lattice in the Heisenberg group. We will see in \eqref{PeriodizationEasy} below that
\[
\mathcal P_\xi f((t,v)\Lambda) =   D_\xi(\Lambda,e) \cdot \sum_{\delta \in \Delta} \phi(v + \delta) \overline{\xi(t+\omega(v, \delta))},
\]
and - up to the normalization constant $ D_\xi(\Lambda,e)$ - this operator appears frequently in time frequency analysis, see \cite{Groechenig}.

It turns out that $\mathcal P_\xi$ takes values in $L^2(\Omega_\Lambda, \nu_\Lambda)_\xi$, but it will be $0$ for all but countably many $\xi \in \widehat{Z}$. To give a criterion for non-triviality of the operator $\mathcal P_\xi$ we choose $r>0$ such that $\Delta$ is $r$-uniformly discrete, and given $q \in Q$ denote by $C_c(B_{r/2}(q))$ the space of continuous functions on $Q$ supported in an $r/2$-neighbourhood around $q$.
\begin{proposition}[Central diffraction coefficients] For every $q \in Q$ the restriction $\mathcal P_\xi|_{C_c(B_{r/2}(q))}$ extends to a bounded linear map
\[
\mathcal P^{(q)}_\xi: L^2(B_{r/2}(q)) \to L^2(\Omega_\Lambda, \nu_\Lambda)_\xi
\]
and there exists a constant $c_\xi>0$ such that for all $q \in Q$,
\[
\|\mathcal P^{(q)}_\xi f\|^2 = c_\xi \cdot \|f\|^2 \quad \text{for all }f \in L^2(B_{r/2}(q)).
\]
In particular, $\|\mathcal P^{(q)}_\xi\|_{\rm op} = c_\xi^{1/2}$ and $\xi \in {\rm spec}^Z_{\rm pp}(L^2(\Omega_\Lambda, \nu_\Lambda))$ if $c_\xi \neq 0$.
\end{proposition}
One can show that $c_1 \neq 0$ and that $0\leq c_\xi \leq c_1$ for all $\xi \in \widehat{Z}$. We will explain in Proposition \ref{IntroDiff} that the numbers $c_\xi$ can be interpreted as (central) diffraction coefficients. Motivated by this interpretation we say that $\xi$ is a \emph{central $(1-\epsilon)$-Bragg peak} with respect to $\nu_\Lambda$ provided $c_\xi \geq (1-\epsilon) c_1$. Then our main result is as follows:
\begin{theorem}[Relative denseness of central Bragg peaks]\label{IntroBragg} Let $\Lambda \subset G$ be a $Z$-aligned uniform approximate lattice with uniformly large fibers and assume that $\Omega_\Lambda$ admits a $G$-invariant measure $\nu_\Lambda$. Then for every $\epsilon \in (0,1)$ the set of central $(1-\epsilon)$-Bragg peaks with respect to $\nu_\Lambda$ is relatively dense in $\widehat{Z}$.
\end{theorem}
In the ``absolute case'' where $Q = \{e\}$ is trivial and thus $G$ is abelian, this recovers a result of Strungaru \cite{St05}. Theorem \ref{IntroBragg} implies in particular that $ {\rm spec}^Z_{\rm pp}(L^2(\Omega_\Lambda, \nu_\Lambda))$ is relatively dense in $\widehat{Z}$, but it is in fact a much stronger statement. For example, in the case of aperiodic model sets the central pure point spectrum is actually dense in $\widehat{Z}$, whereas the set of central $(1-\epsilon)$-Bragg peaks is a Delone set. 

\subsection{Relation to central diffraction}
We now turn to the promised interpretation of the numbers $c_\xi$ as diffraction cofficients. In the absolute case, the coefficients $c_\xi$ can be interpreted as coefficients of the pure-point part of the diffraction measure, and hence the Bragg peaks can be ``seen'' in the diffraction picture of quasi-crystals. This result generalizes to the present setting in the following form.

Recall first that if $G$ is abelian and $\nu_\Lambda$ is a $G$-invariant measure on $\Omega_{\Lambda}$, then one can define the corresponding \emph{auto-correlation measure} $\eta = \eta(\nu_\Lambda)$, which is a Radon measure on $G$ and can be obtained by sampling along $\Lambda'-\Lambda'$ for a $\nu_\Lambda$-generic $\Lambda' \in \Omega_{\Lambda}$. The corresponding \emph{diffraction measure} is then defined as the Fourier transform $\widehat{\eta}$, see. e.g. \cite{BHP1, BG13}.

In the setting of Theorem \ref{IntroBragg} the auto-correlation measure $\eta = \eta(\nu_\Lambda)$ can be decomposed as a sum
\[
\eta = \sum_{\delta \in \Delta^2} {\eta}_\delta,
\]
where $\Delta = \pi(\Lambda)$ and $\eta_\delta$ is a certain fiber measure supported on $\pi^{-1}(q)$. The fiber measure $\eta_e$ over the identity is a positive-definite Radon measure on $Z$, and we refer to it as the \emph{central auto-correlation} of $\Lambda$ with respect to $\nu_\Lambda$. We then define the associated \emph{central diffraction measure} as its Fourier transform $\widehat{\eta_e}$. In this context, the numbers $c_\xi$ admits the following interpretation:
\begin{proposition}[Interpretation of central diffraction coefficients]\label{IntroDiff} In the setting of Theorem \ref{IntroBragg} the pure point part $(\widehat{\eta_e})_{\rm pp}$ of the central diffraction measure is given by
\[
(\widehat{\eta_e})_{\rm pp} = \sum_{\xi \in {\rm spec}^Z_{\rm pp}(L^2(\Omega_\Lambda, \nu_\Lambda))} c_\xi \; \delta_\xi.
\] 
\end{proposition}

\subsection{Nilpotent Lie groups and universally aligned towers}
In nilpotent Lie groups, Theorem \ref{IntroBragg} can be applied iteratively. Indeed, let $G = G_1$ be a nilpotent $1$-connected Lie group, and define $G_2 := G_1/Z(G_1)$, $G_3 := G_2/Z(G_2)$ etc. until you reach an abelian group $G_n$. Let moreover $\Lambda = \Lambda_1$ be a uniform approximate lattice in $G$ and denote by $\Lambda_2, \dots, \Lambda_n$ its projections to $G_2, \dots, G_n$. If $j \in \{1, \dots, n\}$, then $\Lambda_j$ is a uniform approximate lattice in $G_j$ and a relatively dense set of fibers of $\Lambda_j$ consists of Meyerian sets in $Z(G_j)$. Under mild additional assumptions on $\Lambda$, one then has a sequence of continuous $G$-factors
\[
\Omega_{\Lambda_1} \to \Omega_{\Lambda_2} \to \dots \to \Omega_{\Lambda_n}.
\]
If we fix a $G$-invariant measure $\nu_1$ on $\Omega_{\Lambda_1}$ and denote by $\nu_j$ its push-forward to $\Omega_{\Lambda_j}$, then we obtain a corresponding sequence of measurable $G$-factors
\[
(\Omega_{\Lambda_1}, \nu_1) \to (\Omega_{\Lambda_2}, \nu_2) \to \dots \to (\Omega_{\Lambda_n}, \nu_n).
\]
The spectral theory of $(\Omega_{\Lambda_n}, \nu_n)$ can be studied using the well-developed spectral theory of Meyer sets in abelian lcsc groups. Theorem \ref{IntroBragg} allows us to understand the ``relative spectral theory'' of the extensions $(\Omega_{\Lambda_j}, \nu_j) \to (\Omega_{\Lambda_{j+1}}, \nu_{j+1})$.
\begin{example} Consider the $(2n+1)$-dimensional Heisenberg group $G = H_{2n+1}(\R)$. This is a $2$-step nilpotent Lie group with one-dimensional center $Z$, and the quotient $Q = G/Z$ is isomorphic to $\R^{2n}$. The unitary dual $\widehat{G}$ of $G$ is the disjoint union of two parts: The equivalence classes of irreducible unitary $G$-representations of $G$ with trivial central character are parametrized by the unitary dual of $Q$, which is homeomorphic to $\R^{2n}$. Morever, for every $\xi \in \widehat{Z} \setminus\{1\} \cong (\R\setminus\{0\})$ there is a unique equivalence class of unitary $G$-representations with central character $\xi$; representatives are given by the so-called \emph{Schr\"odinger representations}. If we denote this ``Schr\"odinger part'' of the unitary dual of $G$ by $\widehat{G}_{\rm Sch}$, then we have a decomposition
\[
\widehat{G} \quad = \quad  \widehat{G}_{\rm Sch} \;\sqcup\; \widehat{Q} \quad \cong \quad (\R\setminus\{0\}) \;\sqcup \;\R^{2n}.
\]
The induced topologies on $\R\setminus\{0\}$ and $\R^{2n}$ are the usual ones, but $\widehat{G}$ is not Hausdorff, and in the topology of $\widehat{G}$ the subset $\R^{2n}$ is contained in the closure of any neighbourhood of $0$ in $\R\setminus\{0\} \subset \R$. We will say that a subset of $\widehat{G}$ is \emph{relatively dense} if it  is relatively dense in both $\R\setminus\{0\}$ and $\R^{2n}$ with respect to the usual Euclidean metrics. 

Now let $\Lambda$ be a uniform approximate lattice in $G$, let $\Delta := \pi(\Lambda)$ and assume that $\Delta$ has uniformly large fibers. Concrete examples for $n=3$ are given by the sets
\[\Lambda = \left\{
\left(\begin{matrix}
1 &  a_1 + b_1 \sqrt 2 & a_3+b_3\sqrt 2\\
0&1& a_2 + b_2 \sqrt 2)\\
0&0&1
\end{matrix}
\right)  \mid a_j, b_j \in \Z, \, \begin{array}{lcl} |a_1 - b_1 \sqrt 2| &<& R_1,\\  |a_2 - b_2 \sqrt 2| &<& R_2,\\  |a_3 - b_3 \sqrt 2| &<& R_3 \end{array}\right\}.
\]
for arbitrary parameters $R_1, R_2, R_3>0$ and certain relatively dense subsets thereof. If we fix a $G$-invariant probability measure $\nu_\Lambda$ on $\Omega_\Lambda$, then Theorem \ref{IntroBragg} says that for $\epsilon \in (0,1)$ the central $(1-\epsilon)$-Bragg peaks of $\Lambda$ with respect to $\nu_\Lambda$ are relatively dense in $\widehat{G}_{\rm Sch}$, and the classical abelian theory implies that the $(1-\epsilon)$-Bragg peaks of $\Delta$ are relatively dense in $\widehat{Q}$. Their union is thus relatively dense in $\widehat{G}$, and in particular the pure point spectrum of $L^2(\Omega_\Lambda, \nu_\Lambda)$ is relatively dense in $\widehat{G}$.
\end{example}
The example can in principle be extended to higher step nilpotent Lie groups, but some new difficulties arise. Most notably, we do not know how to isolate isotypic components of representations with the same central character, a phenomenon absent in the Heisenberg groups and related to the question of square-integrability of irreducible unitary representations. Also, the bookkeeping becomes more tedious in higher step. We leave it to future work to resolve these issues.

We would like to mention that the theory of iterated aligned central extensions also has applications to the theory of uniform approximate lattice (and, more generally, Meyerian sets) beyond spectral theory. For example, we explain in Appendix \ref{AppendixMeyer} how this theory can be used to generalize theorems of Meyer and Dani--Navada concerning number-theoretic properties of Meyer sets to the setting of nilpotent Lie groups.

\subsection{Organization of this article}

This article is organized as follows: In Section \ref{SecPrelim} we collect various preliminaries concerning uniform approximate lattices, Meyerian sets and extensions of lcsc groups. Further background concerning nilpotent Lie groups is contained in Appendix \ref{AppNilpotent}. In Section \ref{SecAligned} we study Meyerian sets aligned to a given extension and establish Theorem \ref{IntroAligned}. In Section \ref{SecExamples} we discuss various examples and establish Theorem \ref{IntroTower}. Section \ref{SecHof}, Section \ref{SecHorizontal} and Section \ref{SecHullSpec} are the core of this article and establish Theorem \ref{IntroBT}, Theorem \ref{IntroHorizontalFactor} and Theorem \ref{IntroBragg} respectively. Section \ref{SecDiffraction} explains the relation between Bragg peaks and atoms of the central diffraction as summarized in Proposition \ref{IntroDiff}.  Appendix \ref{AppCrazyTower} contains our original proof of Theorem \ref{IntroTower}, which in the body of the text is deduced from Machado's embedding theorem. This alternative proof provides another example of a universally aligned tower. Appendix \ref{AppendixMeyer} gives the application of aligned towers to number theoretic properties of Meyerian sets alluded to above.\\

{\bf Acknowledgements:} We thank Emmanuel Breuillard for pointing out the work of his student Simon Machado, and we thank the latter for making a preliminary version of his work \cite{Machado} available to us.

%\tableofcontents

\section{Preliminaries on Meyerian sets and extensions of lcsc groups}\label{SecPrelim}

\subsection{Delone sets and their hulls}
We start by recalling some basic facts concerning Delone sets in lcsc groups; a convenient reference for our purposes is \cite{BH1}. Given a metric space $(X,d)$ and a metric space and constants $R > r > 0$, a non-empty subset $\Lambda \subset X$ is called
\begin{enumerate}[(i)]
\item $r$-\emph{uniformly discrete} if $d(x, y) \geq r$ for all $x,y \in \Lambda$ with $x \neq y$;
\item $R$-\emph{relatively dense} if for all $x \in X$ there exists $\Lambda \in \Lambda$ such that $d(x, \Lambda) \leq R$;
\item a $(r, R)$-\emph{Delone set} if it is both $r$-uniformly discrete and $R$-relatively dense.
\end{enumerate}
We then say that $(r,R)$ are \emph{Delone parameters} for $\Lambda$. If we do not wish to specify the Delone parameters we simply say that $\Lambda$ is a Delone set.

\begin{remark}[Delone sets in groups, cf. \cite{BH1}]
A metric on a topological group $G$ is called \emph{left-admissible} if it is left-invariant, proper and induces the given topology; such a metric exists if and only if $G$ is lcsc. If $d$ is any left-admissible metric on $G$, then a subset $\Lambda$ is uniformly discrete in $(G,d)$ if $e$ is not an accumulation point of $\Lambda^{-1}\Lambda$ and relatively dense in $(G,d)$ if and only if there exists a compact subset $K \subset G$ such that $G = \Lambda K$. In particular, the notion of a Delone set in a lcsc group is independent of the choice of left-admissible metric used to define it, whereas the Delone parameters may depend on the choice of metric. A Delone set $\Lambda \subset G$ is said to have \emph{finite local complexity (FLC)} if $\Lambda^{-1}\Lambda$ is locally finite, i.e.\ closed and discrete. 
\end{remark}

Two Delone sets $\Lambda, \Lambda' \subset G$ are called \emph{commensurable} if there exist finite subsets $F_1, F_2 \subset G$ such that $\Lambda \subset \Lambda' F_1$ and $\Lambda' \subset \Lambda F_2$. If $\Lambda \subset G$ is a Delone set, then a subset $\Lambda' \subset \Lambda$ is Delone if and only if it is relatively dense in $\Lambda$, i.e.\ if there exists a finite set $F \subset G$ such that $\Lambda \subset \Lambda' F$. In particular, if two Delone sets $\Lambda, \Lambda'$ are contained in a common Delone set, then they are commensurable.\\
  
From now on let $G$ be a lcsc group and let $d$ be a left-admissible metric on $G$. 

\begin{remark}[Chabauty--Fell topology]
Given $R>r>0$ we denote by $\mathcal C(G)$, $\mathcal U_r(G)$, $\mathcal D_R(G) $ and ${\rm Del}_{r,R}(G)$ the collections of closed, $r$-uniformly discrete, $R$-relatively dense and $(r,R)$-Delone subsets of $(G,d)$. The \emph{Chabauty--Fell topology} on $\mathcal C(G)$ is the topology generated by the basic open sets
\[
U_K = \{A \in \mathcal C(G) \mid A \cap K = \emptyset\} \quad \text{and} \quad U^{V} = \{A \in \mathcal C(G) \mid A \cap V \neq \emptyset\},
\]
where $K$ and $V$ runs through all compact, respectively open subsets of $G$. With respect to this topology the space $\mathcal C(X)$ is compact (in particular, Hausdorff) and second countable, hence metrizable. Moreover, the action of $G$ on $\mathcal C(G)$ is jointly continuous and for all $R>r>0$ the subspaces $\mathcal U_r(G)$, $\mathcal D_R(G) $ and ${\rm Del}_{r,R}(G)$ are closed in $\mathcal C(G)$, hence compact.
\end{remark}
Since the Chabauty--Fell topology is metrizable, it is characterized by convergence of sequences (rather than nets), which admits the following explicit description \cite{BenedettiPetronio}:
\begin{lemma}\label{CFConvergence} Let $P_n, P \in \mathcal C(G)$. Then $P_n \to P$ in the Chabauty--Fell topology if and only if the following two conditions hold:
\begin{enumerate}[(i)]
\item If $(n_k)$ is an unbounded sequence of natural numbers and $p_{n_k} \in P_{n_k}$ converge to $p \in G$, then $p \in P$.
\item For every $p \in P$ there exist elements $p_n \in P_n$ such that $p_n \to p$.\qed
\end{enumerate}
\end{lemma}

We can now make precise the notion of a hull as mentioned in the introduction:

\begin{definition}\label{DefHull} Given $\Lambda \in \mathcal C(G)$, the \emph{hull} of $\Lambda$ is defined as the orbit closure
\[
\Omega_{\Lambda} := \overline{\{g.\Lambda \mid g \in G\}} \subset \mathcal C(G).
\]
\end{definition}
From the corresponding properties of $\mathcal C(G)$ one deduces that $\Omega_{\Lambda}$ is compact and that $G$ acts jointly continuously on $\Omega_{\Lambda}$, i.e. $G \curvearrowright\Omega_{\Lambda}$ is a topological dynamical system. Moreover, if $\Lambda$ is $r$-uniformly discrete or $R$-relatively dense in $G$, then every $\Lambda' \in \Omega_\Lambda$ has the same property. In particular, the hull of a Delone set consists of Delone sets, and if $\Lambda$ is relatively dense, then $\emptyset \not \in \Omega_{\Lambda}$; the converse is also true.
\begin{lemma}\label{RelDenseEmptyset} Let $\Lambda\in \mathcal C(G)$. Then $\Lambda$ is relatively dense if and only if $\emptyset \not \in \Omega_{\Lambda}$.\qed
\end{lemma}
Throughout this article we will use the following consequence of Lemma \ref{CFConvergence}:
\begin{lemma}\label{DifferenceSetInclusion} Let $G$ be a lcsc group and $\Lambda \in \mathcal C(G)$. Then for all $\Lambda' \in \Omega_{\Lambda}$ there is an inclusion
$(\Lambda')^{-1}\Lambda' \subset \overline{\Lambda^{-1}\Lambda}$. In particular, if $\Lambda$ has finite local complexity, then
\[(\Lambda')^{-1}\Lambda' \subset {\Lambda^{-1}\Lambda}.\qed\]
\end{lemma}

\subsection{Meyerian sets and their characterizations}
Throughout this subsection, $G$ denotes a lcsc group. Given subsets $X, Y\subset G$ we denote by
\[
XY := X \cdot Y := \{xy \mid x\in X, y \in Y\}
\]
their Minkowski product, by $X^{-1}$ its set of inverses of elements of $X$, and given $k \in \mathbb N$ we denote by \[
X^k = \{x_1 \cdots x_k \mid x_1, \dots, x_k \in X\}.
\]
its \emph{$k$-fold Minkowski product}. For distinction we will denote the $k$-fold Cartesian product of $X$ by $X^{\times k}$. 

\begin{definition}\label{Meyerian} A Delone subset $\Lambda$ of a lcsc group $G$ is called \emph{Meyerian} if $(\Lambda^{-1} \Lambda)^k$ is uniformly discrete for all $k \in \mathbb N$. A Meyerian subset is called a \emph{uniform approximate lattice} if it is moreover symmetric (i.e.\ $\Lambda = \Lambda^{-1}$) and contains the identity. 
\end{definition}
The following characterizations of uniform approximate lattice are established in \cite{BH1}.
\begin{lemma}\label{mProperties} Let $G$ be a lcsc group and let $\Lambda \subset G$ be a relatively dense subset which is symmetric and contains the identity. Then $\Lambda$ is a uniform approximate lattice if and only if one of the following mutually equivalent conditions holds:
\begin{enumerate}[(m1)]
\item $\Lambda^3$ is locally finite, i.e.\ closed and discrete.
\item $\Lambda^k$ is uniformly discrete for all $k \geq 1$.
\item $\Lambda$ is uniformly discrete and there exists a finite subset $F \subset G$ such that $\Lambda^2 \subset \Lambda F$.\qed
\end{enumerate}
\end{lemma}
Characterization (m3) of Lemma \ref{mProperties} relates uniform approximate lattices to approximate subgroups in the sense of Tao \cite{Tao}. Recall that a subset $\Lambda \subset G$ is called an \emph{approximate subgroup} if it is symmetric and contains the identity and there exists a finite subset $F \subset G$ such that $\Lambda^2 \subset \Lambda F$. In this terminology, (m3) says that a uniform approximate lattices is the same as a Delone approximate subgroup (just as a uniform lattice is the same as a Delone subgroup). 

The approximate subgroup property is invariant under many basic constructions (see e.g.\ \cite{BGT, BH1}). For example images and pre-images of approximate subgroups under group homomorphisms are again approximate subgroups. While intersections of approximate subgroups with subgroups need not be approximate subgroups again, we have the following replacement
(see e.\ g.\ \cite{BGT}):
\begin{lemma}\label{BGT} Let $\Gamma$ be a group, let $\Xi < \Gamma$ be a subgroup and let $\Lambda \subset \Gamma$ be an approximate subgroup. Then $\Lambda^2 \cap \Xi$ is an approximate subgroup of $\Xi$.
\end{lemma}
\begin{proof} Let $F \subset \Gamma$ be finite such that  $\Lambda^2 \subset F\Lambda$ and set
\[
T := \{t \in F^3 \mid t\Lambda \cap \Xi \neq \emptyset\}.
\]
For every $t \in T$ pick an element $x_t \in t\Lambda \cap \Xi$. Given $t \in T$ and $x \in t\Lambda \cap \Xi$ we have $x_t^{-1}x \in \Lambda^2 \cap \Xi$, hence $x \in x_t(\Lambda^2 \cap \Xi)$. This shows that
\[
t\Lambda \cap \Xi \subset x_t(\Lambda^2 \cap \Xi),
\]
and since $\Lambda^4 \subset F^3\Lambda$ we deduce that
\[
(\Lambda^2 \cap \Xi)^2 \subset \Lambda^4 \cap H \subset F^3\Lambda \cap \Xi \subset \bigcup_{t \in T} t\Lambda \cap \Xi \subset  \bigcup_{t \in T} x_t (\Lambda^2 \cap \Xi) = \left(  \bigcup_{t \in T} x_t \right)(\Lambda^2 \cap \Xi),
\]
which shows that $\Delta$ is an approximate subgroup of $\Xi$, since $T$ is finite.
\end{proof}
We observe that if $\Lambda$ is a relatively dense subset of a lcsc group $G$, then $\Lambda^{-1}\Lambda$ is still relatively dense (since it contains a translate of $\Lambda$). It is moreover symmetric and contains the identity. In particular, $\Lambda \subset G$ is a Meyerian subset if and only $\Lambda^{-1}\Lambda$ is a uniform approximate lattice. This observation together with Proposition \ref{mProperties} implies:
\begin{corollary}\label{MProperties} Let $G$ be a lcsc group and let $\Lambda \subset G$ be a Delone set. Then $\Lambda$ is Meyerian if and only if one of the following mutually equivalent conditions holds:
\begin{enumerate}[(M1)]
\item $(\Lambda^{-1}\Lambda)^3$ is locally finite, i.e.\ closed and discrete.
\item $(\Lambda^{-1}\Lambda)^k$ is uniformly discrete for all $k \geq 1$.
\item $\Lambda^{-1}\Lambda$ is uniformly discrete and there exists a finite subset $F \subset G$ such that $(\Lambda^{-1}\Lambda)^2 \subset \Lambda^{-1}\Lambda F$.\qed
\end{enumerate}
\end{corollary}

Meyerian subsets of compactly-generated abelian lcsc groups are called \emph{Meyer sets}, and these admit a long list of additional characterizations (see e.g.\ \cite{Moody}), among which the following are relevant to the present article.
\begin{lemma}\label{MeyerAbelian} Let $A$ be a compactly-generated abelian lcsc group and let $\Lambda \subset A$ be a relatively dense subset. Then $\Lambda$ is a Meyer set if and only if one of the following mutually equivalent conditions holds:
\begin{enumerate}[(Me1)]
\item $\Lambda - \Lambda$ is uniformly discrete.
\item $\Lambda \pm \Lambda \pm \dots \pm \Lambda$ is uniformly discrete for any choice of signs.
\item $\Lambda$ is uniformly discrete and there exists a finite subset $F \subset A$ such that $\Lambda-\Lambda \subset \Lambda + F$.
\item $\Lambda$ is harmonious, i.e.\ the subset
\begin{equation}\label{EpsilonDual}
\widehat{\Lambda}^\epsilon = \{\chi \in \widehat{A} \mid \|(\chi-1)|_{\Lambda}\| < \epsilon\} \subset \widehat{A}
\end{equation}
is relatively dense for all $\epsilon > 0$.\qed
\end{enumerate}
\end{lemma}
It is established in \cite{BH1} that if $G$ is a lcsc group which is either compactly-generated or amenable and contains a Meyerian subset (hence a uniform approximate lattice), then it is necessarily unimodular. Since this comprises all the examples we are interested in the current article, we make the following convention.
\begin{convention}\label{ConventionUnimodular} All lcsc groups in this article are assumed to be unimodular.
\end{convention}

\subsection{Extensions of lcsc groups}

We now introduce some notation concerning group extensions which will be used throughout this article. We say that $(G, N, Q, \pi, s)$ is an \emph{extension of lcsc groups} if
$G$ is a lcsc group, $N\lhd G$ is a closed normal subgroup and $Q = G/N$, $\pi: G \to Q$ denotes the canonical projection and $s: Q \to G$ is a Borel section, i.e.\ a Borel mesurable right inverse to $\pi$. Recall from Convention \ref{ConventionUnimodular} that we always implicitly assume $G$, $N$ and $Q$ to be unimodular.

Note that the canonical projection $\pi: G \to Q$ always admits a Borel section by \cite[Appendix B]{Zimmer}. It will be convenient for us to include a fixed choice of such a section into our data. We say that the extension $(G, N, Q, \pi, s)$ is \emph{topologically split} if $s$ is continuous. The existence of such a continuous section is equivalent to triviality of the principal $N$-bundle $G \to Q$; in particular a continuous sections exists whenever $Q$ is contractible.

Given an extension of lcsc groups $(G, N, Q, \pi, s)$ we define
\begin{eqnarray*}
\alpha_s: Q \to {\rm Aut}(N), && \alpha_s(q)(n) := s(q)ns(q)^{-1},\\
\beta_s: Q \times Q \to N, && \beta_s(q_1, q_2) := s(q_1)s(q_2)s(q_1q_2)^{-1}.
\end{eqnarray*}
If $c: N \to {\rm Inn}(N)$ denotes conjugation, then $\alpha := \alpha_s$ and $\beta := \beta_s$ satisfy the \emph{Schreier factor system relations}
\begin{eqnarray}\label{SchreierFactorSystem}
\alpha(q_1) \circ \alpha(q_2) = c(\beta(q_1, q_2))\circ \alpha(q_1q_2),\\ 
\beta(q_1, q_2)\beta(q_1q_2, q_3) = \alpha(q_1)(\beta(q_2, q_3)) \beta(q_1, q_2q_3). \label{SchreierFactorSystem2}
\end{eqnarray}
Under the mutually inverse Borel isomorphisms
\[\iota: N \times Q \to G, \quad (n, q) \mapsto n s(q) \qand \iota^{-1}: G \to N \times Q, \quad g \mapsto (gs(\pi(g))^{-1},\pi(g))\] 
multiplication on $G$ get intertwined with a multiplication on $N \times Q$ which is given by
\[
(n_1, q_1)(n_2, q_2) = (n_1 \alpha_s(q_1)(n_2) \beta_s(q_1, q_1), q_1q_2) \quad (n_1, n_2 \in N, q_1, q_2 \in Q).
\]
Conversely, let $Q$, $N$ be lcsc groups. If $\alpha: Q \to {\rm Aut}(N)$ and $\beta: Q \times Q \to N$ satisfy the Schreier factor system relation \eqref{SchreierFactorSystem} and \eqref{SchreierFactorSystem} and both $\beta$ and the map $ Q \times N \to N$, $(q,n) \mapsto \alpha(q)(n)$ are Borel measurable, then $Q \times N$ becomes a group under the multiplication
\begin{equation}\label{GeneralProductExtension}
(n_1, q_1)(n_2, q_2) = (n_1 \alpha(q_1)(n_2) \beta(q_1, q_1), q_1q_2) \quad (n_1, n_2 \in N, q_1, q_2 \in Q),
\end{equation}
and there is a unique lcsc group topology on $Q \times N$ whose associated Borel structure agrees with the product Borel structure. We denote the resulting lcsc group by $N \times_{\alpha, \beta} Q$. Thus if $(G, N, Q, \pi, s)$ is an extension of lcsc groups, then $G$ is Borel isomorphic to $N \times_{\alpha_s, \beta_s} Q$, and if the extension is topologically split, then this isomorphism is moreover a homeomorphism. 

In the sequel, $(G, N, Q, \pi, s)$ will always denote an extension of lcsc groups. We will always tacitly identify $G$ with $N \oplus_{\alpha_s, \beta_s} Q$. If $\Lambda \subset G$ is a subset, then for every $\delta \in \Delta := \pi(\Lambda)$ we define the \emph{fiber} of $\Lambda$ over $\delta$ by
\begin{equation}\label{Fiber}
\Lambda_\delta :=N \cap \Lambda s(\delta)^{-1} =  \{x \in N \mid x s(\delta) \in \Lambda\} 
\end{equation}
Then, under the identification $G \cong N \oplus_{\alpha_s, \beta_s} Q$, we have
\begin{equation}\label{FiberUnion}
\Lambda = \bigcup_{\delta \in \Delta} \Lambda_\delta \times \{\delta\},
\end{equation}
Moreover, for every $\delta \in \Delta^2$ we have
\begin{equation}
\Lambda^2_\delta = \bigcup_{\{(\delta_1, \delta_2) \in \Delta \times \Delta\mid \delta = \delta_1\delta_2\}} \Lambda_{\delta_1}\alpha_s(\delta_1)(\Lambda_{\delta_2})\beta_s(\delta_1, \delta_2).
\end{equation}
We say that $N \times_{\alpha, \beta} Q$ is \emph{untwisted} if $\alpha$ is the trivial homomorphism; in this case we also write $N \times_\beta Q := N \times_{\alpha, \beta} Q$. Similarly, the extension $(G, N, Q, \pi, s)$ is called \emph{untwisted} if $\alpha_s$ is trivial; it is called \emph{abelian}, if $N$ is an abelian subgroup of $G$. An abelian extension is untwisted iff $N$ is contained in the center of $G$; it is then called a \emph{central extension}. If the extension $(G, N, Q, \pi, s)$ is abelian (respectively central), then we will usually denote the normal subgroup $N$ by $A$ (respectively $Z$).

If $(G, Z, Q, \pi, s)$ is a central extension, then we usually write $Z$ additively; then multiplication on $Z \times_\beta Q$ is given by
\begin{equation}\label{CentralProduct}
(z_1, q_1)(z_2, q_2) = (z_1+z_2 + \beta(q_1, q_2), q_1q_2) \quad(z_1, z_2 \in Z, q_1, q_2 \in Q).\end{equation}
In view of  unimodularity of $Z$, $Q$ and $G$ this implies that if $m_Z$ and $m_Q$ denote Haar measures on $Z$ and $Q$ respectively, then a Haar measure on $G = Z \times_\beta Q$ is given by 
\begin{equation}
m_G = m_Z \otimes m_Q.
\end{equation}
In the sequel, when dealing with central extensions, we will always assume that Haar measures $m_Z$, $m_Q$ and $m_G$ have been chosen in this way.
From \eqref{CentralProduct} we also deduce that
%Associativity of this multiplication then yields the cocycle identity
%\begin{equation}\label{CentralCocycleIdentity}
%\beta(q_1q_2, q_3) = \beta(q_1, q_2q_3) + \beta(q_2, q_3) \quad(q_1, q_2, q_3 \in Q).
%\end{equation}
%We deduce in particular that
\begin{equation}\label{CentralInverse}
(z, q)^{-1} =(-z - \beta(q, q^{-1}), q^{-1}) \quad (z \in Z, q \in Q).
\end{equation}
If the section $s: Q \times Q \to Z$ is symmetric and $s(e) = e$ then $\beta_s(e,q) = \beta_s(q,e) = 0$. Moreover, $(z, q)^{-1} = s(q)^{-1}z^{-1} = z^{-1}s(q^{-1})$ and hence $(z, q)^{-1} = (-z, q)$. Comparing this to \eqref{CentralInverse} we deduce that $\beta_s(q, q^{-1}) = 0$. Thus if $s$ is a symmetric section with $s(e) = e$, then
\begin{equation}\label{SymmetricSection}
\beta_s(e,q) = \beta_s(q,e) = \beta_s(q, q^{-1}) = 0 \quad (q \in Q).
\end{equation}

\section{Extensions aligned to a Meyerian set}\label{SecAligned}

\subsection{Characterizations of aligned extensions}
Consider a uniform lattice $\Gamma$ in $\R^n$ and let $\R^n = V \oplus W$ be a decomposition of $\R^n$. Denote by $\pi: \R^n \to V$ the projection along $W$. If the decomposition is in general position with respect to $\Gamma$, then $\pi(\Gamma)$ will be dense in $V$ and the fibers will be singletons. However, for special choices of $V$ and $W$ the projection will again be a uniform lattice, and all fibers will be translates of a uniform lattice. We then say that $\Gamma$ and the extensions $W \to \R^n \to V$ are \emph{aligned}. More generally we define the following.
\begin{definition} Let $(G, N, Q, \pi, s)$ be an extension of lcsc groups and let $\Lambda$ be a Delone set in $G$. We say that $\Lambda$ and $\pi$ are \emph{aligned} if $\Delta :=\pi(\Lambda) \subset Q$ is a Delone set in $Q$. If $\Lambda$ and $\pi$ are aligned we also say that $\Lambda$ is \emph{$\pi$-aligned} or that $\pi$ is \emph{aligned with $\Lambda$} or that $N$ is \emph{aligned with $\Lambda$}. 

\end{definition}
\begin{definition} Let $G$ be a lcsc group. A normal subgrop $N$ is \emph{universally aligned} if it is aligned with any Meyerian subset of $G$.
\end{definition}
Clearly, $\R^n$ does not admit any non-trivial universally aligned normal subgroups (since no subspace is aligned with respect to an arbitrary uniform lattice), hence the existence of universally aligned normal subgroups is a non-abelian phenomenon. We will construct examples of universally aligned normal subgroups in nilpotent Lie groups below, after establishing various different characterizations of alignment which are summarized in the following theorem. 
\begin{theorem}[Characterizations of alignment]\label{FiberVsImage}
Let $(G, N, Q, \pi, s)$ be an extension of lcsc groups and $\Lambda \subset G$ be a uniform approximate lattice and $\Delta := \pi(\Lambda)$. Then the following are equivalent.
\begin{enumerate}[(i)]
\item There exists $\delta \in \Delta$ such that the fiber $\Lambda_\delta$ is relatively dense.
\item There exists a relatively dense subset $\Theta \subset \Delta$ such that for every $\delta \in \Theta$ the fiber $\Lambda_\delta$ is relatively dense.
\item $\Lambda$ is $\pi$-aligned, i.e.\ the projection $\Delta$ is uniformly discrete.
\item There exists $\delta \in \Delta$ such that the fiber $\Lambda_\delta$ is Meyerian.
\item There exists a relatively dense symmetric subset $\Theta \subset \Delta$ such that for every $\delta \in \Theta$ the fiber $\Lambda_\delta$ is Meyerian.
\item The projection $\Delta$ is a uniform approximate lattice.
\end{enumerate}
\end{theorem}
\begin{remark}[Meyerian sets vs.\ uniform approximate lattices]\label{SymmetricNotNeeded}
The theorem does not apply directly to Meyerian sets wich are not symmetric, but if $\Lambda_o \subset G$ is a Meyerian subset, then we can apply the theorem to the associated uniform approximate lattice $\Lambda := \Lambda_o^{-1}\Lambda_o$. We claim that if $\pi$ is aligned with $\Lambda$, then it is aligned with $\Lambda_o$, hence in particular \emph{$\pi$ is universally aligend iff it is aligned with every uniform approximate lattice in $G$.} To prove the claim we first observe that $\pi(\Lambda_o)$ is always relatively dense, since if if $G = \Lambda_o K$ with $K$ compact, then $Q = \pi(\Lambda_o)\pi(K)$ and $\pi(K)$ is compact. Now if $\pi$ is aligned with $\Lambda$, then $\pi(\Lambda)$ is a uniform approximate lattice, hence if $k \in \mathbb N$, then $(\pi(\Lambda_o)^{-1}\pi(\Lambda_o))^k=\pi(\Lambda)^k$ is uniformly discrete for all $k \in \mathbb N$. Then also $\pi(\Lambda_o)$ is uniformly discrete, hence Delone, and it is Meyerian by (M2) from Lemma \ref{MProperties}.
\end{remark}
\begin{remark} Some of the implications in Theorem \ref{FiberVsImage} are immediate:
\begin{itemize}
\item Firstly, the implications (ii)$\Rightarrow$(i), (v)$\Rightarrow$(iv)$\Rightarrow$(i) and (vi)$\Rightarrow$(iii) hold by definition.
\item Secondly, the image of an approximate subgroup under a group homomorphism is always an approximate group. Moreover, if $\Lambda$ is relatively dense in $G$, then $\Delta$ is relatively dense in $Q$ by the argument in Remark \ref{SymmetricNotNeeded}. In view of Characterization (m3) from Lemma \ref{mProperties} this shows that (iii)$\Leftrightarrow$(vi).
\item Dually one can observe that if $\Lambda \subset G$ is uniformly discrete, then the fibers $\Lambda_\delta$ are uniformly discrete for all $\delta \in \pi(\Lambda)$. Indeed, if there are $x_n \neq y_n \in \Lambda_\delta$ such that $x_n^{-1}y_n \to e$, then, by definition, $x_ns(\delta), y_ns(\delta) \in \Lambda$ and
\[
(x_ns(\delta))^{-1} y_ns(\delta) = s(\delta)^{-1}x_n^{-1}y_n s(\delta) \to s(\delta)^{-1}e s(\delta) = e,
\]
contradicting uniform discreteness of $\Lambda$. Now if $\Lambda$ is Meyerian, then $(\Lambda^{-1}\Lambda)^k$ is uniformly discrete for all $k \in \mathbb N$, hence so is $(\Lambda^{-1}\Lambda)^k \cap N = ((\Lambda^{-1}\Lambda)^k)_e$ by the previous argument. It thus follows from \eqref{Fiber} that for all $\delta \in \Delta$,
\[
(\Lambda_\delta^{-1}\Lambda_{\delta})^k \subset ((\Lambda s(\delta)^{-1})^{-1} \Lambda s(\delta)^{-1})^k \cap N = s(\delta)((\Lambda^{-1}\Lambda)^k \cap N) s(\delta)^{-1},\]
which shows that $(\Lambda_\delta^{-1}\Lambda_{\delta})^k$ is uniformly discrete. We thus see from Characterization (M2) from Lemma \ref{MProperties} that a fiber $\Lambda_\delta$ is Meyerian if and only if it is relatively dense. This shows that (i)$\Leftrightarrow$(iv) and (ii)$\Leftrightarrow$(v).
\end{itemize}
We are thus left with proving the implications (i)$\Rightarrow$(iii) and (iii)$\Rightarrow$(ii).
\end{remark}
As far as the implication (i)$\Rightarrow$(iii) is concerned, it is a special case of the following more general observation. Here a subset of a lcsc group $G$ is called \emph{locally finite} if it is discrete and closed. Every uniformly discrete set has this property, and hence powers of approximate lattices have this property by Characterization (m2) from Lemma \ref{mProperties}. 
\begin{proposition}\label{DenseFiberUD} Let $\Lambda \subset G$ be a subset such that $\Lambda^{-1}\Lambda^2 \subset G$ is locally finite and let $\Delta: = \pi(\Lambda)$. Assume that for some $\delta \in \Delta$ the fiber $\Lambda_\delta$ is relatively dense in $N$.
Then $\Delta$ is uniformly discrete.
\end{proposition}
\begin{proof}  Let $K_\delta \subset N$ be compact such that $N = \Lambda_\delta K_\delta$ and $\lambda_k, \mu_k \in \Lambda$ such that $\pi(\lambda_k)^{-1}\pi(\mu_k) \to e$. There then exist $n_k \in N$ such that $\lambda_k^{-1}\mu_k n_k \to e$. Write $n_k = \xi_k u_k$ where $\xi_k \in \Lambda_\delta$ and $u_k \in K_\delta$. Then $\xi_k s(\delta) \in \Lambda$ and hence
\[
 (\lambda_k^{-1}\mu_k \xi_k s(\delta))(s(\delta)^{-1}u_k) = \lambda_k^{-1}\mu_k n_k \to e.
\]
Since $s(\delta)^{-1}u_k \in s(\delta)^{-1}K_\delta$, we may assume by passing to a subsequence that $s(\delta)^{-1}u_k$ converges to some $g \in G$, and hence
\[
\alpha_k := \lambda_k^{-1}\mu_k \xi_k s(\delta) \to g^{-1}.
\]
Now, by definition, $\alpha_k \in \Lambda^{-1}\Lambda^2$, which is locally finite. Thus for all sufficiently large $k$ we have $\alpha_k = g^{-1}$, and since $\xi_k \in N$ we deduce that
\[
\pi(\lambda_k)^{-1}\pi(\mu_k) = \pi(g^{-1})\delta^{-1}. 
\]
Since $\pi(\lambda_k)^{-1}\pi(\mu_k) \to e$ we have $ \pi(g^{-1})\delta^{-1} = e$, and hence $\pi(\lambda_k) = \pi(\mu_k)$ for all sufficiently large $k$. Thus $e$ is not an accumulation point of $\Lambda^{-1}\Lambda$. 
\end{proof}
We mention in passing that, unlike Theorem \ref{FiberVsImage}, Proposition \ref{DenseFiberUD} also applies to Meyerian sets which are not symmetric:
\begin{corollary}\label{MeyerianDenseFiberUD} Let $\Lambda \subset G$ be Meyerian and let $\Delta := \pi(\Lambda)$. Assume that for some $\delta \in \Delta$ the fiber $\Lambda_\delta$ is relatively dense in $N$. Then $\Delta =\pi(\Lambda)$ is Meyerian.
\end{corollary}
\begin{proof} Consider $\Theta := \Lambda^{-1}\Lambda$. Since $\Theta$ is a uniform approximate lattice, $\Theta^{-1}\Theta^2$ is locally finite. If $\Lambda$ has a relatively dense fiber, then $\Theta$ has a relatively dense fiber. It thus follows from the lemma that $\pi(\Theta)$ is uniformly discrete, and since $\Delta$ is contained in a translate of $\pi(\Theta)$, we deduce that it is uniformly discrete as well, hence Meyerian.
\end{proof}
At this point we have completed the proof of Theorem \ref{FiberVsImage} except for the implication (iii)$\Rightarrow$(ii), which will be proved in the next subsection using dynamical tools.  Before we turn to this proof, let us discuss a possible strengthening on Condition (ii) of the theorem. Given a Delone set $\Lambda \subset G$ with projection $\Delta := \pi(\Lambda)$ and $R>0$ we define 
\[
\Delta^{(R)} := \{\delta \in \Delta\mid \Lambda_\delta \text{ is }R\text{-relatively dense in }N\}.
\]
Condition (ii) says that $\bigcup \Delta^{(R)}$ is relatively dense in $Q$. We say that $\Lambda$ has \emph{$R$-large fibers} if $\Delta^{(R)}$ is relatively dense in $Q$ for some fixed $R > 0$.
\begin{corollary}[Large aligned approximate lattices]\label{CorSquares} Let $(G, N, Q, \pi, s)$ be a measurably split extension and let $\Lambda, \Lambda'$ be uniform approximate lattices in $G$ such that $\Lambda^2 \subset \Lambda'$. Then $\Lambda$ is $\pi$-aligned if and only if $\Lambda'$ is $\pi$-aligned, and in this case there exists $R>0$ such that $\Lambda'_e$ is $R$-relative dense and $\Lambda'$ has $R$-large fibers. In particular, if $\Lambda$ is a uniform approximate lattice, and $\Lambda^2$ is $\pi$-aligned, then the identity fiber $\Lambda^2_e$ is a uniform approximate lattice in $N$. 
\end{corollary}
\begin{proof} Let $\Delta := \pi(\Lambda)$ and $\Delta' := \pi(\Lambda')$. If $\Lambda'$ is $\pi$-aligned, then $\Delta'$ is uniformly discrete. Since $\Delta'$ contains a translate of $\Delta$, it follows that $\Delta$ is uniformly discrete, hence $\Lambda$ is $\pi$-aligned. Conversely, if $\Lambda$ is $\pi$-aligned, then $\Lambda$ contains a relatively dense fiber. It follows that $\Lambda^2$ and hence $\Lambda'$ contain a relatively dense fiber as well. This shows that $\Lambda'$ is $\pi$-aligned. We have thus established that $\Lambda$ is $\pi$-aligned if and only if $\Lambda'$ is $\pi$-aligned.

Now let $\Theta \subset \Delta$ as in Condition (ii) of Theorem \ref{FiberVsImage} and let $\delta_o \in \Theta$. Pick $R>0$ such that $\Lambda_{\delta_o}$ is $R$-relatively dense. Then for every $\delta \in \Theta$ we have a chain of inclusions
\[
(\Lambda_{\delta_o}s(\delta_o))(\Lambda_\delta s(\delta)) \subset \Lambda^2_{\delta_o\delta}s(\delta_o\delta) \subset \Lambda'_{\delta_o\delta}s(\delta_o\delta),
\]
i.e. $\Lambda'_{\delta_o\delta}$ contains a translate of $\Lambda_{\delta_o}$. Thus if we define $\Theta^{(2)} := \delta_o \Theta \subset \Delta'$, then for every $\theta \in \Theta^{(2)}$ the fiber $\Lambda'_\theta$ is $R$-relatively dense. Now the set $\Theta^{(2)}$ is relatively dense, since $\Theta$ is, and it contains the identity since $\Theta$ is symmetric and thus $e = \delta_o\delta_o^{-1} \in \Theta^{(2)}$. This finishes the proof.
\end{proof}

Many natural examples of uniform approximate lattices contain the square of a uniform approximate lattices. For example, this is true for the class of model sets discussed in Subsection \ref{SecModel} below, since every model set contains the square of a model set associated with the same cut-and-project scheme, but with smaller window. However, It is not the case that \emph{every} uniform approximate lattice contains the square of a uniform approximate lattice. The answer is negative, even in $\R$ (despite the fact that every uniform approximate lattice in $\R$ is a \emph{subset} of a model set), but explicit counterexamples are hard to come by.
\begin{example}[A uniform approximate lattice in $\R$ not containing a square]
As far as we know, K\v{r}\'{i}\v{z} \cite[Proposition 33]{Kr87} was the first to construct, albeit in a very different language, a uniform approximate lattice $\Lambda$ in $\bR$ which does not contain any set of the form $B-B$, where $B \subset \bR$ is relatively dense. In particular, such $\Lambda$ cannot contain the square of a uniform approximate lattice.
Since the translation between K\v{r}\'{i}\v{z}' language and ours is not immediate, we take some time here to explain the connection. 

We say that a set $P \subset \bZ$ is \emph{density intersective} if $P \cap (A-A) \neq \emptyset$ for every subset $A \subset \bN$ with positive upper
asymptotic density, and we say that $R \subset \bZ$ is \emph{chromatically intersective} if $R \cap (B-B) \neq \emptyset$ for every relatively dense subset 
$B \subset \bZ$. Every density intersective set is chromatically intersective, but the converse is not true: As explained in \cite[Subsection 3.3]{Mc}, Kriz' Theorem (which is stated in graph-theoretical terms in his paper) implies that there exists a chromatically intersective set $R$, which is \emph{not} density intersective. Let us fix a set $A \subset \bZ$ with positive upper asymptotic density such that $R \cap (A-A) = \emptyset$, and define $\Lambda := A-A$. By a result 
of F\o lner \cite{Fo54}, the set $\Lambda$ is relatively dense in $\bZ$, and thus in $\bR$. Since it is symmetric and contains the identity, $\Lambda$ is an
approximate lattice in $\bR$. We claim that $\Lambda$ does not contain a set of the form $B-B$, where $B \subset \bR$ is relatively dense. Indeed, if it did,
then $B$ must clearly be a relatively dense subset of $\bZ$. However, since $R$ is chromatically intersective,
\[
\emptyset \neq R \cap (B-B) \subset R \cap \Lambda = \emptyset,
\]
which is a contradiction.
\end{example}

\subsection{Killing inessential fibers}

The purpose of this subsection is to complete the proof of Theorem \ref{FiberVsImage} by establishing the remaining implication (iii)$\Rightarrow$(ii). Throughout this subsection we denote by $(G, N, Q, \pi, s)$ an extension of lcsc groups. Given a subset $\Lambda \subset G$ with projection $\Delta := \pi(\Lambda)$ we define the \emph{essential} and \emph{inessential part} of $\Delta$ by
\[
{\Delta}_{\rm ess} := \{\delta \in \Delta \mid \Lambda_\delta \subset N \text{ relatively dense}\} \qand \Delta_{\rm iness} := \Delta \setminus \Delta_{\rm ess}.
\]
Fibers over points in ${\Delta}_{\rm ess}$ and  ${\Delta}_{\rm iness}$ are called \emph{essential} and \emph{inessential fibers} respectively.
\begin{lemma}[Killing inessential fibers]\label{KillOneFiber}
Let $\Lambda \subset G$ be a non-empty subset such that $\Delta := \pi(\Lambda)$ is discrete and let $\delta_o \in \Delta_{\rm iness}$. Then there exists a sequence $(n_k)$ in $N$ with the following properties:
\begin{enumerate}[(i)]
\item $n_k \Lambda$ converges to some $\Lambda' \in \Omega_{\Lambda}$.
\item $\pi(\Lambda') \subset \Delta \setminus\{\delta_o\}$.
\item If for some $R>0$ and $\delta \in \Delta$ the fiber $\Lambda_\delta$ is $R$-relatively dense in $N$, then so is $\Lambda'_\delta$.
\end{enumerate}
\end{lemma}
\begin{proof} Since $\Lambda_{\delta_0}$ is not relative dense in $N$ we can find for every $R>0$ and element $a_R \in N$ such that $\overline{B}_R(a_R) \cap \Lambda_{\delta_o} = \emptyset$. Then $a_R^{-1}\Lambda_{\delta_o} \cap \overline{B}_R(e) = \emptyset$ and hence $a_R^{-1}\Lambda_{\delta_o} \to \emptyset$ in $\mathcal C(N)$. Since $\mathcal C(G)$ is compact, we can find a sequence of radii $R_k$ such that $a_{R_k}^{-1}\Lambda$ converges to some $\Lambda'$ in $\mathcal C(G)$. We fix such a sequence once and for all and set $n_k := a_{R_k}^{-1}$. Then the sequence $(n_k)$ satisfies (i), and we will show that it also satisfies (ii) and (iii). Indeed we have
\[
\Lambda = \bigcup_{\delta \in \Delta} \Lambda_\delta \times \{\delta\} \quad \Rightarrow \quad n_k.\Lambda =  \bigcup_{\delta \in \Delta} (n_k.\Lambda_\delta) \times \{\delta\} 
\]
and we claim that
\[
\Lambda':= \bigcup_{\delta \in \Delta}\Lambda'_\delta \times \{\delta\},\text{ where }\Lambda'_\delta = \lim_{k \to \infty} n_k \Lambda_\delta.
\]
This claim implies (ii), since $n_k \Lambda_{\delta_o} = a_{R_n}^{-1}\Lambda_{\delta_o} \to \emptyset$, and also (iii), since being $R$-relatively dense is a closed property in the Chabauty-Fell topology. 

To prove the claim we recall that by Lemma \ref{CFConvergence} the assumption $n_k\Lambda \to \Lambda'$ is equivalent to the two conditions
\begin{eqnarray*}
(n_{k_j} \in N_{k_j}, (t_{k_j}, \delta_{j_k}) \in \Lambda, n_{k_j}(t_{k_j}, \delta_{k_j}) \to (t, \delta)) \Rightarrow (t, \delta) \in \Lambda',\\
\forall (t, \delta) \in \Lambda'\;\exists (t_k, \delta_k) \in \Lambda:\; n_k(t_k, \delta_k) \to (t, \delta). 
\end{eqnarray*}
Now we use the fact that for all $n \in N$ we have $\pi(n \Lambda) = \pi(\Lambda) = \Delta$ and that $\Delta$ is discrete. The condition $n_{k_j}(t_{k_j}, \delta_{k_j}) \to (t, \delta)$ thus implies that $n_{k_j}\delta_{k_j} = \delta$ for sufficiently large $j$ and in particular $\delta \in \Delta$. We may thus restate the first condition as
\[
(n_{k_j} \in N_{k_j}, t_{k_j}\in \Lambda_\delta, n_{k_j}t_{k_j} \to t) \Rightarrow t \in \Lambda'_\delta,
\]
Similarly, using again discreteness of $\Delta$, we can reformulate the second condition as 
\[
\forall t \in \Lambda_\delta'\;\exists t_k \in \Lambda_\delta:\; n_kt_k \to t. 
\]
Appealing again to Lemma \ref{CFConvergence} we deduce that $n_k \Lambda_\delta \to \Lambda'$ as claimed.
\end{proof}
\begin{corollary}[Fibrocide lemma]\label{Fibrocide}
Let $\Lambda \subset G$ be a non-empty countable subset such that $\Delta := \pi(\Lambda)$ is discrete. Then there exists a sequence $(n_k)$ in $N$ with the following properties:
\begin{enumerate}[(i)]
\item  $n_k \Lambda$ converges to some $\Lambda' \in \Omega_{\Lambda}$.
\item $\pi(\Lambda') = \Delta_{\rm ess}$.
\item For every $\delta \in \Delta_{\rm ess}$ the fiber $\Lambda'_\delta$ is relatively dense in $N$.
\end{enumerate} 
\end{corollary}
\begin{proof} We fix an enumeration of the countable set $\Delta_{\rm iness}$, say $\Delta_{\rm iness} = \{\delta_1, \delta_2,\dots\}$ and for $\delta \in \Delta_{\rm ess}$ we choose $R_\delta >0$ such that $\Lambda_\delta$ is $R$-relatively dense. By Lemma \ref{KillOneFiber} we find a sequence $(n_k^{(1)})$ in $Z$ such that $n_k^{(1)}\Lambda \to \Lambda^{(1)}$ with $\pi(\Lambda^{(1)}) \cap \Delta_{\rm iness} \subset \{\delta_2, \dots\}$ and such that $\Lambda^{(1)}_\delta$ is $R_\delta$-relatively dense for all $\delta \in {\rm Q}^{\rm ess}$. Another application of the lemma then yields a sequence $(n_k^{(2)})$ in $Z$ such that $n_k^{(2)}\Lambda^{(1)} \to \Lambda^{(2)}$ with $\pi(\Lambda^{(2)}) \cap \Delta_{\rm iness} \subset \{\delta_3, \dots\}$ and such that $\Lambda^{(1)}_\delta$ is still $R_\delta$-relatively dense for all $\delta \in \Delta_{\rm ess}$. Inductively we produce sets $\Lambda^{(1)}, \Lambda^{(2)}, \dots \in \Omega_{\Lambda}$ with the following properties:
\begin{enumerate}[(i)]
\item $\Lambda^{(n)} \in \overline{N.\Lambda} \subset \Omega_\Lambda$.
\item $\pi(\Lambda^{(n)}) = \Delta_{\rm ess} \cup \{\delta_{n+1}, \delta_{n+2}, \dots\}$.
\item $\Lambda^{(n)}_\delta$ is $R_\delta$-relatively dense for all $\delta \in \Delta_{\rm ess}$ and $n \in \mathbb N$.
\end{enumerate}
 By compactness of $\overline{N.\Lambda}$ a subsequence $\Lambda^{(n_k)}$ converges to some $\Lambda' \in \overline{N.\Lambda}$ and by discreteness of $\Delta$ we have (as in the proof of Lemma \ref{KillOneFiber})
 \[
 \Lambda^{(n_k)}_\delta \to \Lambda'_\delta \text{ for all }\delta \in \Delta.
 \]
 If $\delta \in \Delta_{\rm iness}$, then we deduce that $\Lambda'_\delta = \emptyset$, which establishes (ii). If $\delta \in \Delta_{\rm ess}$, then all of the sets $\Lambda^{(n_k)}_\delta$ are $R_{\delta}$-relatively dense, hence also $\Lambda'_\delta$ is $R_\delta$-relatively dense. This establishes (iii) and finishes the proof.
\end{proof}
Using fibrocide we can now complete the proof of Theorem \ref{FiberVsImage}.
\begin{proof}[Proof of Theorem \ref{FiberVsImage}] Only the implication (iii)$\Rightarrow$(ii) remains to prove. Assume for contradiction that $\Delta$ is discrete, but that $\Delta_{\rm ess}$ is not relatively dense in $Q$. Then by Corollary \ref{Fibrocide} there exist $n_k \in N$ such that $n_k \Lambda$ converges in $\mathcal C(G)$ to some $\Lambda'$ with $\pi(\Lambda') = \Delta_{\rm ess}$ not relatively dense in $Q$. On the other hand, if $R>0$ is chosen such that $\Lambda$ is $R$-relatively dense in $G$, then each of the translates $n_k\Lambda$ is also $R$-relatively dense in $G$, and hence the limit $\Lambda'$ is $R$-relatively dense in $G$. This implies that $\pi(\Lambda')$ is relatively dense in $Q$, which is a contradiction.
\end{proof}

\subsection{A fiber dichotomy for one-dimensional extensions}
We conclude this section by a lemma concerning fibers of Meyerian sets in one-dimensional extensions. Note that if $\Gamma$ is a uniform lattice in $\R^n$ and $V$ is a codimension one subspace of $\R^n$, then any projection of $\Gamma$ onto $V$ either is aligned, or otherwise all of the fibers are singletons. We can generalize this observation to our setting by using the following observation of A. Fish.
\begin{lemma}[Fish]\label{FishLemma} Let $\Delta \subset \R$ be an approximate subgroup. Then $\Delta$ is relatively dense in $\R$ if and only if it is infinite.\qed
\end{lemma}
Using this lemma we can show that when projecting uniform approximate lattices along one-dimensional extensions either all fibers are uniformly finite, or a lot of fibers are infinite. 
\begin{proposition}[Fiber dichotomy for $\R$-extensions]\label{OneDimDichotomy}
Let  $(G, N, Q, \pi, s)$ be an extension of lcsc groups with $N \cong \R$ and let $\Lambda \subset G$ be a uniform approximate lattice. Then exactly one of the following two alternatives hold:
\begin{enumerate}
\item $\Lambda$ is $\pi$-aligned, and hence the fiber over a relatively dense subset of $\pi(\Lambda)$ are relatively dense in $N$.
\item For every fixed $k \in \mathbb N$ the fibers of $\Lambda^k$ are uniformly finite. 
\end{enumerate}
\end{proposition}
\begin{proof} Let $\Delta := \pi(\Lambda)$. We first observe that 
\[
(\Lambda^{2k})_e \supset \bigcup_{\delta \in \Delta^k}(\Lambda^k)_\delta  s(\delta)^{-1}(\Lambda^k)_{\delta^{-1}})s(\delta),
\]
and hence if the fibers of $\Lambda^k$ are not uniformly finite, then the identity fiber $(\Lambda^{2k})_e$ of $\Lambda^{2k}$ is infinite. In this case $\Lambda^{2k}_e = (\Lambda^k)^2 \cap N$ is an infinite approximate subgroup of $N \cong \R$ by Lemma \ref{BGT}, and hence relatively dense in $N$ by Lemma \ref{FishLemma}. By Theorem \ref{FiberVsImage} this implies that $\pi(\Lambda^k)$ is uniformly discrete, and hence $\pi(\Lambda)$ is uniformly discrete, which shows that $\Lambda$ is $\pi$-aligned. This shows that (1) holds whenever (2) fails. By Theorem \ref{FiberVsImage}, (1) implies that $\Lambda$ has a relatively dense fiber, hence an infinite fiber, which excludes (2), showing that exactly one of the two alternatives holds.
\end{proof}

\section{Examples of aligned Meyerian sets}\label{SecExamples}

\subsection{Aligned model sets}\label{SecModel}
An important class of examples of Meyerian sets is given by \emph{model sets} in lcsc groups \cite{MeyerBook, BH1}. We briefly recall the definition; we warn the reader that the terminology varies in the literature.
\begin{definition} Let $G$ and $H$ be lcsc groups and denote by $\pi_G$ and $\pi_H$ the coordinate projections of $G \times H$, and let $\Gamma < G \times H$ be a discrete subgroup and $W \subset G$ compact. Then the associated \emph{cut-and-project set} is
\[
\Lambda := \Lambda(G, H, \Gamma, W) := \pi_G((G \times W) \cap \Gamma).
\]
The triple $(G, H, \Gamma)$ is called a \emph{uniform cut-and-project scheme} if $\Gamma$ is a uniform lattice which projects injectively to $G$ and densely to $H$, and $W \subset H$ is called a \emph{window} if it is compact with non-empty interior. If $(G, H, \Gamma)$ is a uniform cut-and-project scheme and $W$ is a window, then $ \Lambda(G, H, \Gamma, W)$ is called a \emph{model set}.
\end{definition}
If $(G, H, \Gamma)$ is a cut-and-project scheme, then we set $\Gamma_G := \pi_G(\Gamma)$ and $\Gamma_H := \pi_H(\Gamma)$; we then define the associated $\ast$-map $\tau: \Gamma_G \to H$ by the relation $ \tau \circ \pi_G|_{\Gamma} = \pi_H|_\Gamma$, which is well-defined since $ \pi_G|_{\Gamma}$ is injective. We then have
\[
\Lambda := \Lambda(G, H, \Gamma, W) = \tau^{-1}(W),\; \Gamma_G = \langle \Lambda \rangle,\; \Gamma_H = \tau(\Gamma_G)\;\text{and}\; \Gamma = {\rm graph}(\tau).
\]
Such a description is typically not available for more general cut-and-project sets.
\begin{example} Let $G := H := \R$, $\Gamma := \{(a+ b\sqrt 2, a-b\sqrt 2) \in \R \times \R \mid a,b \in \Z\}$ and $W := [-R, R]$ for some $R>0$. Then the associated model set in $\R$ is given by
\begin{equation}\label{ThetaR}
\Theta_R = \{a + b \sqrt 2 \in \R \mid a,b \in \Z, \, |a-b\sqrt 2| \leq R\}.
\end{equation}
The associated $*$-map is given by Galois conjugation $\tau: \Z[\sqrt 2] \to \R$, $a+b \sqrt 2 \mapsto (a+b\sqrt 2)^* :=  a-b\sqrt 2$.
\end{example}
\begin{remark}[Model sets and uniform approxiate lattices]\label{CuPVsModel}
Concerning the relation between model sets and uniform approximate lattices we observe:
\begin{enumerate}[(i)]
\item Every model set is a Meyer set, and hence every relatively dense subset of a model set is a Meyer set \cite{MeyerBook, BH1}. For example, the sets $\Theta_R$ from \eqref{ThetaR} are uniform approximate lattices. 
\item If $(G, H, \Gamma)$ is a uniform cut-and-project scheme and $W$ is a symmetric window which contains the identity, then the model set $ \Lambda(G, H, \Gamma, W)$ is symmetric and contains the identity, hence is a uniform approximate lattice by (i).
\item General cut-and-project sets need not be relatively dense, even if $\Gamma$ is a uniform lattice; counterexamples are given by \emph{weak model sets} like the set of primitive lattice points in $\Z^2$, cf. \cite{BG13}.
\item However, every cut-and-project set is necessarily \emph{discrete}, and even of finite local complexity. Indeed, if $\Lambda = \Lambda(G, H, \Gamma, W)$ is a cut-and-project set 
\begin{eqnarray*} \Lambda^{-1}\Lambda &=& \pi_G((G\times W)\cap \Gamma)^{-1} \pi_G((G\times W)\cap \Gamma) \quad = \quad \pi_G((G\times W^{-1})\cap \Gamma) \pi_G((G\times W)\cap \Gamma) \\
&=& \pi_G(((G\times W^{-1})\cap \Gamma)((G\times W)\cap \Gamma)) \quad \subset \quad  \pi_G(((G\times W^{-1})(G\times W))\cap \Gamma)\\
&\subset& \pi_G((G \times W^{-1}W) \cap \Gamma),
\end{eqnarray*} 
and hence for every compact subset $K \subset G$ we have
\[
\Lambda^{-1}\Lambda \cap K  \quad \subset \quad \pi_G((G \times W^{-1}W) \cap \Gamma)\cap K \quad \subset \quad \pi_G((K \times W^{-1}W) \cap \Gamma),
\]
which is finite since $K \times W^{-1}W$ is compact and $\Gamma$ is locally finite, since it is discrete and a subgroup. 
\item In the generality of lcsc groups it is not known whether every uniform approximate lattice is a relatively dense subset of a model set (\cite[Problem 1]{BH1}).
\end{enumerate}
\end{remark}
In the compactly-generated abelian case, the following is a famous theorem of Meyer \cite{MeyerBook}.
\begin{theorem}[Meyer embedding theorem]\label{MET} Every Meyer set (in particular, every uniform approximate lattice) in a compactly-generated lcsc abelian group can be realized as a relatively dense subset of a model set.\qed
\end{theorem}
This was recently extended to the case of connected nilpotent Lie groups by Machado \cite{Machado}:
\begin{theorem}[Machado embedding theorem]\label{Machado} Every uniform approximate lattice in a connected nilpotent Lie group $G$ can be realized as a relatively dense subset of a model set. 
\end{theorem}
\begin{remark}\label{Machado+} Actually, Machado establishes several more refined results in \cite{Machado}. We need the following version below: If $G$ is $1$-connected and $\Lambda \subset G$ is a uniform approximate lattice, then $\Lambda \subset \Lambda(G,H, \Gamma, W)$, where $H$ is a $1$-connected nilpotent Lie group, $\Gamma$ is a uniform lattice in $G \times H$ and $W \subset H$ is a window. If we insist that $H$ be $1$-connected, then we cannot ensure that $\Gamma$ projects densely to $H$ (even if $G$ and $\Gamma$ are abelian), so in general $\Lambda(G, H, \Gamma, W)$ will not be a model set, but we will not need this property anyway.
\end{remark}
Questions about alignment of model sets to a given projection can sometimes be reduced to questions about alignment of the underlying lattice.
\begin{proposition}[Alignment of model sets vs. alignment of lattices]\label{ModelSetAligned} Let $(G ,H, \Gamma)$ be a cut-and-project scheme, and let $\pi: G \to \overline{G}$ and $\pi': H \to \overline{H}$ be surjective morphisms of topological groups. If $\Gamma$ is $(\pi \times \pi')$-aligned, then for every window $W$ the model set $\Lambda = \Lambda(G, H, \Gamma, W)$ is $\pi$-aligned. 
\end{proposition}
\begin{proof} Denote by $\pi_G, \pi_H$ (respectively $\pi_{\overline{G}}, \pi_{\overline H})$ the factor projections of $G \times H$ (respectively $\overline{G} \times\overline{H}$). By assumption, $\overline{\Gamma} := (\pi \times \pi')(\Gamma)$ is a discrete subgroup of $\overline{G} \times \overline{H}$. Moreover, since $W$ is compact, the set $\overline{W} := \pi'(W)$ is also compact. Since the diagram
\[\begin{xy}\xymatrix{
G \times H \ar[rrr]^{\pi_G}\ar[d]_{\pi \times \pi'} &&& G \ar[d]^{\pi}\\
 \overline{G} \times \overline{H} \ar[rrr]_{\pi_{\overline{G}}} &&& \overline{G}
}\end{xy}\]
commutes we deduce that
\begin{eqnarray*}
\pi(\Lambda) &=& \pi(\pi_G((G\times W) \cap \Gamma)) \quad = \quad  \pi_{\overline{G}}((\pi\times \pi')((G\times W) \cap \Gamma))\\
 &\subset&  \pi_{\overline{G}}((\pi\times \pi')(G\times W) \cap (\pi\times \pi')(\Gamma)) \quad = \quad \pi_{\overline{G}}((\overline{G} \times \overline{W}) \cap \overline{\Gamma}).
\end{eqnarray*}
Now $ \pi_{\overline{G}}((\overline{G} \times \overline{W}) \cap \overline{\Gamma}) = \Lambda(\overline{G}, \overline{H}, \overline{\Gamma}, \overline{W})$ is a cut-and-project set, hence uniformly discrete by Remark \ref{CuPVsModel}. It thus follows that $\pi(\Lambda)$ is uniformly discrete, i.e.\ $\Lambda$ is $\pi$-aligned.
\end{proof}
\begin{remark}\label{NilpotentModels+} The proof of Proposition \ref{ModelSetAligned} actually shows that if $\Lambda$ is a subset of a model set associated with the cut-and-project scheme $(G, H, \Gamma)$, then $\pi(\Lambda)$ is a subset of a cut-and-project set associated with the cut-and-project scheme $(\overline{G}, \overline{H}, \overline{W})$.
\end{remark}
\subsection{Nilpotent Lie groups and universally aligned towers}
We now apply Proposition \ref{ModelSetAligned} to the special case of uniform approximate lattices in nilpotent Lie groups (see Appendix \ref{AppNilpotent} for background and notation). We are going to show:
\begin{theorem}[Centers of $1$-connected nilpotent Lie groups are universally aligned]\label{MachadoConvenient} The center $Z(G)$ of a $1$-connected nilpotent Lie group $G$ is universally aligned.
\end{theorem}
\begin{remark}
\begin{enumerate}[(i)]
\item For Theorem \ref{MachadoConvenient} to hold, it is crucial that we divide by the full center of the group $G$. If $Z$ is merely a central normal subgroup of $G$, then the image of a Meyerian subset of $G$ in $G/Z$ need not be Meyerian. If $G$ is abelian, it can even be dense.
\item  It $\Lambda \subset G$ is an approximate lattice, then also $\Lambda^2$ is an approximate lattice in $G$, andTheorem \ref{MachadoConvenient} implies that both $\Lambda$ and $\Lambda^2$ are $\pi$-aligned. With Theorem \ref{FiberVsImage} and Corollary \ref{CorSquares} we conclude that $\Delta = \pi(\Lambda)$ and $\Lambda^2_e$ are approximate lattices in $Q$ and $N$ respectively and that there exists $R>0$ such that $\Lambda^2_e$ is $R$-relative dense and $\Lambda^2$ has $R$-large fibers. 
\end{enumerate}
\end{remark}
Theorem \ref{MachadoConvenient} is a consequence of Proposition \ref{ModelSetAligned}, Theorem \ref{Machado} and the following standard observation concerning lattices in nilpotent Lie groups.
\begin{proposition}\label{Raghunathan} Let $G$ be a $1$-connected nilpotent Lie group with center $Z(G)$ and let $\pi: G \to G/Z(G)$. Then every uniform lattice
 $\Gamma \cap Z(G)$ is $\pi$-aligned, and $\pi(\Gamma)$ is a uniform lattice in $G/Z(G)$.
\end{proposition}
\begin{proof} By  \cite[Prop. 2.17]{Raghunathan} the intersection $\Gamma \cap Z(G)$ is a uniform lattice in $Z(G)$. It then follows from Theorem \ref{FiberVsImage} that $\Gamma$ is $\pi$-aligned.\end{proof}
\begin{proof}[Proof of Theorem \ref{MachadoConvenient}] By the results of Machado mentioned in Remark \ref{Machado+} we can find a $1$-connected nilpotent Lie group $H$, lattice $\Gamma < G \times H$ and window $W \subset H$ such that $\Lambda \subset \Lambda(G, H, \Gamma, W)$. Since the center of $G\times H$ is given by $Z(G \times H) = Z(G) \times Z(H)$, it follows from Proposition \ref{Raghunathan} that the lattice $\Gamma$ is aligned to the projection $ G\times H \to G/Z(G) \times H/Z(H)$. It thus follows from Proposition \ref{ModelSetAligned} that $\Lambda(G, H, \Gamma, W)$ is $Z(G)$-aligned. Consequently, if $\pi: G \to G/Z(G)$ denotes the canonical projection, then $\pi(\Lambda) \subset \pi(\Lambda, G, H, W)$ is uniformly discrete, i.e.\ $\Lambda$ is $Z(G)$-aligned as well.
\end{proof}
The construction in Theorem \ref{MachadoConvenient} can of course be iterated. To discuss such iterations we introduce the following terminology.
\begin{definition} Let $G_1, \dots, G_{n+1}$ be $1$-connected nilpotent Lie groups and for every $j \in \{1, \dots, n\}$ let $A_j \lhd G_j$ be a non-trivial closed abelian normal subgroup. We say that $(G_1, \dots, G_{n+1})$ is an \emph{abelian tower} for $G_1$ if $G_{n+1}$ is abelian and $G_{j+1} = G_j/A_j$ is abelian for all $j=1, \dots, n$. This abelian tower is called a \emph{characteristic tower} (respectively a \emph{central tower}) if $A_j$ is characteristic (respectively central) in $G_j$ for every  $j \in \{1, \dots, n\}$.
\end{definition}
\begin{example} Let $G$ be an arbitrary nilpotent Lie group. Define $G_1 := G$, $G_2 := G_1/Z(G_1)$, $G_3 := G_2/Z(G_2)$ etc.\ until you reach an abelian group $G_{n+1}$. Since $G_1, \dots, G_{n+1}$ are $1$-connected by Remark \ref{ClosedSubgpNilpotentGp}, this defines a characteristic central tower for $G$ called the \emph{maximal central tower}. 
\end{example}
If $G = G_1, \dots, G_{n+1}$ is an abelian tower for $G$ and $1\leq j < k \leq n+1$, then we denote by $\pi_j^{k}: G_j \to G_{k}$ the canonical projections.
\begin{definition} Let $G$ be a $1$-connected nilpotent Lie groups and let $\Lambda \subset G$ be Meyerian. We say that an abelian tower $(G_1, \dots, G_{n+1})$ for $G$ is \emph{aligned} with $\Lambda$ if all of the projections $(\pi_{j}^k)_{1\leq j < k \leq n+1}$ are aligned with $\Lambda$. We say that $(G_1, \dots, G_{n+1})$ is \emph{universally aligned} if it is aligned with every uniform approximate lattice $\Lambda \subset G$.
\end{definition}
Iterating Theorem \ref{MachadoConvenient} yields:
\begin{corollary}[Existence of universally aligned towers]\label{UniversallyAlignedTower}  Every $1$-connected nilpotent Lie group $G$ admits a universally aligned characteristic central tower. In fact, the maximal central tower is universally aligned.\qed
\end{corollary}
\begin{remark}
If $(G_1, \dots, G_{n+1})$ is aligned with $\Lambda$, then the iterated projections $\Lambda_j := \pi_1^j(\Lambda)$ are uniform approximate lattice for all $j= \{1, \dots, n+1\}$. In particular, $\Lambda_{n+1}$ is a Meyer set in the abelian group $G_{n+1}$, and the uniform lattice $\Lambda_n$ in $G_n$ projects to $\Lambda_{n+1}$ and has a relatively dense set of fibers which are Meyer sets in the abelian group $A_n$. More generally, for all $j \in \{1, \dots, n\}$ the approximate lattice $\Lambda_{j}$ projects to $\Lambda_{j+1}$ and has a relatively dense set of fibers which are Meyer sets in the abelian group $A_j$. Thus if $\Lambda$ admits an aligned tower, then it can be constructed starting from a Meyer set in an abelian groups by iterated extensions, in which most fibers are again Meyer sets. 
\end{remark}
Let us mention that the maximal central tower is not the only universally aligned characteristic abelian tower in $G$. The following curious example will be discussed in Appendix \ref{AppCrazyTower} below.
\begin{example} Given a $1$-connected nilpotent Lie group $G$ we define
\[A[G] := \left\{\begin{array}{ll} Z(G), & \text{if } G \text{ is } 2\text{-step nilpotent},\\
 Z(C_G([G, G])), & \text{if } G \text{ is at least }  3\text{-step nilpotent}.
 \end{array}\right.\]
Then $A[G]$ is a closed connected characteristic abelian subgroup of $G$, and hence $G/A[G]$ is $1$-connected (see Remark \ref{ClosedSubgpNilpotentGp} and Lemma \ref{CGGG}). We thus obtain a characteristic abelian tower for $G$ by setting $G_1 := G$, $G_2:= G_1/A[G_1]$, $G_3 := G_2/A[G_2]$ etc. We will show in Appendix \ref{AppCrazyTower} that this tower is also universally aligned (but not central), and this proof does not use the Machado embedding theorem (Theorem \ref{Machado}). 
\end{example}

When we started working on approximate lattices in nilpotent Lie groups, the Machado embeddding theorem and Corollary \ref{UniversallyAlignedTower} were not yet available, and much of the theory below was developed with the tower from the previous example in mind. Nevertheless we will work with the maximal central tower below, since it is technically convenient to work with central extensions only.

\subsection{Split uniform approximate lattices and symplectic products}

We now provide some further examples of centrally aligned uniform approximate lattices, which are not a priori assumed to be subsets of model sets. Throughout this subsection let $(G, Z, Q, \pi, s)$ be a central extension with associated cocycle $\beta = \beta_s: Q \times Q \to Z$. Given subset $\Xi \subset Z$ and $\Delta \subset Q$ we denote their Cartesian product by 
$\Xi \oplus_\beta \Delta \subset  Z \oplus_\beta Q $. Under our standing identification of $G$ with $Z \oplus_\beta Q$ this corresponds to the subset $\Xi s(\Delta) \subset G$. We want to find conditions on $\Xi$ and $\Delta$ under which $\Lambda := \Xi \oplus_\beta \Delta$ is Meyerian or a uniform approximate lattice and aligned to $\pi$. 
\begin{definition} A uniform approximate lattice $\Lambda \subset G$ is called a \emph{split uniform approximate lattice} if it is $\pi$-aligned and of the form $\Lambda = \Xi \oplus_\beta \Delta$ for subsets  $\Xi \subset Z$ and $\Delta \subset Q$.
\end{definition}
There is an obvious necessary condition on $\Xi$ and $\Delta$ for $\Lambda$ to be $\pi$-aligned.
\begin{lemma}[Necessary condition]\label{SplitNecessaryCondition} If $\Lambda = \Xi \oplus_\beta \Delta$ is a split uniform approximate lattice in $G$, then $\Xi$ and $\Delta$ are uniform approximate lattices in $Z$ and $Q$ respectively.
\end{lemma}
\begin{proof}
If $\Lambda = \Xi \oplus_\beta \Delta$, then by construction $\pi(\Lambda) = \Delta$ and for all $\delta \in \Delta$ we have $\Lambda_\delta = \Xi s(\delta)^{-1}$. Thus if $\Lambda$ is a $\pi$-aligned uniform approximate lattice, then $\Delta$ has to be a uniform approximate lattice and $\Xi s(\delta)^{-1}$ has to be Meyerian for some $\delta \in \Delta$. The latter implies that $\Xi$ is Meyerian, and hence a uniform approximate lattice, since $\Xi = \Lambda_e$ is symmetric and contains the identity.
\end{proof}
\begin{remark}\label{UniformlyLargeFibersSplitCase}
From the proof of Lemma \ref{SplitNecessaryCondition} we see that if $\Lambda$ is split, then all of its fibers are translates of each other, hence they are all $R$-relatively dense for the same $R>0$. 
\end{remark}
Towards sufficient conditions for split uniform approximate lattices we observe:
\begin{lemma}  If $\Xi \subset Z$ and $\Delta \subset Q$ are relatively dense, then $\Lambda := \Xi \oplus_\beta \Delta$ is relatively dense in $G$, and if $\Xi$ and $\Delta$ are uniformly discrete, then $\Lambda$ is uniformly discrete.
\end{lemma}
\begin{proof} Assume that there exist compact subsets $K_Z \subset Z$ and $K_Q \subset Q$ such that $Z = \Xi K_Z$ and $Q = \Delta K_Q$. Given $(z, q) \in G$ we choose $\delta \in \Delta$ and $k_2 \in K_Q$ such that $q = \delta k_2$ and $\xi \in \Xi$ and $k_1 \in K_Z$ such that $\xi + k_1 = z - \beta(\delta, k_2)$. Then $
(\xi, \delta)(k_1, k_2) = (\xi + k_1 + \beta(\delta, k_2), \delta k_2) = (z,q)$, which shows that $\Lambda(K_1 \oplus_\beta K_2) = G$. This proves the first statement; for the second statement assume that $\Xi$ and $\Delta$ are uniformly discrete, let $\xi_n, \xi_n' \in \Xi$, $\delta_n, \delta_n' \in \Delta$ and assume that $(\xi_n, \delta_n)(\xi_n', \delta_n')^{-1} \to e$. By \eqref{CentralInverse} we have
\begin{eqnarray*}
(\xi_n, \delta_n)(\xi_n', \delta_n')^{-1} &=& (\xi_n, \delta_n)(-\xi'_n - \beta(\delta_n', (\delta_n')^{-1}), (\delta_n')^{-1})\\ &=& (\xi_n - \xi_n'- \beta(\delta_n', (\delta_n')^{-1} + \beta(\delta_n, (\delta_n')^{-1}), \delta_n(\delta_n')^{-1}) \quad \longrightarrow (0,e)
\end{eqnarray*}
Since $\Delta$ is uniformly discrete, considering the second coordinate yields $\delta_n = \delta_n'$ for all sufficiently large $n$, hence the first coordinate implies that $\xi_n - \xi_n' \to 0$, which by uniform discreteness of $\Xi$ yields $\xi_n = \xi_n'$ for sufficiently large $n$. This finishes the proof.
\end{proof}
It thus follows from Characterization (m1) from Lemma \ref{mProperties} that $\Lambda = \Xi \oplus \Delta$ is a $\pi$-aligned uniform approximate lattice if and only if $\Xi$ and $\Delta$ are uniform approximate lattices and $\Lambda^3$ is locally finite. Now if $\xi_i \in \Xi$ and $\delta_i \in \Delta$, then
\begin{eqnarray*}
\prod_{i=1}^3(\xi_i, \delta_i) &=& (\xi_1+\xi_2+ \xi_3 + \beta(\delta_1, \delta_2) + \beta(\delta_1\delta_2, \delta_3), \delta_1\delta_2\delta_3) \\
&=& (\xi_1+\xi_2+ \xi_3 + \beta(\delta_1, \delta_2\delta_3) + \beta(\delta_2, \delta_3), \delta_1\delta_2\delta_3).
\end{eqnarray*}
The second coordinate is contained in $\Delta^3$, which is uniformly discrete. Thus $\Lambda^3$ is uniformly discrete if and only if the set
\[
\Xi^3 + \{\beta(\delta_1, \delta_2) + \beta(\delta_1\delta_2, \delta_3)\mid \delta_1, \delta_2, \delta_3 \in \Delta\} \subset \Xi^3 + \beta(\Delta^2, \Delta) + \beta(\Delta, \Delta)
\]
is uniformly discrete. If $\beta(\Delta^2, \Delta) \subset \Xi^k$ for some $k \in \mathbb N$, then this set is contained in $\Xi^{3+2k}$, which is uniformly discrete by assumption. This proves:
\begin{proposition}[Sufficient conditions]\label{SplitSufficientCondition} If $\Xi \subset Z$ and $\Delta \subset Q$ are uniform approximate lattices and $\beta(\Delta^2, \Delta) \subset \Xi^k$ for some $k \in \mathbb N$, then $\Lambda = \Xi \oplus_\beta \Delta$ is a $\pi$-aligned uniform approximate lattice.\qed
\end{proposition}

We now provide concrete examples of split uniform approximate lattices in $1$-connected $2$-step nilpotent Lie groups. We recall from Proposition \ref{NormalFormNilpotent} that these are of the form $G = Z \oplus_\beta Q$ where $Q$ and $Z$ are real vector spaces of positive dimension and $\beta: Q \times Q \to Z$ is a non-degenerate, antisymmetric bilinear map. We will consider $G$ as a central extension of $Q$ by $Z$ and denote by $\pi: G \to Q$ the canonical projection. If $Z = Z(G)$ is one-dimensional, then $\beta$ is simply a symplectic form; for this reason we refer to split uniform-approximate lattices in $1$-connected $2$-step nilpotent Lie groups as \emph{symplectic products}. 

By Proposition \ref{SplitSufficientCondition} if $\Xi \subset Z$ and $\Delta \subset Q$ are uniform approximate lattices, then $\Xi \oplus_\beta \Delta$ is a symplectic product provided $\beta(\Delta^2, \Delta) \subset \Xi^k$ for some $k \in \mathbb N$. Since $\beta$ is bilinear we have $\beta(\Delta^2, \Delta) \subset \beta(\Delta, \Delta) + \beta(\Delta, \Delta)$, hence the condition simplifies to $\beta(\Delta, \Delta) \subset \Xi^{2k}$. 

We can now give a concrete example of a symplectic product in the Heisenberg group.
\begin{example} The three-dimensional Heisenberg group is given by $H_3(\R) = \R \oplus_\beta \R^2$ with underlying symplectic form $\beta((x_1, x_2), (y_1, y_2)) := x_1y_2 -x_2y_1$. If we abbreviate
\[
\tau: \Z[\sqrt 2] \to \R, \quad a+ b\sqrt 2\mapsto (a+ b\sqrt 2)^* := a-b\sqrt 2,
\]
then by \eqref{ThetaR} for every $R>0$ we obtain a uniform approximate lattice in $\R$ by
\[
\Theta_R = \{x \in \Z[\sqrt 2] \mid x^* \in [-R, R]\}.
\]
Note that for $R_1, R_2>0$ we have $\Theta_{R_1}\Theta_{R_2} \subset \Theta_{R_1R_2}$ since $\tau$ is a ring homomorphism  and thus
\[
|(x_1x_2)^*| = |x_1^*| \cdot |x_2^*| \leq R_1R_2 \quad (x_1 \in \Theta_{R_1}, x_2 \in \Theta_{R_2}).
\]
Thus if we choose $R_1, R_2, R_3 >0$ arbitrarily and define
\[
\Delta:= \Theta_{R_1} \times \Theta_{R_2} \qand \Xi:= \Theta_{R_3},
\]
then or all $k> 2R_1R_2/R_3$ we have $\beta(\Delta, \Delta) \subset 2\Theta_{R_1}\Theta_{R_2} \subset \Xi^k$. We thus deduce that $\Lambda := \Xi \oplus_\beta\Delta$ is a symplectic product in $H_3(\R)$. Explicitly,
\begin{equation}\label{HeisenbergLambda}\Lambda = \left\{(a_3+b_3\sqrt 2, ( a_1 + b_1 \sqrt 2,  a_2 + b_2 \sqrt 2)) \in H_3(\R)  \mid a_j, b_j \in \Z, \, \begin{array}{lcl} |a_1 - b_1 \sqrt 2| &<& R_1,\\  |a_2 - b_2 \sqrt 2| &<& R_2,\\  |a_3 - b_3 \sqrt 2| &<& R_3 \end{array}\right\}.
\end{equation}
\end{example}

\section{Existence of twisted fiber densities in aligned Meyerian sets}\label{SecHof}

\subsection{Asymptotic densities with respect to nice F\o lner sequences}
If $N$ is a lcsc group and $\Xi \subset N$ is a locally finite subset, then the \emph{lower} and \emph{upper asymptotic densities} of $\Xi$ with respect to 
a family $(F_T)$ of subsets of $N$ of positive Haar measure are defined as
\[
\underline{d}(\Xi) := \varliminf_T  \frac{|\Xi \cap F_T|}{m_H(F_T)} \qand \overline{d}(\Xi) := \varlimsup_T  \frac{|\Xi \cap F_T|}{m_H(F_T)}.
\]
It the two coincide then we say that $\Xi$ has an \emph{asymptotic density} with respect to $(F_T)$. In this section we are going to consider the questions of existence of asymptotic densities for fibers of an aligned uniform approximate lattices in certain amenable extension of lcsc groups. The sequences $(F_T)$ we will consider will be special F\o lner sequences.
\begin{definition}\label{Nice}
Let $N$ be an amenable group. A right-F\o lner sequence $(F_T)$ for $N$ is called \emph{nice} provided it satisfies the following properties:
\begin{itemize}
\item $(F_T)$ is nested. %and each $F_T$ is symmetric.
\item For every $K \subset N$ compact there exists $T_K >0$ such that for all $T>0$
\begin{equation}\label{Nice1}
F_TK \cup F_TK^{-1} \subset F_{T + T_K}.
\end{equation}
\item $(F_T)$ has \emph{exact volume growth}, i.e.\ if $m_N$ denote a left-Haar measure on $N$, then for every $T_0>0$ we have
\begin{equation}\label{Nice2}
\lim_{T \to \infty} \frac{m_N(F_{T+T_0})}{m_N(F_T)}=1.
\end{equation}
\end{itemize}
\end{definition}
If $N \cong \R^n$, then Euclidean balls $F_T := B_T(0)$ form a nice (right-)F\o lner sequence, and this is the example we are most interested in. For $1$-connected nilpotent Lie groups, it was established by Breuillard \cite{Breuillard} that balls with respect to suitable metrics form nice right-F\o lner sequences. We do not know, which other amenable groups admit nice F\o lner sequences. 

Concerning upper/lower asymptotic densities with respect to nice right-F\o lner sequences we observe the following general bounds.
\begin{lemma}\label{DensityBounds} Let $N$ be an amenable lcsc group, let $(F_T)$ be a nice right-F\o lner sequence in $N$ and let $\Xi \subset N$ be a locally finite subset.
\begin{enumerate}[(i)]
\item If $\Xi$ is $r$-uniformly discrete, then $\overline{d}(\Xi) \leq (m_N(B_r(e)))^{-1} < \infty$.
\item If $\Xi$ is $R$-relatively dense, then $\underline{d}(\Xi) \geq (m_N(B_{2R}(e)))^{-1} > 0$.
\end{enumerate}
\end{lemma}
\begin{proof} Let us abbreviate $\Xi_T := \Xi \cap F_T$.

(i) By \eqref{Nice1} we may choose $T_o>0$ such that for all $T>0$ we have $N_r(\Xi_T) \subset F_T B_r(e) \subset F_{T+T_o}$.
Since $\Xi$ is $r$-uniformly discrete we have $N_r(\Xi_T) = \bigsqcup_{\xi \in \Xi_T}B_r(\xi)$. This implies
\[
m_N(N_r(\Xi_T)) = |\Xi_T| \cdot m_N(B_r(e)),
\]
and hence
\[
\frac{|\Xi_T|}{m_N(F_T)} = \frac{m_N(N_r(\Xi_T))}{m_N(F_T)\cdot m_N(B_r(e))} \leq \frac{m_N(F_{T+T_o})}{m_N(F_T)} \cdot \frac{1}{ m_N(B_r(e))}.
\]
In view of \eqref{Nice2} this implies the desired bound.

(ii) By \eqref{Nice1} we may choose $T_o>0$ such that for all sufficiently large $T$ we have $F_{T-T_o}B_R(e) \subset F_T$. We claim that $F_{T-T_o} \subset N_{2R}(\Xi_T)$; this would imply $m_N(F_{T-T_o}) \leq |\Xi_T| \cdot m_N(B_{2R}(e))$ and hence
\[
\frac{|\Xi_T|}{m_N(F_T)} \geq \frac{m_N(F_{T-T_o})}{m_N(F_T)}   \cdot \frac{1}{m_N(B_{2R}(e))},
\]
which implies the desired bound in view of \eqref{Nice2}. Thus assume for contradiction that the claim fails. Then for arbitrarily large $T$ we coulde find $x \in F_{T-T_o}$ such that $d(x, \xi) \geq 2R$ for all $\xi \in \Xi_T$. On the other hand, since $\Xi$ is $R$-relatively dense we find $\eta \in \Xi$ such that $d(x, \eta) \leq R$. This implies $x^{-1}\eta \in B_R(e)$ and thus $\eta = x(x^{-1}\eta) \in F_{T-T_o}B_R(e) \subset F_T$, hence $\eta \in \Xi \cap F_T = \Xi_T$, a contradiction.
\end{proof}

\subsection{Twisted fiber densities}

For the remainder of this section we are going to consider Delone sets $\Lambda$ of finite local complexity which are aligned with a given extension $(G, N, Q, \pi, s)$ of lcsc groups, where $N$ is an amenable group which admits a nice right-F\o lner sequence $(F_T)$ (e.g.\ abelian or a $1$-connected nilpotent Lie group). We are going to show that generic fibers of elements of the hull of $\Lambda$ have densities with respect to $(F_T)$. In fact, we are going to establish a more general result concerning twisted versions of densities along such F\o lner sequences. This more general result will play a major role in our study of the spectral theory of hulls of Meyerian subsets of nilpotent Lie groups below. To state our result, we denote by ${\rm Hom}({N}, \mathbb T)$ the space of continuous group homomorphism from $N$ into the unit circle $\mathbb T \subset \mathbb C^\times$. 

\begin{theorem}[Existence of twisted asymptotic fiber densities]\label{BombieriTaylor1} Let $(G, N, Q, \pi, s)$ be an extension of lcsc groups, and assume that $N$ is amenable and admits a nice right-F\o lner sequence $(F_T)$. Let $\Lambda \subset G$ be a $\pi$-aligned Delone set of finite local complexity and assume that there exists a $G$-invariant probability measure $\nu_\Lambda$ on $\Omega_\Lambda$. Then for every $\xi \in {\rm Hom}({N}, \mathbb T)$ there exists a $G$-invariant $\nu_\Lambda$-conull subset $\Omega_\xi \subset \Omega_\Lambda$ such that for all $\Lambda' \in \Omega_\xi$ and $\delta' \in \pi(\Lambda')$ the following limit exists:
\[
D_\xi(\Lambda', \delta') := \lim_{T \to \infty} \frac{1}{m_N(F_T)}\sum_{z \in \Lambda'_{\delta'} \cap F_T} \overline{\xi(z)}
\]
\end{theorem}
The case where $Q = \{e\}$ and $N = G$ is abelian was treated by Hof in the context  of his work on the Bombieri--Taylor conjecture \cite{Hof95}. By definition, $D_1(\Lambda', \delta')$ is the asymptotic density of the fiber $\Lambda'_{\delta'}$ with respect to $(F_T)$. In the sequel, given a character $\xi \in {\rm Hom}({N}, \mathbb T)$, we thus refer to $D_{\xi}$ as the \emph{$\xi$-twisted fiber density function}. 

\subsection{The image of the hull under an aligned projection}
In this subsection we establish a general result about the image of hulls of FLC Delone sets under aligned projection. We will need this result for the proof of Theorem \ref{BombieriTaylor1},  but it is also of independent interest. For the moment we can allow $(G, N, Q, \pi, s)$ to be an arbitrary extension of lcsc groups, i.e.\ we do not need to assume that $N$ is amenable.
\begin{proposition}[Uniform uniform discreteness of the hull under an aligned projection]\label{ProjectionUUDiscrete} Let $\Lambda \subset G$ be a $\pi$-aligned FLC Delone set whose projection $\Delta := \pi(\Lambda)$ is uniformly discrete. Then for every $\Lambda' \in \Omega_{\Lambda}$ the projection $\Delta' := \pi(\Lambda')$ is uniformly discrete. Moreover, there exists a uniform constant $r>0$ such that for every $\Lambda' \in \Omega_{\Lambda}$ the projection $\Delta' := \pi(\Lambda')$ is $r$-uniformly discrete.
\end{proposition}
\begin{proof} Since $\Delta$ is uniformly discrete there exists $r>0$ such that for all distinct $\delta_1, \delta_2 \in \Delta$ we have $d(\delta_1, \delta_2)>r$. Now let $\Lambda' \in \Omega_{\Lambda}$ and $\Delta' := \pi(\Lambda')$. By Lemma \ref{DifferenceSetInclusion} we have $(\Lambda')^{-1}\Lambda' \subset \Lambda^{-1}\Lambda$ and hence $(\Delta')^{-1}\Delta' \subset \Delta^{-1}\Delta$. In particular, for all distinct $\delta_1', \delta_2' \in \Delta'$ there exist (necessarily distinct) $\delta_1, \delta_2 \in \Delta$ such that $(\delta_1')^{-1}\delta_2' = \delta_1^{-1}\delta_2$ and hence
\[
d(\delta_1', \delta_2') = d((\delta_1')^{-1}\delta_2', e) = d(\delta_1^{-1}\delta_2, e) = d(\delta_1, \delta_2) > r.
\]
This shows that $\Delta'$ is $r$-uniformly discrete.
\end{proof}
In the proof of Theorem \ref{BombieriTaylor1} we will apply Proposition \ref{ProjectionUUDiscrete} in the following form.
\begin{corollary}\label{BomTayU} In the situation of Theorem \ref{BombieriTaylor1} there exists an open symmetric identity neighbourhood $U \subset Q$ with the following property: If $\Lambda' \in \Omega_\Lambda$ and $\Delta' := \pi(\Lambda')$, then 
\[
(\Delta')^{-1}\Delta' \cap U^2  = \{e\}.
\] 
\end{corollary}
\begin{proof} By Proposition \ref{ProjectionUUDiscrete} there exists a uniform constant $r>0$ such that all $\Delta' = \pi(\Lambda')$ are $r$-uniformly discrete. This implies that $(\Delta')^{-1}\Delta' \cap B_{r/2}(e) = \{e\}$, and hence we may choose $U := B_{r/4}$.
\end{proof}

\subsection{Proof of Theorem \ref{BombieriTaylor1}}\label{BTProof}
We now return to the setting of Theorem \ref{BombieriTaylor1}. We will reduce the proof to the following version of the pointwise ergodic theorem.
\begin{lemma}[Pointwise ergodic theorem]\label{LemmaPET} For every $h \in C(\Omega_\Lambda)$ and $\xi \in {\rm Hom}(N, \mathbb T)$ there exists a $\nu_\Lambda$-conull subset $\Omega_\xi(h) \subset \Omega_{\Lambda}$ such that the limit
\begin{equation}\label{PET}
\pi_\xi(h)(\Lambda') := \lim_{T \to \infty} \frac{1}{m_N(F_T)}\int_{F_T}\overline{\xi(n)}h(n^{-1}\Lambda')dm_N(n)
\end{equation}
exists for all $\Lambda' \in \Omega_\xi(h)$.
\end{lemma}
\begin{remark} By a suitable Wiener-Wintner theorem one can get rid of the dependence of $\Omega_\xi(h)$ on $\xi$; since the above lemma is sufficient for our purposes, we will not pursue this here.
\end{remark}
To apply the pointwise ergodic theorem, we need a way to produce continuous functions on the hull. We recall from \cite{BH1} that there is a well-defined \emph{periodization map}
\[
\mathcal P: C_c(G) \to C(\Omega_\Lambda), \quad \mathcal Pf(\Lambda') := \sum_{x \in \Lambda'} f(x).
\]
In fact, such a periodiziation map exists whenever $\Lambda \subset G$ is a Delone subset of finite local complexity. For the proof of Theorem \ref{BombieriTaylor1} we will apply the pointwise ergodic theorem to periodizations of certain functions on $G$ which we now construct.
\begin{remark}[A family of functions on $Q$]\label{KappaConstruction}
Using Corollary \ref{BomTayU} we choose a symmetric identity neighbourhood $U$ in $Q$ such that $U^2 \cap (\Delta')^{-1}\Delta' = \{e\}$ for all $\Delta' \in \pi(\Omega_\Lambda)$. We also choose a smaller identity neighbourhood $V$ such that $\overline{V} \subset U$. Since $Q$ is second countable we then find a countable set $P \subset Q$ such that $PV = Q$. The sets $U$, $V$ and $P$ will be fixed from now on. We also fix a function 
$\kappa \in C_c(Q)$ such that $\supp{\kappa} \subset U$ and $\kappa|_V = 1$. 

Given $q_1, q_2 \in Q$ we set $\kappa_{q_1}(q_2) := \kappa(q_1^{-1}q_2)$. Since $PV = Q$ we then find for every $q \in Q$ an element $p_q \in Q$ such that $\kappa_{p_q}(q) = 1$. Since ${\rm supp}(\kappa) \subset U$ we have $p_q^{-1} \in Uq^{-1}$.
Note that if  $\Lambda' \in \Omega_\Lambda$ and $\Delta' := \pi(\Lambda')$, then for all $\delta_1,\delta_2 \in \Delta'$ we have \[\kappa_{p_{\delta_1}}(\delta_1) = 1 \qand \kappa_{p_{\delta_1}}(\delta_2) = 0 \text{ if }\delta_1 \neq \delta_2.\] Indeed, if $\kappa_{p_{\delta_1}}(\delta_2) \neq 0$, then $p_{\delta_1}^{-1}\delta_2 \in U$, hence $p_{\delta_1}^{-1} \in U\delta_2^{-1} \cap U \delta_1^{-1}$. This implies $\delta_2^{-1}\delta_1 \in UU^{-1} \cap (\Delta')^{-1}\Delta' = \{e\}$, and hence $\delta_1 = \delta_2$ and thus $\kappa_{p_{\delta_1}}(\delta_2)=1$ by the choice of $p_{\delta_1}$. We deduce that if $\delta' \in \Delta'$, then for any function $\theta: Q \to \C$ we have
\begin{equation}\label{KappaFunction}
\sum_{\delta'' \in \Delta'} \kappa_{p_{\delta'}}(\delta'')\theta(\delta'') = \theta(\delta').
\end{equation}
\end{remark}
We can now prove Theorem \ref{BombieriTaylor1} in the untwisted case $\xi = 1$. We will need the following lemma, that follows from standard convolution estimates and will be proved in Subsection \ref{SubsecConvolution} below. Here, given $\psi\in C_c(N)$ and a locally finite subset $\Xi \subset N$, we abbreviate 
\[\mathcal Q \psi(\Xi) := \sum_{t \in \Xi} \psi(t).\]
\begin{lemma}\label{Convolution1} Let  $K \subset N$ be a compact identity neighbourhood and $\rho \in C_c(N)$ be non-negative with $\|\rho\|_1 = 1$ and ${\rm supp}(\rho) \subset K$. Then for every locally finite subset $\Xi \subset N$ we have
\begin{equation}\label{ineq2}
\big| \Xi \cap F_T \big| \leq \int_{F_{T+T_K}} (\mathcal Q\rho)(n^{-1}.\Xi) \, dm_N(n) \leq \big| \Xi \cap F_{T+2T_K} \big|.
\end{equation}
\end{lemma}
\begin{proof}[Proof of Theorem \ref{BombieriTaylor1} for $\xi = 1$] Let $U,V, P, \kappa$ as in Remark \ref{KappaConstruction}, let $K \subset N$ be a compact identity neighbourhood and $\rho \in C_c(N)$ be non-negative with $\|\rho\|_1 = 1$ and ${\rm supp}(\rho) \subset K$.  Then for every $p \in P$ we have $\rho \otimes \mathcal \kappa_p \in C_c(G)$ and thus $\mathcal P(\rho \otimes \mathcal \kappa_p) \in C(\Omega_\Lambda)$.  For every $p \in P$ Lemma \ref{LemmaPET} now yields a corresponding conull subset $\Omega_1(\mathcal P(\rho \otimes \kappa_p)) \subset \Omega_\Lambda$. Since $P$ is countable the subset
\[
\Omega_1 := \bigcap_{p \in P} \Omega_1(\mathcal P(\rho \otimes \kappa_p))
\]
is conull in $\Omega_{\Lambda}$. By definition, for all $\Lambda' \in \Omega_1$ and $\delta' \in \pi(\Lambda')$ the limit
\[
I_1(\Lambda',\delta'): = \lim_{T \ra \infty} \frac{1}{m_N(F_T)} \int_{F_{T}} \mathcal P(\rho \otimes \kappa_{p_{\delta'}})(n^{-1}.\Lambda') \, dm_N(n) 
\]
exists. By \eqref{Nice1} we can choose $T_K>0$ such that for all $T >0$ we have $F_T \subset F_TK^{-1} \subset F_{T + T_K}$, and in view of exact volume growth of the sequence $(F_T)$ and the definition of $\mathcal P$ we can then rewrite the limit $I_1(\Lambda',\delta')$ as
\[
I_1(\Lambda',\delta') = \lim_{T \ra \infty} \frac{1}{m_N(F_T)} \int_{F_{T+T_K}} \mathcal P(\rho \otimes \kappa_{p_{\delta'}})(n^{-1}.\Lambda') \, dm_N(n).
\]
We are going to show that $D_1(\Lambda', \delta')$ is equal to this limit. From now on we fix $\Lambda' \in \Omega_1$ and $\delta' \in \Delta' := \pi(\Lambda')$. We first observe that by \eqref{KappaFunction},
\[
D_1(\Lambda', \delta', T) := \frac{|\Lambda'_{\delta'}\cap F_T|}{|m_N(F_T)|}  = \frac{1}{|m_N(F_T)|} \sum_{\delta'' \in \Delta'} \kappa_{p_{\delta'}}(\delta'') |\Lambda'_{\delta''}\cap F_T|.
\]
By \eqref{ineq2} we thus deduce that
\begin{eqnarray*}
 D_1(\Lambda', \delta', T)& = &\frac{1}{|m_N(F_T)|} \sum_{\delta'' \in \Delta'} \kappa_{p_{\delta'}}(\delta'') |\Lambda'_{\delta''}\cap F_T|\\
 &\leq&  \frac{1}{|m_N(F_T)|} \sum_{\delta'' \in \Delta'} \kappa_{p_{\delta'}}(\delta'') \int_{F_{T+T_K}} (\mathcal Q\rho)(n^{-1}.\Lambda'_{\delta''}) \, dm_N(n)\\
 &\leq& \frac{1}{|m_N(F_T)|} \sum_{\delta'' \in \Delta'} \kappa_{p_{\delta'}}(\delta'') |\Lambda'_{\delta''}\cap F_{T+2T_K}|\\
 &=& \frac{|m_N(F_{T+2T_K})|}{|m_N(F_T)|} \cdot D_1(\Lambda', \delta', T+2T_K).
\end{eqnarray*}
Now observe that, unravelling definitions, we have 
\begin{equation}\label{PvsQ}
 \sum_{\delta'' \in \Delta'} \kappa_{p_{\delta'}}(\delta'') (\mathcal Q\rho)(n^{-1}.\Lambda'_{\delta''}) =  \mathcal P(\rho \otimes \kappa_{p_{\delta'}})(n^{-1}.\Lambda'),
\end{equation}
and hence
\[
 D_1(\Lambda', \delta', T) \leq  \frac{1}{m_N(F_T)} \int_{F_{T+T_K}} \mathcal P(\rho \otimes \kappa_{p_{\delta'}})(n^{-1}.\Lambda') \, dm_N(n)\leq  \frac{|m_N(F_{T+2T_K})|}{|m_N(F_T)|} \cdot D_1(\Lambda', \delta', T+2T_K).
\]
We deduce that
\[
\varlimsup D_1(\Lambda', \delta', T) \leq I_1(\Lambda', \delta') \leq \lim_{T \to \infty } \frac{|m_N(F_{T+2T_K})|}{|m_N(F_T)|}  \varliminf D_1(\Lambda', \delta', T+2T_K) = \varliminf D_1(\Lambda', \delta', T),
\]
which proves that $ D_1(\Lambda', \delta', T)$ converges to $I_1(\Lambda', \delta')$.
\end{proof}
The case of twisted fiber densities can be reduced to the case $\xi=1$, using the following variant of Lemma \ref{Convolution1}; here given $\eta, \psi \in L^1(N, m_N)$ we denote their convolution by
\[
(\eta \ast \psi)(t) := \int_N \eta(n) \psi(n^{-1}t)dm_N(n) = \int_N \eta(n^{-1}) \psi(nt) dm_N(n).
\]
\begin{lemma}\label{Convolution2} Let $K \subset N$ be a compact identity neighbourhood and $\rho \in C_c(N)$ with ${\rm supp}(\rho) \subset K$. Then for every locally finite subset $\Xi \subset N$ and $\eta \in L^\infty(N)$ we have
\begin{equation}
\label{ineq3}
\left| \sum_{t \in \Xi \cap F_T} (\eta * \rho)(t) - \int_{F_{T+T_K}} \eta(n) \, (\mathcal Q\rho)(n^{-1}.\Xi) \, dm_N(n) \right| \quad \leq \quad \left|(F_{T+2T_K} \setminus F_T) \cap \Xi\right| \cdot \|\rho\|_1 \cdot \|\eta\|_\infty.
\end{equation}
\end{lemma}
\begin{proof}[Proof of Theorem \ref{BombieriTaylor1}] Fix $\xi \in {\rm Hom}(N, \mathbb T)$. Let $\rho \in C_c(N)$ non-negative with $\|\rho\|_1= 1$ and $\widehat{\rho}(\bar{\xi}) \neq 0$. Let $K$ be a compact identity neighbourhood containing  ${\rm supp}(\rho)$ and choose $T_K>0$ such that for all $T >0$ we have $F_T \subset F_TK^{-1} \subset F_{T + T_K}$. As in the untwisted case we then define a conull subset of $\Omega_\Lambda$
\[
\Omega_{\xi} := \bigcap_{p \in P} \Omega_{\xi}(\mathcal P(\rho \otimes \kappa_p)), 
\]
where $P$ and $\kappa$ are as before, and fix $\Lambda' \in \Omega_\xi$ and$\delta' \in \Delta' := \pi(\Lambda')$. We are going to show that the approximants
$D_\xi(\Lambda', \delta', T) := \frac{1}{|m_N(F_T)|} \sum_{t \in |\Lambda'_{\delta'}\cap F_T|} \overline{\xi(t)}$ converge to $\widehat{\rho}(\bar \xi)^{-1} \cdot I_{\xi}(\mathcal P(\rho \otimes \kappa_{\delta'}))(\Lambda')$.

Note that by unimodularity of $N$ we have
\[
(\bar{\xi} \ast \rho)(t) = \int_{N}\xi(n)\rho(nt)dm_N(n) =  \int_{N}\xi(nt^{-1})\rho(n)dm_N(n) = \bar{\xi}(t) \cdot \widehat{\rho}(\bar{\xi}),
\]
and thus 
\[
\sum_{t \in \Xi \cap F_T} (\bar \xi * \rho)(t) = \widehat{\rho}(\bar{\xi}) \cdot \sum_{t \in \Xi \cap F_T} \overline{\xi(t)}.
\]
Applying \eqref{ineq3} wtih $\Xi :=\Lambda'_{\delta'}$ and $\eta := \bar{\xi}$ thus yields
\begin{equation}\label{BTMainEstimate2}
\left|  \widehat{\rho}(\bar \xi)  \cdot \left(\sum_{t \in \Lambda'_{\delta'} \cap F_T}  \overline{\xi(t)}\right) \,
- 
\int_{F_{T+T_K}} \overline{\xi(n)} \, Q(\rho)(n^{-1}.\Lambda'_{\delta'})\, dm_N(n) \right| 
\quad \leq \quad 
\left| F_{T+2T_K} \cap \Lambda'_{\delta'}\right|
-
\left| F_{T} \cap \Lambda'_{\delta'}\right|
\end{equation}
Now multiply both sides by $m_N(F_T)^{-1}= \kappa_{p_K\delta'}(\delta') m_N(F_T)^{-1}$; using \eqref{KappaFunction} and \eqref{PvsQ} the left-hand side of \eqref{BTMainEstimate2} becomes
\[
\left|  \widehat{\rho}(\bar \xi)  \cdot D_\xi(\Lambda', \delta', T) \,
- 
\frac{m_N(F_{T+T_K})}{m_N(F_T)} \frac{1}{m_N(F_{T+T_K})}\int_{F_{T+T_K}} \overline{\xi(n)} \, P(\rho \otimes \kappa_{p_{\delta'}})(n^{-1}.\Lambda') \, dm_N(n)\right|,
\]
and its second term converges to $I_{\xi}(\mathcal P(\rho \otimes \kappa_{\delta'}))(\Lambda')$ by the pointwise ergodic theorem and exact volume growth. On the other hand, multiplying the right hand side of \eqref{BTMainEstimate2} by $m_N(F_T)^{-1}$ yields
\[
\left|\frac{m_N(F_{T+2T_K})}{m_N(F_T)}  \cdot D_1(\Lambda', \delta', T+2T_K)- D_1(\Lambda', \delta', T)\right|
\]
and since $D_1(\Lambda', \delta', T)$ converges we deduce that this expression converges to $0$. This shows that
\[
 \widehat{\rho}(\bar \xi)  \cdot D_\xi(\Lambda', \delta', T) \quad \longrightarrow \quad I_{\xi}(\mathcal P(\rho \otimes \kappa_{\delta'}))(\Lambda'),
\]
and finishes the proof.
\end{proof}
At this point we have finished the proof of Theorem \ref{BombieriTaylor1} modulo Lemmas \ref{Convolution1} and \ref{Convolution2} which will be established in the next subsection.
\subsection{Convolution estimates used in the proof of Theorem \ref{BombieriTaylor1}}\label{SubsecConvolution}
The purpose of this subsection is to establish Lemmas \ref{Convolution1} and \ref{Convolution2}. We will work in the following more general setting. Throughout this subsection let $N$ be a unimodular lcsc group with Haar measure $m_N$. We denote by $C_c(N)^+$ the set of non-negative compactly supported continuous functions on $N$. Then the following general lemma implies Lemma \ref{Convolution1}.
\begin{lemma}\label{BB+1} Let $B, B^+$ be Borel sets in $N$ and $K \subset N$ compact such that $B \subset BK^{-1} \subset B^+$. Then for every $\rho \in C_c(N)^+$ with $\|\rho\|_1 = 1$ and ${\rm supp}(\rho) \subset K$ we have
\[
\chi_B \leq \chi_{B^+} \ast \rho \leq \chi_{B^+K}.
\]
\end{lemma}
\begin{proof} Since  $BK^{-1} \subset B^+$ we have
\[
\bigcap_{t \in B}(B^+)^{-1} t \supset \bigcap_{t \in B} KB^{-1}t \supset K \supset {\rm supp}(\rho).
\]
Since $N$ is unimodular, we obtain for all $t \in B$
\[
(\chi_{B^+} \ast \rho)(t) =  \int_{B^+} \rho(n^{-1}t)dm_N(n) =  \int_{(B^+)^{-1} t }\rho(x)dm_N(x) \geq  \int_{{\rm supp}(\rho)}\rho(x)dm_N(x) = 1,
\]
which shows that $\chi_{B^+} \ast \rho \geq \chi_B$. Conversely, if $(\chi_{B^+} \ast \rho)(t) \neq 0$, then there exists $n \in B^+$ with $k := n^{-1}t \in {\rm supp}(\rho) \subset K$ and hence $t = nk \in B^+K$. Since $(\chi_{B^+} \ast \rho)(t) \leq \|\rho\|_1 =  1$, the lemma follows.
\end{proof}
\begin{proof}[Proof of Lemma \ref{Convolution1}] Since $F_T \subset F_TK^{-1} \subset F_{T+T_K}$ and $F_{T+T_K}K \subset F_{T+2T_K}$ we can apply Lemma \ref{BB+1} with $B := F_T$ and $B^+ := F_{T+T_K}$ to obtain
\[
\chi_{F_T} \leq \chi_{F_{T+T_K}} \ast \rho \leq \chi_{F_{T+T_K}K} \leq \chi_{F_{T+2T_K}}.
\]
Now integrate this inequality against the Dirac comb $\delta_\Xi := \sum_{\xi \in \Xi} \delta_\xi$ and observe that
\[
\int_N (\chi_{F_{T+T_K}} \ast \rho) \; d\delta_\Xi =  \sum_{t \in \Xi} \int_N \chi_{F_{T+T_K}}(n) \rho(n^{-1}t) dm_N(n) = \int_{F_{T+T_K}} (\mathcal Q\rho)(n^{-1}.\Xi) \, dm_N(n),
\]
whereas $\int \chi_{F_T} d\delta_\Xi = |\Xi \cap F_T|$ and  $\int \chi_{F_{T+2T_K}} d\delta_\Xi = |\Xi \cap F_{T+2T_K}|$.
\end{proof}
Similarly, the following general lemma implies Lemma \ref{Convolution2}.
\begin{lemma}\label{BB+2} Let $B, B^+$ be Borel sets in $N$ and $K \subset N$ compact such that $B \subset BK^{-1} \subset B^+$. Then for every $\eta \in L^\infty(N, m_N)$ and $\rho \in C_c(K)$ and all $t \in B$ we have
\[
(\eta \ast \rho)(t) = \int_{B^+} \eta(n)\rho(n^{-1}t) dm_N(n).
\]
\end{lemma}
\begin{proof} If $t \in B$ and $\rho(n^{-1}t) \neq 0$, then $n^{-1}t \in K$ and thus $n \in BK^{-1} \subset B^+$. The lemma follows.
\end{proof}
\begin{proof}[Proof of Lemma \ref{Convolution2}] Apply Lemma \ref{BB+2} with $B := F_T$ and $B^+ := F_{T+T_K}$ and use that if $n \in F_{T+T_K}$ and $n^{-1}t\in {\rm supp}(\rho)$, then $t \in F_{T+T_K}K \subset F_{T+2T_K}$ to conclude that
\[
 \int_{F_{T+T_K}} \eta(n) \, (\mathcal Q\rho)(n^{-1}.\Xi) \, dm_N(n) = \sum_{t \in \Xi}\int_{F_{T+T_K}}\eta(n)\rho(n^{-1}t)dm_N(n) = \sum_{t\in \Xi \cap F_{T+2T_K}} (\eta * \rho)(t).
\]
The left-hand side of \eqref{ineq3} is thus given by
\[
\left| \sum_{t \in (\Xi \cap F_{T+2T_K}) \setminus (\Xi \cap F_T)} (\eta * \rho)(t) \right| \leq  \left|(F_{T+2T_K} \setminus F_T) \cap \Xi\right| \|\eta\ast \rho\|_\infty \leq \left|(F_{T+2T_K} \setminus F_T) \cap \Xi\right| \cdot \|\rho\|_1 \|\eta\|_\infty,
\]
which finishes the proof.
\end{proof}

\section{Existence of horizontal factors}\label{SecHorizontal}

\subsection{The horizontal factor theorem}
Throughout this section we consider \emph{topologically split} extensions $(G, Z, Q, \pi, s)$ of lcsc groups. We will be mostly interested in the case where $G$ is a $1$-connected nilpotent Lie group and $Z$ is a closed connected central subgroup of $G$. In this case, $Q$ is automatically contracible by Remark \ref{ClosedSubgpNilpotentGp}, hence the existence of a continuous section $s$ is automatic.

Now let $\Lambda \subset G$ be a $\pi$-aligned uniform approximate lattice, so that in particular $\Delta := \pi(\Lambda)$ is a uniform approximate lattice in $Q$. If $\Lambda$ happens to be a \emph{subgroup} of $G$ (i.e.\ a uniform lattice), then $\pi$ induces a continuous $G$-equivariant surjection 
\[
\pi_*: \Omega_\Lambda  = G/\Lambda \to Q/\Delta = \Omega_\Delta, \quad \Lambda' \mapsto \pi(\Lambda')
\]
between the respective hulls. One might expect naively, that a similar result also holds for general uniform approximate lattices but, somewhat surprisingly, this is \emph{not} the case in complete generality.

To explain the problem that arises when one tries to extend the projection $\pi$ to the hull of $\Lambda$, we recall that $\Lambda$ is said to have $R$-large fibers if the subset 
\[
\Delta^{(R)} := \{\delta \in \Delta\mid \Lambda_\delta \text{ is }R\text{-relatively dense in }N\}.
\]
is relatively dense in $Q$. If moreover $\Delta = \Delta^{(R)}$, then we say that $\Lambda$ has \emph{uniformly $R$-large fibers}. The same notion also applies to Meyerian sets which are not symmetric. 

If a Meyerian set $\Lambda_o$ has uniformly $R$-large fibers, then so does the associated uniform approximate lattice $\Lambda := \Lambda_o^{-1}\Lambda_o$, hence by Theorem \ref{FiberVsImage} $\pi(\Lambda)$ is uniformly discrete. This then implies that also $\pi(\Lambda_o)$ is uniformly discrete, hence Meyerian.  Thus if a Meyerian set has uniformly $R$-large fibers, then it is $\pi$-aligend and its image under $\pi$ is also Meyerian. 

We can now formulate the main result of this section:
\begin{theorem}[Horizontal factor theorem]\label{HorizontalFactorTheorem} Let $(G, Z, Q, \pi, s)$ be a topologically split central extension of lcsc groups, let $\Lambda \subset G$ be Meyerian and $\Delta := \pi(\Lambda)$.
\begin{enumerate}[(i)]
\item If $\Lambda$ has uniformly large fibers, then there exists a unique $G$-equivariant surjection $\pi_*: \Omega_{\Lambda} \to \Omega_{\Delta}$ which maps $\Lambda \to \Delta$, and this map is given by $\Lambda' \mapsto \pi(\Lambda')$. 
\item Conversely, if $\Lambda$ is uniformly discrete and the map $\Omega_\Lambda \to \mathcal C(Q)$ given by $\Lambda' \mapsto \pi(\Lambda')$ is continuous, then $\Lambda$ has uniformly large fibers.
\end{enumerate}
\end{theorem}
The proof of Theorem \ref{HorizontalFactorTheorem} will be given in the next subsection. In the language of dynamical systems, the theorem says that if $\Lambda$ has uniformly large fibers, then $\Omega_\Delta$ is a continuous $G$-factor of $\Omega_\Lambda$. It is called the \emph{horizontal factor} of $\Omega_\Lambda$ with respect to $\pi$.

In view of the theorem, if we want to relate the dynamical systems of a Meyerian set to the dynamical system of its projection in a continuous way, then we have no choice but to assume uniformly large fibers. It is thus important for us to investigate, how restrictive this condition is. Recall first that, by Remark \ref{UniformlyLargeFibersSplitCase}, every split uniform approximate lattice
has uniformly large fibers. 

In general, a $\pi$-aligned uniform approximate lattice need not have uniformly large fibers, and even if it does, this property can easily be destroyed. Indeed, if $\Lambda$ has uniformly large fibers and if we define $\Lambda'$ as the uniform approximate lattice obtained from $\Lambda$ by replacing one fiber by a singleton, then $\Lambda'$ no longer has uniformly large fibers. Note, however, that $\Lambda$ and $\Lambda'$ are commensurable in this case. 

\begin{proposition}[Enforcing uniformly large fibers]\label{FiberEnforce} Let $(G, Z, Q, \pi, s)$ be a topologically split central extension and let  $\Lambda$ be a $\pi$-aligned uniform approximate lattice in $G$. Then there exists a uniform approximate lattice $\Lambda' \subset G$ which is $\pi$-aligned, has uniformly large fibers and is commensurable with $\Lambda$ and contained in $\Lambda^2$.
\end{proposition}
The proof is based on the following observation.
\begin{lemma}\label{MakeUniform} Let $(G, Z, Q, \pi, s)$ be a topologically split central extension, let $\Lambda$ be a $\pi$-aligned uniform approximate lattice in $G$ and $\Delta := \pi(\Lambda)$. If for some $R>0$ the subset
\[
\Delta^{(R)} := \{\delta \in \Delta \mid \Lambda_\delta \text{ $R$-relatively dense in }Z\} \subset \Delta
\]
is relatively dense in $Q$, then
\[
\Lambda^{(R)} := \bigcup_{\delta \in \Delta^{(R)}} \Lambda_\delta \times \{\delta\} \quad  \subset \quad \Lambda.
\]
is a $\pi$-aligned uniform approximate lattice with uniformly $R$-large fibers.
\end{lemma}
\begin{proof} Since $\Lambda$ is symmetric and $\Lambda_e$ is $R$-relatively dense, we see that $\Lambda^{(R)}$ is symmetric and contains the identity. Moreover, 
for all $k \in \mathbb N$ we have $(\Lambda^{(R)})^k \subset \Lambda^k$, which shows that  $(\Lambda^{(R)})^k$ is uniformly discrete. By Characterization (m2) of Lemma \ref{mProperties} it thus remains to show only that $\Lambda^{(R)}$ is Delone. For this we first observe that if we denote by $K$ the closed ball of radius $R$ in $Z$, then for all $\delta \in \Delta^{(R)}$ we have $\Lambda^{(R)}_\delta + K = \Lambda_\delta + K = Z$. On the other hand, there exists $L \subset Q$ such that $\Delta^{(R)}L = Q$. Define $M := K \times L \subset G$; we claim that $\Lambda^{(R)}M = G$. Indeed, let $(z, q) \in G$. Since $Q = \Delta^{(R)}L$ there exist $\delta \in \Delta^{(R)}$ and $l \in L$ such that $q = \delta l$. Since $Z = \Lambda^{(R)}_\delta + K$ there then exists $t \in \Lambda^{(R)}_\delta$ and $k \in K$ such that $t + k + \beta(\delta, l) = z$. This shows that
\[
(z, q) = (t+ k + \beta(\delta, l), \delta l) = (t,\delta)(k, l) \in \Lambda^{(R)}M,
\]
which finishes the proof.
\end{proof}
\begin{proof}[Proof of Proposition \ref{FiberEnforce}] If $\Lambda$ be an arbitrary $\pi$-aligned uniform approximate lattice, then we deduce from Corollary \ref{CorSquares} that $\Lambda^2$ is still $\pi$-aligned and that there exists $R>0$ such that $\pi(\Lambda^2)^{(R)}$ is relatively dense and contains the identity.  It thus follows from Lemma \ref{MakeUniform} that $\Lambda' := (\Lambda^{2})^{(R)}$ has $R$-uniformly large fibers, and $\pi(\Lambda') \subset \pi(\Lambda^2)$ is uniformly discrete, hence $\Lambda'$ is $\pi$-aligned. Since $\Lambda$ and $\Lambda'$ are both relatively dense subsets of $\Lambda^2$, they are commensurable.
\end{proof}
%\begin{corollary}[Generalized horizontal factor theorem] Let $\Lambda \subset G$ be a $\pi$-aligned uniform approximate lattice. Then there exists a uniform approximate lattice $\Lambda' \subset G$ with the following properties:
%\begin{enumerate}[(i)]
%\item $\Lambda'$ is commensurate with $\Lambda$ and $\pi$-aligned; in particular $\Delta' := \pi(\Lambda')$ is a uniform approximate lattice in $Q$.
%\item There exists a unique $G$-equivariant surjection $\pi_*: \Omega_{\Lambda'} \to \Omega_{\Delta'}$ which maps $\Lambda' \to \Delta'$, and this map is given by $\Lambda'' \mapsto \pi(\Lambda'')$. 
%\end{enumerate}
%\end{corollary}
\subsection{Proof of the horizontal factor theorem}
The proof of the horizontal factor theorem makes use of the following general observation. Here $(G, Z, Q, \pi, s)$ is an arbitrary central extension of lcsc groups.
\begin{proposition}[Fiberwise convergence]\label{FiberwiseConvergenceConvenient}  Let $\Lambda \subset G$ be an FLC Delone set whose projection $\Delta := \pi(\Lambda)$ is uniformly discrete and assume that $\Lambda_n \to \Lambda'$ in $\Omega_{\Lambda}$. Set $\Delta_n := \pi(\Lambda_n)$ and $\Delta' := \pi(\Lambda')$. Then there exists $\epsilon > 0$ such that for every for every $\delta' \in \Delta'$ there exist $n_0(\delta') \in \mathbb N$ such that for all $n \geq n_0(\delta')$ there exist $\delta_n \in \Delta_n$ with
\[
\Delta_n \cap B_{\epsilon}(\delta') = \{\delta_n\} \qand (\Lambda_n)_{\delta_n} \to \Lambda'_{\delta'} \text{ in }\mathcal C(Z).
\]
\end{proposition}
In view of Proposition \ref{ProjectionUUDiscrete}, Proposition \ref{FiberwiseConvergenceConvenient} is a special case of the following lemma:
\begin{lemma} Let $\Omega \subset \mathcal C(G)$ be a closed subset and assume that there exists $r>0$ such that for every $\Lambda \in \Omega$ the projection $\Delta := \pi(\Lambda)$ is $r$-uniformly discrete. Assume that $\Lambda_n \to \Lambda$ in $\Omega$ and set $\Delta_n := \pi(\Lambda_n)$ and $\Delta := \pi(\Lambda)$. Then for every $\delta \in \Delta$ there exists $n_0(\delta) \in \mathbb N$ such that for all $n \geq n_0(\delta)$ there exist $\delta_n \in \Delta_n$ with
\[
\Delta_n \cap B_{r/4}(\delta) = \{\delta_n\} \qand (\Lambda_n)_{\delta_n} \to \Lambda_\delta \text{ in }\mathcal C(Z).
\]
\end{lemma}
\begin{proof} Let $\delta \in \Delta$ and choose $\lambda \in \Lambda_\delta \times \{\delta\}$. By Lemma \ref{CFConvergence}.(ii) there exist $\lambda_n \in \Lambda_n$ with $\lambda_n \to \lambda$. If we set $\delta_n := \pi(\lambda_n)$, then $\delta_n \in \Delta_n$ and $\delta_n \to \delta$. In particular there exists $n_0(\delta)$ such that for all $n \geq n_0(\delta)$ we have $d(\delta_n, \delta) < r/4$. We claim that for $n \geq n_0(q)$ we have $\Delta_n \cap B_{r/4}(\delta) = \{\delta_n\}$. Indeed, if $\delta_n' \in \Delta_n \cap B_{r/4}(q)$ then $d(\delta_n, \delta_n')\leq r/2$, and hence $\delta_n' = \delta_n$, since $\Delta_n$ is $r$-uniformly discrete by assumption. It remains to show only that $ (\Lambda_n)_{\delta_n} \to \Lambda_\delta$. For this we have to check Conditions (i) and (ii) from Lemma \ref{CFConvergence}: As for (i), assume that $t_{n_k} \in (\Lambda_{n_k})_{\delta_{n_k}}$ converges to some $ t\in Z$. Then $(t_{n_k}, \delta_{n_k}) \in \Lambda_{n_k}$ and $(t_{n_k}, \delta_{n_k}) \to (t, \delta)$, and hence we deduce that $(
t, \delta) \in \Lambda$. This shows that $t \in \Lambda_\delta$ and establishes (i). As for (ii), given $t \in \Lambda_\delta$ we have $(t, \delta) \in \Lambda$, hence there exist $\lambda_n = (t_n, \delta_n) \in \Lambda_n$ with $\lambda_n \to \lambda$. Then $t_n \in (\Lambda_n)_{\delta_n}$ and $t_n \to t$; this establishes (ii) and finishes the proof.
\end{proof}
We can now prove the horizontal factor theorem.
\begin{proof}[Proof of Theorem \ref{HorizontalFactorTheorem}] (i) Uniqueness is clear from the fact that the $G$-orbit of $\Lambda$ is dense in $\Omega_{\Lambda}$, hence any continuous $G$-equivariant map is uniquely determined by the image of $\Lambda$. As for existence, let $\mathcal G := \overline{N.(\Lambda, \Delta)}$. We claim that if $(\Lambda', \Delta') \in \mathcal G$, then $\Delta' = \pi(\Lambda')$. Assuming the claim for the moment, let us prove the proposition. It follows from the claim that $\mathcal G$ is the graph of the map $\Omega_{\Lambda} \to \mathcal C(Q)$ which sends $\Lambda$ to $\pi(\Lambda)$. Since $\mathcal G$ is closed, the closed graph theorem implies that this map is continuous, and since $\mathcal G$ is $N$-invariant, this map is $N$-equivariant. Since $\pi(\Lambda) = \Delta$ and the $G$-orbit of $\Lambda$ is dense in $\Omega_{\Lambda}$, the image is given by $\Omega_{\Delta}$. This finishes the proof, and it thus remains to establish the claim.

Thus let $(\Lambda', \Delta') \in \mathcal G$. We have to show that $\Delta' = \pi(\Lambda')$. By definition of $\mathcal G$ we find elements $m_k = (z_k, q_k) \in G$ such that
\[
m_k. \Lambda \to \Lambda' \qand q_k\Delta \to \Delta'.
\]
If $\lambda' \in \Lambda'$ then by Lemma \ref{CFConvergence}.(ii) there exist $\lambda_k = (t_k, \delta_k) \in \Lambda$ such that 
\[
m_k.\lambda_k \to \lambda', \quad \text{and hence} \quad \delta_kq_k = \pi(m_k.\lambda_k) \ \to \pi(\lambda'). 
\]
It thus follows from Lemma \ref{CFConvergence}.(i) that $\pi(\lambda') \in \Delta'$, and hence $\pi(\Lambda') \subset \Delta'$.

To establish the converse inclusion we fix $\delta' \in \Delta'$. By Lemma \ref{CFConvergence}.(ii) there exists elements $\delta_k \in \Delta$ such that $q_k\delta_k \to \delta'$. By assumption there exists a universal constant $R>0$ such that each fiber $\Lambda_{\delta_k}$ is $R$-relatively dense in $Z$ with respect to some left-admissible metric $d_Z$ on $Z$. This means that there exist elements $t_k \in \Lambda_{\delta_k}$ such that 
\begin{equation}\label{HorizontalFactorMainEstimate}
d_Z(z_k + \beta(q_k, \delta_k), -t_k) = d_Z(z_k + t_k + \beta(q_k, \delta_k), e) \leq R.
\end{equation}
Since $t_k \in \Lambda_{\delta_k}$ we have $(t_k, \delta_k) \in \Lambda$ and hence $m_k.(t_k, \delta_k) \in m_k.\Lambda$. Note that, explicitly,
\[
m_k(t_k, \delta_k) =  (z_k, q_k)(t_k, \delta_k) = (z_k + t_k + \beta(q_k, \delta_k), q_k\delta_k).
\]
Now the first coordinate is bounded by \eqref{HorizontalFactorMainEstimate}, and the second coordinate converges to $\delta'$ by assumption. Passing to a subsequence we may thus assume that it converges to some $z \in Z$. By Lemma \ref{CFConvergence}.(i) we deduce that \[(z, \delta') \in \lim_{k \to \infty} m_k.\Lambda = \Lambda'.\]
We deduce that $\delta' \in \pi(\Lambda')$, and since $\delta' \in \Delta'$ was arbitrary this shows that $\Delta' \subset \pi(\Lambda')$ and finishes the proof.

(ii) Assume that $\Lambda$ does not have uniformly large fibers. Then for every $n\in \mathbb N$ we find $\delta_n \in \Delta$ such that $\Lambda_{\delta_n}$ is not $n$-relatively dense. We then find $z_n \in Z$ such that $z_n \Lambda_{\delta_n} \cap B_{n/3}(e) = \emptyset$. Now let $\Lambda^{(n)} := s(\delta_n)^{-1}z_n\Lambda$  and $\Delta^{(n)} := \pi(\Lambda^{(n)}) = \delta_n^{-1}\Delta$. Then
\[
\Lambda^{(n)}_e = \beta(\delta_n^{-1},e)z_n\Lambda_{\delta_n} = z_n\Lambda_{\delta_n}.
\]
We conclude that $\Lambda^{(n)}_e \cap B_{n/3}(e) = \emptyset$ and that $\Lambda^{(n)}_e \neq \emptyset$ and thus $e\in \Delta^{(n)}$ for all $n$. Since $\Delta^{(n)}$ is a translate of $\Delta$ for all $n \in \mathbb N$, there exists $r>0$ such that all $\Delta^{(n)}$ are $r$-uniformly discrete. We conclude that $\Delta^{(n)} \cap B_r(e) = \{e\}$ for all $n \in \mathbb N$.

Since $\Omega_{\Lambda}$ is compact we find a subsequence $n_k$ such that $\Lambda^{(n_k)}$ converges to some $\Lambda' \in \Omega_\Lambda$. By Proposition \ref{FiberwiseConvergenceConvenient} we would then have convergence $\Lambda^{(n_k)}_e \to \Lambda'_e$. Since $\Lambda^{(n_k)}_e \cap B_{n_k/3}(e) = \emptyset$ this implies $\Lambda'_e = \emptyset$ and thus $e \not \in \Delta' := \pi(\Lambda')$.

If the projection was continuous, then we would have convergence $\Delta^{(n_k)} \to \Delta'$, but since $e \in \Delta^{n_k}$ for all $k \in \mathbb N$ and $e \not \in \Delta'$, this is a contradiction with Lemma \ref{CFConvergence}.
\end{proof}

\section{Pure-point spectrum of the hull for central extensions}\label{SecHullSpec}

\subsection{General setting}
The goal of this section is to develop the spectral theory of hulls of Meyerian subsets of nilpotent groups. Throughout this section we consider a central extension $(G, Z, Q, \pi, S)$ of lcsc groups and a $\pi$-aligned uniform approximate lattice $\Lambda \subset G$. 
\begin{convention}\label{AssumptionsSpectral}
Throughout this section we will make the following assumptions concerning $G$ and $\Lambda$:
\begin{itemize}
\item We assume that $Q$ is connected and that $Z$ is $1$-connected. The latter implies in particular that $Z$ admits a nice F\o lner sequence $(F_T)$ in the sense of Definition \ref{Nice}, and we fix such a F\o lner sequence once and for all. For example, we may simply choose Euclidean balls.
\item We assume that the section $s$ is continuous (so that the extension is topologically split). By \cite[Thm.\ 2]{Shtern} this can always be arranged since $Q$ is connected and $Z$ is  $1$-connected. 
\item We assume that $s$ is symmetric (i.e.\ $s(q^{-1}) = s(q)^{-1}$ for all $q \in Q$) and satisfies $s(e) = e$. If $G$ is a $1$-connected nilpotent Lie group, or more generally
if elements of $G$ have a unique continous square root $g^{1/2}$ , then such a section always exists. Indeed, if $s_o$ is an arbitrary continuous section, then we can choose $s(q) := (s_o(q)s_o(q^{-1})^{-1})^{1/2}$.
\item We assume that there exists a $G$-invariant probability measure on $\Omega_\Lambda$. If $G$ is amenable, then this is automatically satisfied since $\Omega_\Lambda$ is compact, but in general it is a non-trivial assumption.
\item We assume that $\Lambda$ has uniformly large fibers. 
\end{itemize}
The assumption on $\Lambda$ can always be achieved by passing to a commensurable uniform approximate lattice (Proposition \ref{FiberEnforce}). The assumptions on $G$ are satisfied if $G$ is a $1$-connected nilpotent Lie group and $Z$ is closed and connected (see Remark \ref{ClosedSubgpNilpotentGp}), which is anyway our main case of interest. A typical example is given by the symplectic product $\Lambda$ from \eqref{HeisenbergLambda} in the $3$-dimensional Heisenberg group $H_3(\R)$.
\end{convention}

From the assumptions in Convention \ref{AssumptionsSpectral} we can draw the following conclusions to be used throughout this section.
\begin{itemize}
\item By \eqref{SymmetricSection}, the cocycle $\beta = \beta_s$ associated with $s$ satisfies $\beta_s(e,q) = \beta_s(q,e) = \beta_s(q, q^{-1}) = 0$ for all $q \in Q$.
\item Since $s$ is continuous, we can identify $G$ as a topological group with $Z \oplus_\beta Q$.
\item Since $\Lambda$ is $\pi$-aligned, $\Delta := \pi(\Lambda)$ is a uniform approximate lattice in $Q$.
\item By the horizontal factor theorem (Theorem \ref{HorizontalFactorTheorem}) we have a continuous $G$-factor map
\[
\pi_*: \Omega_\Lambda \to \Omega_\Delta, \quad \Lambda' \mapsto \pi(\Lambda').
\]
In particular the map $\Lambda' \mapsto \pi(\Lambda')$ is Borel measurable on $\Omega_\Lambda$. 
\end{itemize}

We now fix once and for all a $G$-ergodic probability measure $\nu_{\Lambda}$ on $\Omega_{\Lambda}$ and consider the probability measure preserving dynamical system $G \curvearrowright (\Omega_{\Lambda}, \nu_{\Lambda})$ and the associated Koopman representation $G \to U(L^2(\Omega_{\Lambda}, \nu_{\Lambda}))$.

We introduce the following notations: Given a unitary character $\chi \in \widehat{Z}$ we denote by
\[
 L^2(\Omega_{\Lambda}, \nu_{\Lambda})_\chi := \{f \in L^2(\Omega_{\Lambda}, \nu_{\Lambda}) \mid \forall \Lambda' \in \Omega_{\Lambda}, z\in Z: \; f(z^{-1}\Lambda')= \chi(z)f(\Lambda')\}
\]
the corresponding eigenspace of $ L^2(\Omega_{\Lambda}, \nu_{\Lambda})$ and by $\pi_\chi: L^2(\Omega_{\Lambda}, \nu_{\Lambda}) \to L^2(\Omega_{\Lambda}, \nu_{\Lambda})_\chi$ the orthogonal projection onto this eigenspace. Since $Z$ is central in $G$, the subspace $ L^2(\Omega_{\Lambda}, \nu_{\Lambda})_\chi $ is $G$-invariant and $\pi_\chi$ is a $G$-equivariant linear map. 
\begin{example} If $G = H_3(\R)$ is the $3$-dimensional Heisenberg group, $Z$ its center and $\chi$ a non-trivial character of $Z$, then there is up to unitary equivalence precisely one irreducible unitary $G$-representation $\pi_\chi$ with central character $\chi$, the so-called \emph{Schr\"odinger representation}. In this case the eigenspace $ L^2(\Omega_{\Lambda}, \nu_{\Lambda})_\chi$ is simply the $\pi_\chi$-primary component of $ L^2(\Omega_{\Lambda}, \nu_{\Lambda})$. 
\end{example}
\begin{definition} The \emph{pure point spectrum} of $Z$ in $L^2(\Omega_{\Lambda}, \nu_{\Lambda})$ is the subset of $\widehat{Z}$ given by 
\[
{\rm spec}^Z_{\rm pp}(L^2(\Omega_{\Lambda}, \nu_{\Lambda})) := \{\chi \in \widehat{Z} \mid  L^2(\Omega_{\Lambda}, \nu_{\Lambda})_\chi \neq \{0\}\}
\] 
\end{definition}
We are going to show:
\begin{theorem}[Relatively dense central pure point spectrum]\label{RelativelyDenseSpectrum} Under the standing assumptions from Convention \ref{AssumptionsSpectral} the pure point spectrum ${\rm spec}^Z_{\rm pp}(L^2(\Omega_{\Lambda}, \nu_{\Lambda}))$ of $Z$ in $L^2(\Omega_{\Lambda}, \nu_{\Lambda})$ is relatively dense in $\widehat{Z}$.
\end{theorem}
We will actually establish a more refined result, which can be interpreted as ``relatively density of $(1-\epsilon)$-Bragg peaks for the central diffraction'' for all $\epsilon \in (0,1)$, see Corollary \ref{Bragg}.
\subsection{Properties of twisted fiber densities}
Throughout this subsection we adopt the notation of Theorem \ref{RelativelyDenseSpectrum}. We recall from Theorem \ref{BombieriTaylor1} that for every $\xi \in \widehat{Z}$ there  exists a $G$-invariant $\nu_\Lambda$-conull subset $\Omega_\xi \subset \Omega_\Lambda$ such that for all $\Lambda' \in \Omega_\xi$ and $\delta' \in \pi(\Lambda')$ the limit
\begin{equation}\label{DxiDef}
D_\xi(\Lambda', \delta') = \lim_{T \to \infty} \frac{1}{m_Z(F_T)}\sum_{z \in \Lambda'_{\delta'} \cap F_T} \overline{\xi(z)}
\end{equation}
exists. We will denote
\begin{equation}\label{OmegaTildeXi}
\widetilde{\Omega}_\xi := \{\Lambda' \in \Omega_\Lambda \mid \forall \delta' \in\pi(\Lambda'): \; D_1(\Lambda', \delta') \text{ and }D_\xi(\Lambda', \delta') \text{ exist}\}.
\end{equation}
Since $\widetilde{\Omega}_\xi \supset \Omega_1 \cap \Omega_\xi$, the set $\widetilde{\Omega}_\xi$ is a conull set in $\Omega_\Lambda$.
\begin{lemma}[Equivariance of twisted fiber densities]\label{TheoremTDE} 
The set $\widetilde{\Omega}_\xi$ is $G$-invariant, and for all $g = (z,q) \in G$,  $\Lambda' \in \widetilde{\Omega}_\xi$ and $\delta' \in \pi(\Lambda')$ we have
\begin{equation}\label{TwistedDensityEquivariance}
D_\xi(g.\Lambda', q.\delta') = \overline{\xi(z)}\, \overline{\xi(\beta(q,\delta'))}D_\xi(\Lambda', \delta').
\end{equation}
\end{lemma}
\begin{proof} Given $(\Lambda', \delta') \in \widetilde{\Omega}_\xi$ we abbreviate
\[
D_\xi(\Lambda', \delta', T) := \frac{1}{m_Z(F_T)} \sum_{t \in \Lambda'_{\delta'} \cap F_T} \overline{\xi(t)}.
\]
Given $g = (z,q) \in G$ we then consider the error term
\begin{eqnarray*}
r_T(\xi, g ) &:=& \left| D_\xi(g\Lambda', q\delta', T) -  \overline{\xi(z)}\, \overline{\xi(\beta(q,\delta'))} D_\xi(\Lambda', \delta', T)\right|\\
&=&  \left|  {\xi(z+\beta(q,\delta))} D_\xi(g\Lambda', q\delta', T) - D_\xi(\Lambda', \delta', T)\right|\\
&=& \frac{1}{m_Z(F_T)}\big|\; \sum_{t \in (g.\Lambda')_{q.\delta'}\cap F_T} \overline{\xi(t-z-\beta(q,\delta))} - \sum_{t \in  \Lambda'_{\delta'} \cap F_T} \overline{\xi(t)}\;\big|.
\end{eqnarray*}
Now $(g.\Lambda')_{q.\delta'} = \Lambda'_{\delta'} + z + \beta(q, \delta')$, and thus
\[
(g.\Lambda')_{q.\delta'}\cap F_T = \left(\Lambda'_{\delta'} \cap (F_T - z - \beta(q, \delta')\right)+ z + \beta(q, \delta').
\]
Plugging this into the above expression for $r_T(\xi, g)$ yields
\begin{eqnarray*}
r_T(\xi, g)  &=& \frac{1}{m_Z(F_T)} \big|\; \sum_{t \in\Lambda'_{\delta'} \cap (F_T - z - \beta(q, \delta')) } \overline{\xi(t)} -  \sum_{t \in  \Lambda'_{\delta'} \cap F_T} \overline{\xi(t)}\;\big|\\
&\leq& \frac{\left| (\Lambda'_{\delta'} \cap (F_T - z - \beta(q, \delta')))\; \Delta \; (\Lambda'_{\delta'} \cap F_T) \right|}{m_Z(F_T)}.
\end{eqnarray*}
By \eqref{Nice1} we can now find some $T_o = T_o(q, \delta') > 0$ such that for all sufficiently large $T$ we have $F_{T - T_o} \subset F_T - z - \beta(q, \delta') \subset F_{T+T_o}$, and hence
\[
(\Lambda'_{\delta'} \cap (F_T - z - \beta(q, \delta')))\; \Delta \;(\Lambda'_{\delta'} \cap F_T) \subset [(\Lambda_\delta' \cap F_{T+T_o}) \setminus (\Lambda'_{\delta'} \cap F_T)] \cup  [(\Lambda_\delta' \cap F_{T}) \setminus (\Lambda'_{\delta'} \cap F_{T-T_o})].
\]
We deduce that
\[
r_T(\xi, g) \; \leq \; \frac{|\Lambda'_{\delta'} \cap F_{T+T_o}|}{m_Z(F_{T})} - \frac{|\Lambda'_{\delta'} \cap F_{T-T_o}|}{m_Z(F_T)}.
\]
Since $(\Lambda', \delta') \in \widetilde{\Omega}_\xi$ we conclude with \eqref{Nice2} that both expressions converge to $D_1(\Lambda', \delta')$, hence $r_T({\xi, g}) \to 0$. Since $D_\xi(\Lambda', \delta', T) \to D_\xi(\Lambda', \delta')$ this then implies that 
\[
D_\xi(g\Lambda', q\delta', T) \to  \overline{\xi(z)}\, \overline{\xi(\beta(q,\delta'))} D_\xi(\Lambda', \delta').
\]
In particular, $D_\xi(g\Lambda', q\delta')$ exists, and similary $D_1(g\Lambda', q\delta')$ exists. This shows that $(g\Lambda', q\delta') \in \widetilde{\Omega}_{\xi}$, and that it satisfies the desired formula.
\end{proof}
We will also need the following boundedness property of twisted fiber densities:
\begin{lemma}[Boundedness of twisted fiber densities]\label{FiberDensityBounded} There exists a constant $C>0$ such that for all $\xi \in \widehat{Z}$ and all $\Lambda' \in \widetilde{\Omega}_\xi$ and all $\delta' \in \pi(\Lambda')$ we have $|D_\xi(\Lambda', \delta')| \leq C$.
\end{lemma}
\begin{proof} Let $\lambda' \in \Lambda'$ such that $\pi(\lambda') = \delta'$. Then by Lemma \ref{TheoremTDE} we have
\[
|D_\xi(\Lambda', \delta')| = |D_\xi((\lambda')^{-1}\Lambda', e)| = \left| \lim_{T \to \infty} \frac{1}{m_Z(F_T)}\sum_{z \in  ((\lambda')^{-1}\Lambda)_e \cap F_T} \overline{\xi(z)} \right| \leq  \varlimsup_{T \to \infty} \frac{|((\Lambda')^{-1}\Lambda')_e \cap F_T|}{m_Z(F_T)}.
\]
Now by \eqref{DifferenceSetInclusion} we have $((\Lambda')^{-1}\Lambda' \subset \Lambda^{-1}\Lambda = \Lambda^2$, hence
\[
|D_\xi(\Lambda', \delta')| \;\leq\; \varlimsup_{T \to \infty} \frac{(\Lambda^2)_e \cap F_T}{m_Z(F_T)} \;=\; \overline{d}((\Lambda^2)_e).
\] 
Since $(\Lambda^2)_e$ is uniformly discrete, the lemma follows from Lemma \ref{DensityBounds}.(i).
\end{proof}
\begin{remark}\label{RemRhoXi}
Let us reformulate Lemma \ref{TheoremTDE} and Lemma \ref{FiberDensityBounded} in terms of the space
\[
\Omega_{\Lambda, \pi} := \{(\Lambda', \delta') \in \Omega_\Lambda \times Q \mid \delta' \in \pi(\Lambda')\}\quad \subset \quad \Omega_\Lambda \times Q.
\]
This space comes equipped with a natural measure class $[\eta]$ represented by the Radon measure $\eta$ defined by
\[
\eta(\psi)= \int_{\Omega_\Lambda} \left( \sum_{\delta' \in \pi(\Lambda')} \psi(\Lambda', \delta') \right)d\nu_{\Lambda}(\Lambda') \quad (\psi \in C_c(\Omega_{\Lambda, \pi})).
\]
To see that $\eta$ is well-defined, we observe that the sum under the integral is finite, and hence continuity of $\pi$ implies continuity of the integrand.
For every $\xi$ we have an $[\eta]$-conull subset $\Omega_{\Lambda, \pi, \xi} \subset \Omega_{\Lambda, \pi}$ given by
\begin{equation}\label{PalmSpace}
\Omega_{\Lambda, \pi, \xi} := \{(\Lambda', \delta') \in \widetilde{\Omega}_\xi \times Q \mid \delta' \in \pi(\Lambda')\},
\end{equation}
and the twisted fiber density function $D_\xi$ is defined and bounded on $\Omega_{\Lambda, \pi, \xi}$ by Lemma \ref{FiberDensityBounded}. Being measurable (as a pointwise limit of continuous functions) it thus defines a function class $D_\xi\in L^\infty(\Omega_{\Lambda, \pi}, [\eta])$, and the function
\begin{equation}\label{DefRhoXi}
\rho_\xi: \Omega_{\Lambda, \pi, \xi} \to [0, \infty), \quad (\Lambda', \delta') \mapsto  |D_\xi(\Lambda', \delta')|^2
\end{equation}
defines an element of $L^\infty(\Omega_{\Lambda, \pi}, [\eta])^G$, the subspace of $G$-invariants.
\end{remark}

\subsection{Constructing eigenfunctions from twisted fiber densities}
Throughout this subsection we adopt the notation of Theorem \ref{RelativelyDenseSpectrum}. Given a character $\xi \in \widehat{Z}$ we are going to construct a family of explicit eigenfunctions for $\xi$ using twisted fiber density functions. A priori, it will not be clear whether these eigenfunctions are zero; we will give a criterion for their non-triviality in the next subsection.

From now on we fix a character $\xi \in \widehat{Z}$. Given $\phi \in C_c(Q)$ and $\Lambda' \in \Omega_\xi$ we define
\[
(\mathcal P_\xi \phi)(\Lambda') := \sum_{\delta' \in \pi(\Lambda')} \phi(\delta') D_\xi(\Lambda', \delta'),
\]
where the sum is actually finite since $\Delta' := \pi(\Lambda') \in \Omega_{\Delta}$ is uniformly discrete. In fact, since there exists a uniform $r>0$ such that all $\Delta' \in \Omega_\Delta$ are $r$-uniformly discrete, the number of summmands in the sum defining $\mathcal P_\xi\phi(\Lambda')$ is bounded independently of $\Lambda'$. We thus deduce from Lemma \ref{FiberDensityBounded} that $\mathcal P_\xi \phi \in L^\infty(\Omega_\Lambda, \nu_\Lambda) \subset L^2(\Omega_\Lambda, \nu_\Lambda)$. We refer to the operator 
\[
\mathcal P_\xi: C_c(Q) \to L^2(\Omega_\Lambda, \nu_\Lambda), \quad (\mathcal P_\xi \phi)(\Lambda') := \sum_{\delta' \in \pi(\Lambda')} \phi(\delta') D_\xi(\Lambda', \delta'),
\]
as the \emph{$\xi$-twisted periodization operator}. 
\begin{remark} Lemma \ref{TheoremTDE} allows one to give a simple formula for the twisted periodization operators in the case where $\Lambda$ is actually a uniform lattice and moreover split in the sense that $\Lambda = \Xi \oplus_\beta \Delta$. In this case, $\Omega_\Lambda = G.\Lambda$, and for all $(z,q) \in G$ we have
\begin{eqnarray*}
D_\xi((z,q)\Lambda, q\delta) &=& D_\xi((z,e)(e,q)\Lambda, eq\delta) \quad = \quad {\overline{\xi(z)}}\,\overline{\xi(\beta(e, q\delta))} D_\xi((e,q)\Lambda, q\delta)\\
&=& {\overline{\xi(z)}}\,\overline{\xi(\beta(e, q\delta))}\overline{\xi(\beta(q, \delta))} D_\xi(\Lambda, e)\\
&=& \overline{\xi(z+\beta(q, \delta))} D_\xi(\Lambda, e),
\end{eqnarray*}
hence the twisted periodization operator $\mathcal P_\xi$ is given by
\begin{equation}\label{PeriodizationEasy}
\mathcal P_\xi f((z,q)\Lambda) =   D_\xi(\Lambda,e) \cdot \sum_{\delta \in \Delta} \phi(q\delta) \overline{\xi(z+\beta(q, \delta))} .
\end{equation}
\end{remark}

\begin{proposition}[Eigenfunctions from twisted periodizations]\label{EFTP} For every $\xi \in \widehat{Z}$ the twisted periodization operator $\mathcal P_\xi$ takes values in the eigenspace $L^2(\Omega_\Lambda, \nu_\Lambda)_\xi$. In particular, if the twisted periodization operator $\mathcal P_\xi$ is non-zero, then $\xi \in {\rm spec}^Z_{\rm pp}(L^2(\Omega_{\Lambda}, \nu_{\Lambda}))$.
\end{proposition}
\begin{proof} Let $z \in Z$; we identify $z$ with the element $(z,e) \in G$. We deduce from \eqref{TwistedDensityEquivariance} that for all $\varphi \in C_c(Q)$ and $\Lambda' \in \widetilde{\Omega}_\xi$,
\begin{eqnarray*} (z.(\mathcal P_\xi\phi))(\Lambda') &=& \sum_{\delta' \in \pi(\Lambda')} \phi(\delta') D_\xi((z^{-1}, e)\Lambda', e.\delta')\\
 &=&  \sum_{\delta' \in \pi(\Lambda')} \phi(\delta') \overline{\xi(z^{-1})}\, \overline{\xi(\beta(e,\delta))}D_\xi(\Lambda',\delta')\\
 &=&\xi(z)  \sum_{\delta' \in \pi(\Lambda')} \phi(\delta') \overline{\xi(0)}D_\xi(\Lambda',\delta')\\
 &=& {\xi(z)} \cdot \mathcal P_\xi \phi(\Lambda').
\end{eqnarray*}
Since $\widetilde{\Omega}_\xi$ is a conull set, this shows that $\mathcal P_\xi \phi(\Lambda') \in L^2(\Omega_\Lambda, \nu_\Lambda)_\xi$.
\end{proof}

\subsection{Palm measures and diffraction coefficients}
In order to derive a lower bound on the pure point spectrum $ {\rm spec}^Z_{\rm pp}(L^2(\Omega_{\Lambda}, \nu_{\Lambda}))$ from Proposition \ref{EFTP} we need to give a criterion which ensures that the twisted periodization operator $\mathcal P_\xi$ is non-zero. In the present subsection, we will derive such a criterion using Palm measures.

Given $\rho \in L^\infty(\Omega_{\Lambda, \pi}, [\eta])$ we can define a bounded linear functional $\eta_{\rho}$ on $C_c(Q)$ by
\[
\eta_{\rho}(\phi)= \int_{\Omega_\Lambda} \left( \sum_{\delta' \in \pi(\Lambda')} \phi(\delta') \rho(\Lambda', \delta') \right)d\nu_{\Lambda}(\Lambda')
\]
If $\rho$ is $G$-invariant, then so is $\eta_{\rho}$, since
\begin{eqnarray*}
 \int_{\Omega_\Lambda} \left( \sum_{\delta' \in \pi(\Lambda')} \phi(g^{-1}\delta') \rho(\Lambda', \delta') \right)d\nu_{\Lambda}(\Lambda') &=&  \int_{\Omega_\Lambda} \left( \sum_{\delta' \in \pi(g\Lambda')} \phi(g^{-1}\delta') \rho(g\Lambda', \delta') \right)d\nu_{\Lambda}(\Lambda')\\ &=&  \int_{\Omega_\Lambda} \left( \sum_{\delta'' \in \pi(\Lambda')} \phi(\delta'') \rho(g\Lambda', g\delta'') \right)d\nu_{\Lambda}(\Lambda').
\end{eqnarray*}
From uniqueness of the Haar measure we thus deduce that for every $\rho \in L^\infty(\Omega_{\Lambda, \pi}, [\eta])^G$ there exists a constant $\alpha(\rho) > 0$ such that for all $\phi \in C_c(Q)$,
\begin{equation}\label{DefPalmMeasure}
\eta_{\rho}(\phi) = \alpha(\rho) \cdot \int_Q \phi \, dm_Q.
\end{equation}
The linear functional $\alpha$ is closely related to the so-called Palm measure from the theory of point processes. By abuse of notation we also refer to $\alpha$ as the \emph{Palm measure} on $\Omega_{\Lambda, \pi}$. In particular, we can apply this Palm measure to the functions $\rho_\xi$ from \eqref{DefRhoXi}. In the sequel we will abbreviate
\[
c_\xi := \alpha(\rho_\xi)= \alpha(|D_\xi|^2).
\]
Unravelling definitions, we see that $c_\xi$ is characterized by the equation
\begin{equation}\label{DiffractionCoefficient}
\int_{\Omega_\Lambda} \left( \sum_{\delta' \in \pi(\Lambda')} \phi(\delta') |D_\xi(\Lambda', \delta')|^2 \right)d\nu_{\Lambda}(\Lambda') = c_\xi \cdot \int_Q \phi \, dm_Q \quad (\phi \in C_c(Q)).
\end{equation}
\begin{remark}
At this point the numbers $(c_\xi)$ are just non-negative real numbers derived from the twisted fiber densities. We will see in Theorem \ref{ThmDiffractionCoefficients} below that they can be interpreted as coefficients of the pure point part of the central diffraction measure of $\Lambda$ with respect to $\nu_\Lambda$.
\end{remark}
We can now provide the promised criterion for non-triviality of the twisted periodization operators $\mathcal P_\xi$.
\begin{proposition}[Twisted periodization operators and diffraction coefficients]\label{TPDC}
The twisted periodization operator $\mathcal P_\xi$ is non-zero if and only if $c_\xi \neq 0$. Thus,
\[
\{\xi \in \widehat{Z}\mid c_\xi \neq 0\} \quad \subset \quad  {\rm spec}^Z_{\rm pp}(L^2(\Omega_{\Lambda}, \nu_{\Lambda})).
\]
\end{proposition}
Concerning the proof of Proposition \ref{TPDC} we introduce the following notation. We fix a choice $(r,R)$ of Delone parameters for $\Delta := \pi(\Lambda)$ with respect to our fixed left-admissible metric on $Q$. Given $q \in Q$, we denote by $B_{r/2}(q) \subset Q$ the open ball of radius $r$ around the identity in $Q$ and by $C_c(B_{r/2}(q)) \subset C_c(Q)$ the subset of functions supported inside $B_{r/2}(q)$. Then Proposition \ref{TPDC} will follow from the following lemma.
\begin{lemma}\label{TPDCExplicit}
For every $q \in Q$ and $\phi \in C_c(B_{r/2}(q))$,
\[
\|\mathcal P_\xi \phi\|^2_{L^2(\Omega_\Lambda, \nu_\Lambda)} = c_\xi \cdot \|\phi\|^2_{L^2(Q, m_Q)}. 
\]
\end{lemma}
\begin{proof} For every $\phi \in C_c(Q)$, $\Lambda' \in \Omega_\Lambda$ and $\Delta' := \pi(\Lambda')$ we have 
\[
|\mathcal P_\xi \phi(\Lambda')|^2 = \sum_{\delta_1, \delta_2 \in \Delta'  \cap \supp(\phi)} \phi(\delta_1)\; D_\xi(\Lambda', \delta_1) \; \overline{\phi(\delta_2)} \;\overline{D_\xi(\Lambda', \delta_2)}.
\]
Since $\Delta' \in \Omega_{\Delta}$, we know that $\Delta'$ is $r$-uniformly discrete. Thus if ${\rm supp}(\phi)$ has diameter less than $r$, then $|\Delta'  \cap \supp(\phi)| \leq 1$. This then implies that 
\[
|\mathcal P_\xi \phi(\Lambda')|^2 = \sum_{\delta' \in \Delta'} |\phi(\delta')|^2\; |D_\xi(\Lambda', \delta')|^2,
\]
and we deduce with \eqref{DiffractionCoefficient} that
\[
\|\mathcal P_\xi \phi(\Lambda')\|^2_{L^2(\Omega_\Lambda, \nu_\Lambda)} = \int_{\Omega_\Lambda} \left( \sum_{\delta' \in \pi(\Lambda')} |\phi(\delta')|^2 \; |D_\xi(\Lambda', \delta')|^2\right)d\nu_{\Lambda}(\Lambda') =  c_\xi \cdot \int_Q |\phi|^2 \, dm_Q = c_\xi \cdot \|\phi\|^2_{L^2(Q, m_Q)},
\]
which finishes the proof.
\end{proof}
\begin{proof}[Proof of Proposition \ref{TPDC}] If $c_\xi \neq 0$ then we can pick any $q \in Q$ and $\phi \in C_c(B_{r/2}(q))$ to produce an eigenfunction $\mathcal P_\xi \phi \in L^2(\Omega_{\Lambda}, \nu_{\Lambda})_\xi$, which is then of non-zero norm by Lemma \ref{TPDCExplicit}. Conversely, assume that $c_\xi = 0$ and let $\phi \in C_c(Q)$ Using a partition of unity we can write $\phi$ as a finite sum of the form $\phi = \phi_1 + \dots +\phi_n$ where each $\phi_j$ is supported in an open ball of radius at most $r/2$. By Lemma \ref{TPDCExplicit} we then have $\mathcal P_\xi\phi_j = 0$ for all $j = 1,\dots, n$ and thus $\mathcal P_\xi \phi  = 0$. Since $\phi \in C_c(Q)$ was arbitrary this proves $\mathcal P_\xi = 0$.
\end{proof}
\begin{remark} Let us define a norm on $C_c(Q)$ by
\[
\|\phi\|  := \inf\left\{\|\phi_1\|_{L^2(Q, m_Q)} + \dots + \|\phi_n\|_{L^2(Q, m_Q)} \mid \phi = \phi_1 + \dots + \phi_n, \; \exists q_1, \dots, q_n \in Q: {\rm supp (\phi_j)} \subset B_{r/2}(q_j)\right\}
\]
and denote by $\mathcal F_r(Q)$ the completion of $C_c(Q)$ under this norm. Since the $Q$-action on $C_c(Q)$ preserves $\|\cdot\|$, it extends to an isometric action on the Banach space $\mathcal F_r(Q)$. Also note that if ${\rm supp}(\phi)$ has diameter at most $r$, then $\|\phi\| = \|\phi\|_{L^2(Q, m_Q)}$, and hence $L^2(B_{r/2}(q), m_Q) \subset \mathcal F_r(Q)$ for all $q \in Q$.

Now assume that $c_\xi \neq 0$, then for all $\phi \in C_c(Q)$ and $\phi_1, \dots, \phi_n \in C_c(Q)$ with support of diameter at most $r$ satisfying $\phi = \phi_1 + \dots + \phi_n$ we have
\[
\|\mathcal P_\xi\phi\|_{L^2(\Omega_\Lambda, \nu_\Lambda)} \leq \sum_{j=1}^n \|\mathcal P_\xi\phi_j\|_{L^2(\Omega_\Lambda, \nu_\Lambda)} = c_\xi^{1/2} \cdot \sum_{j=1}^n  \|\phi_j\|_{L^2(Q, m_Q)},\]
and thus passing to the infimum yields $\|\mathcal P_\xi \phi\|_{L^2(\Omega_\Lambda, \nu_\Lambda)} \leq c_\xi^{1/2} \|\phi\|$, with equality if the support of $\phi$ has diameter at most $r$. We thus deduce that $\mathcal P_\xi$ extends to a bounded linear operator
\[
\mathcal P_\xi: \mathcal F_r(Q) \to L^2(\Omega_\Lambda, \nu_\Lambda)_\xi \quad \text{with} \quad \|\mathcal P_\xi\|_{\rm op} = c_\xi^{1/2}.
\]
From \eqref{TwistedDensityEquivariance} one deduces that this operator has the equivariance property
\[
(z,q).(\mathcal P_\xi \phi) = \chi(z) \mathcal P_\xi \left(\phi \cdot \overline{\beta(q^{-1}, \cdot)}\right).
\]
\end{remark}

\subsection{Estimates on the diffraction coefficients} 
In view of Proposition \ref{TPDC} we have reduced the proof of Theorem \ref{RelativelyDenseSpectrum} to showing that $c_\xi \neq 0$ for a relatively dense set of characters $\xi \in \widehat{Z}$. In this subsection we will establish a refined version of this result.

Since $\Lambda$ is $\pi$-aligned, the fiber $\Lambda_e$ is relatively dense in $Z$, and hence $(\Lambda^2)_e$ is relatively dense in $Z$, hence a Meyer set (see Corollary \ref{CorSquares}). It is thus harmonious by Lemma \ref{MeyerAbelian}, i.e.\ for every $\epsilon > 0$ the $\epsilon$-dual
\[
\widehat{(\Lambda^2)_e}^{\epsilon} =  \{\xi \in \widehat{Z} \mid \|(\xi-1)|_{\Lambda^2}\| < \epsilon\} \subset \widehat{Z}
\]
is relatively dense in $\widehat{Z}$. In view of this fact, the following proposition concludes the proof of Theorem \ref{RelativelyDenseSpectrum}.
\begin{proposition}[Estimates on diffraction coefficients]\label{cxiInequalities} In the situation of Theorem \ref{RelativelyDenseSpectrum} the numbers $c_\xi$ from \eqref{DiffractionCoefficient} satisfy the following estimates.
\begin{enumerate}[(i)]
\item $c_1 >0$.
\item If $\epsilon \in (0,1)$ and $\xi \in \widehat{(\Lambda^2)_e}^{\epsilon}$, then $ (1-\epsilon) \cdot c_{1} \leq c_\xi \leq  c_1$.
%\[
% (1-\epsilon) \cdot c_{1} \quad\leq\quad c_\xi \quad\leq\quad c_1.
%\]
\end{enumerate}
In particular, if $\epsilon \in (0,1)$ and $\xi \in \widehat{(\Lambda^2)_e}^{\epsilon}$, then $c_\xi \neq 0$ and hence $\xi \in {\rm spec}_{\rm pp}^Z(L^2(\Omega_\Lambda, \nu_\Lambda))$.
\end{proposition}
\begin{proof} (i) If $c_1=0$, then by \eqref{DiffractionCoefficient} for all $\phi \in C_c(Q)$ we would have
\[
\int_{\Omega_\Lambda} \left( \sum_{\delta' \in \pi(\Lambda')} \phi(\delta') |D_1(\Lambda', \delta')|^2 \right)d\nu_{\Lambda}(\Lambda') = 0.
\]
This would contradict Lemma \ref{DensityBounds}.(ii).

(ii) Fix  $\epsilon \in (0,1)$ and $\xi \in \widehat{(\Lambda^2)_e}^{\epsilon}$. In view of \eqref{DiffractionCoefficient} it suffices to show that for all $\Lambda'$ in the conull subset $\Omega_\xi \cap \Omega_1 \subset \Omega_\Lambda$ and all $\delta' \in \pi(\Lambda')$ we have
\[
  (1-\eps)\cdot|D_1(\Lambda',\delta')| \;\leq\; |D_\xi(\Lambda',\delta')| \;\leq\; |D_1(\Lambda',\delta')|.
\]
The upper bound holds for all $\xi \in \widehat{Z}$ by the triangle inequality, since $\|\xi\|_\infty \leq 1$. To establish the lower bound we fix $\Lambda' \in \Omega_\xi \cap \Omega_1$, $\delta' \in \pi(\Lambda')$ and $n' \in \Lambda'_{\delta'}$. Since $(n', \delta') \in \Lambda'$ we have $(n', \delta')^{-1}\Lambda' \subset \Lambda^{-1}\Lambda$ by Lemma \ref{DifferenceSetInclusion}. We deduce that 
\[
|\xi(t)-1| < \epsilon \quad \text{for all }t \in ((n', \delta')^{-1}\Lambda')_e.
\]
On the other hand, by \eqref{TwistedDensityEquivariance} we have
\[
|D_\xi(\Lambda', \delta')| = |D_\xi((n', \delta')^{-1}.\Lambda', e)|.
\]
Combining these two observations we deduce that
\begin{eqnarray*}
|D_\xi(\Lambda',\delta')| 
&=& 
|D_\xi((t,\delta')^{-1}\Lambda',e_Q)| \\
&=& \big|\; \lim_{T\to \infty} \frac{1}{m_N(F_T)} \sum_{t \in ((n,\delta')^{-1}\Lambda')_{e_Q}} \overline{\xi(t)} \; \big|  \quad = \quad  \big|\; \lim_{T\to \infty} \frac{1}{m_N(F_T)} \sum_{t \in ((n,\delta')^{-1}\Lambda')_{e_Q}} \left(1+(\overline{\xi(t)}-1)\right) \; \big|  \\
&\geq & \left(\lim_{T \to \infty} \frac{1}{m_N(F_T)} \sum_{t \in ((n,\delta')^{-1}\Lambda')_{e_Q}} 1\right)  - \varlimsup_{T \to \infty} \frac{1}{m_N(F_T)} \sum_{t \in ((n,\delta')^{-1}\Lambda')_{e_Q}} \left| \overline{\xi(t)} - 1\right| \\
&\geq & D_1(\Lambda',\delta') -\epsilon D_1(\Lambda', \delta').
\end{eqnarray*}
This establishes the lower bound and finishes the proof.
\end{proof}

\section{Central diffraction}\label{SecDiffraction}

\subsection{Auto-correlation measures}
We recall some basic facts about auto-correlation measures. The material of this subsection is taken from \cite{BHP1}. Let $G$ be a lcsc group and $\Lambda \subset G$ be a Delone subset of finite local complexity. Recall that the periodization map $\mathcal P$ was defined by
\[
\mathcal P: C_c(G) \to C(\Omega_\Lambda), \quad \mathcal Pf(\Lambda') := \sum_{x \in \Lambda'} f(x).
\]
Given a $G$-invariant probability measure $\nu_{\Lambda}$ on $\Omega_\Lambda$, the \emph{auto-correlation measure} $\eta = \eta(\Lambda, \nu_\Lambda)$ of $\Lambda$ with respect to the measure $\nu_{\Lambda}$ is the unique positive-definite Radon measure on $G$ satisfying
\[
\eta(f^* \ast f) = \|\mathcal Pf\|_{L^2(\Omega_\Lambda, \nu_\Lambda)} \quad (f \in C_c(G)).
\]
If $G$ is amenable with Haar measure $m_G$, then there exists a F\o lner sequence $(B_t)$ in $G$ such that for a $\nu_{\Lambda}$-generic choice of $\Lambda' \in \Omega_{\Lambda}$ the auto-correlation measure is given by
\[
\eta(f) =  \lim_{T \to \infty}\frac{1}{m_G(B_T)} \sum_{x \in \Lambda' \cap B_T} \sum_{y \in \Lambda'} f(x^{-1}y) \quad (f \in C_c(G)),
\]
see \cite{BHP1}. If $\Omega_\Lambda$ is uniquely ergodic, then one may choose $\Lambda' := \Lambda$. From the above formula one deduces that the auto-correlation measure is supported on $(\Lambda')^{-1}\Lambda'$. With Lemma \ref{DifferenceSetInclusion} we thus deduce that
\begin{equation}\label{AutocorrelationSupport}
{\rm supp}(\eta) \quad \subset \quad \Lambda^{-1}\Lambda.
\end{equation}
Since $\Lambda^{-1}\Lambda$ is closed and discrete, this implies in particular, that $\eta$ is pure point.

\subsection{Definition of central diffraction}
From now on we fix a $G$-invariant probability measure $\nu_\Lambda$ on $\Omega_\Lambda$ and denote by $\eta = \eta(\Lambda, \nu_\Lambda)$ the associated auto-correlation measure, which by definition is a positive-definite pure point Radon measure on $G$.
\begin{remark}[The abelian case] Assume that $Q = \{e\}$ is the trivial group so that $G = Z$ is abelian. Since $\eta$ is positive-definite, it admits a Fourier transform $\widehat{\eta}$, which is a positive Radon measure on the Pontryagin dual $\widehat{G}$ of $G$. This measure is then called the \emph{diffraction measure} of of $\Lambda$ with respect to the measure $\nu_{\Lambda}$.
\end{remark}
In the general case, the group $G$ is non-abelian, hence it is not immediately clear how to define an associated diffraction measure. We will however be able to admit a ``central diffraction measure'' for the abelian subgroup $Z<G$.

To explain the construction of this measure, we recall from \eqref{AutocorrelationSupport} that $\eta$ is supported in the uniform approximate lattice $\Lambda^2 = \Lambda^{-1}\Lambda$, say 
\[
\eta = \sum_{x\in \Lambda^{-1}\Lambda} c_\eta(x) \delta_x.
\]
By our assumption, the set $\Lambda^2$ is $\pi$-adapted, and hence $\Delta^2 = \pi(\Lambda^2)$ is a uniform approximate lattice in $Q$. We can now decompose $\eta$ as a discrete sum 
\[
\eta = \sum_{\delta \in \Delta^2} \widetilde{\eta}_\delta, \quad \text{where} \quad \widetilde{\eta}_\delta := \sum_{x \in \pi^{-1}(\delta)} c_\eta(x) \delta_x.
\]
Here for each $\delta \in \Delta^2$ the measure $\widetilde{\eta}_\delta$ is supported on $\pi^{-1}(\delta) = Zs(\delta)$, and we can identify it with a measure $\eta_{\delta}$ on the abelian group $Z$ by translating it by right-multiplication by $s(\delta)^{-1}$. We refer to the measures $(\eta_\delta)_{\delta \in \Delta^2}$ as the \emph{fiber measures} of $\eta$.

While the auto-correlation measure $\eta$ is positive-definite, the fiber measures $\eta_\delta$ will not be positive-definite in general, but he fiber measure
\[
\eta_e = \sum_{x \in \Delta^2 \cap Z} c_\eta(x) \delta_x
\]
over the identity is always positive definite, hence Fourier-transformable. Indeed if $\phi \in C_c(Q)$ is supported in a sufficiently small identity neighbourhood and $\psi \in C(Z)$, then $\eta_e(\psi^* \ast \psi) = \eta((\psi \otimes \phi)^* \ast (\psi \otimes \phi)) \geq 0$.
\begin{definition} The measure $\widehat{\eta_e} \in \mathcal R(\widehat{Z})$ is called the \emph{central diffraction measure} of $\Lambda$ with respect to $\nu_\Lambda$.
\end{definition}
\subsection{Atoms of the central diffraction measure}
We keep the notations of Convention \ref{AssumptionsSpectral} and the previous subsection. The goal of this subsection is to provide a formula for the pure point part $\widehat{\eta_e}^{\rm pp}$ of the central diffraction measure $\widehat{\eta_e}$ of $\Lambda$ with respect to $\nu_\Lambda$ so that, by definition,
\[
\widehat{\eta_e}^{\rm pp} = \sum_{\{\xi \in \widehat{Z} \mid \widehat{\eta_e}(\{\xi\}) \neq 0\}}  \widehat{\eta_e}(\{\xi\}) \cdot \delta_\xi.
\]
The following theorem relates the atoms of the central diffraction to the number $c_\xi$ from \eqref{DiffractionCoefficient}.
\begin{theorem}\label{ThmDiffractionCoefficients} 
The pure point part $\widehat{\eta_e}^{\rm pp}$ of the central diffraction measure $\widehat{\eta_e}$ of $\Lambda$ with respect to $\nu_\Lambda$ is given by
\[
\widehat{\eta_e}^{\rm pp} = \sum_{\{\xi \in \widehat{Z}\mid c_\xi \neq 0\}} c_\xi \delta_\xi,
\]
where $c_\xi$ is given by \eqref{DiffractionCoefficient}.
\end{theorem}
We postpone the proof of the theorem to first derive some consequences.
\begin{remark}
Recall that if $c_\xi \neq 0$, then by Proposition \ref{TPDC} the twisted periodization operator $\mathcal P_\xi$ is non-zero and thus $L^2(\Omega_\Lambda, \nu_\Lambda)_\xi \neq \{0\}$. It thus follows from Theorem \ref{ThmDiffractionCoefficients} that the atoms of the central diffraction measure $\widehat{\eta_e}$ are contained in the pure point spectrum of $Z \curvearrowright (\Omega_\Lambda, \nu_\Lambda)$. This generalizes the classical observation that \emph{``the diffraction spectrum is contained in the dynamical spectrum''}.
\end{remark}
\begin{remark}[Bragg peaks] Combining Proposition \ref{cxiInequalities} and Theorem \ref{ThmDiffractionCoefficients} we deduce that the atomic part of the central diffraction measure is non-trivial, and that the largest atom of the central diffraction measure is located at the trivial character $\xi = 1$. The size of the other atoms should then be considered in comparison to the maximal atom. We thus say that the central diffraction measure has an \emph{$(1-\epsilon)$-Bragg peak} at a character $\xi \in \widehat{Z}$ if the central diffraction measure has an atom of size at least $(1-\epsilon)c_1$. With this terminology we can now formulate our results as follows. 
\end{remark}
\begin{corollary}[Relative density of $(1-\epsilon)$-Bragg peaks]\label{Bragg}
For every $\epsilon \in (0,1)$ the central diffraction has an $(1-\epsilon)$-Bragg peak at every $\xi \in \widehat{(\Lambda^{-1}\Lambda)_e}^\epsilon$. In particular, the central diffraction admits a relatively dense set of $(1-\epsilon)$ Bragg peaks. \qed
\end{corollary}
We now turn to the proof of Theorem \ref{ThmDiffractionCoefficients}. Our starting point is the following formula for $\widehat{\eta_e}$. 
\begin{proposition} Let $(r, R)$ be Delone parameters for $\Delta^2$ and let $\phi \in C_c(Q)$ with support contained in $B_{r/4}(e)$ and normalized to $\|\phi\|_{L^2(Q, dm_Q)} = 1$. Then for all $\psi \in C_c(Z)$ we have
\[
\widehat{\eta_e}(|\widehat{\psi}|^2) = \|\mathcal P(\psi \otimes \phi)\|^2_{L^2(\Omega_\Lambda, \nu_\Lambda)}.
\]
\end{proposition}
\begin{proof} Let $\psi \in C_c(Z)$ and let $\phi \in C_c(Q)$ with support contained in $B_{r/4}(e)$. If we abbreviate $f := \psi \otimes \phi$, then we have
\begin{eqnarray*}
\eta(f^* \ast f) &=&\sum_{x \in \Lambda^{-1}\Lambda}c_\eta(z,q)\left(\int_G \overline{f(z', q')} f((z', q')(z,q)) dm_G(z', q')\right)\\
&=&  \int_Z \int_Q \sum_{(z,q) \in \Lambda^{-1}\Lambda}c_\eta(z,q)\overline{\psi}(z') \psi(z'+z+\beta(q',q)) \overline{\phi(q')}\phi(q'q) dm_Q(q') dm_Z(z')
\end{eqnarray*}
Note that if $(z,q) \in \Lambda^{-1}\Lambda$, then $\phi(q) = 0$ unless $q \in {\rm supp}(\phi) \cap \Delta^2 = \{e\}$, and hence $\beta(q, q') = 0$ and $z \in \Lambda^{-1}\Lambda \cap Z$. If we assume moreover that $\phi$ is normalized to $\|\phi\|_{L^2(Q, dm_Q)} = 1$, then we obtain
\begin{eqnarray*}
\eta(f^* \ast f) &=& \sum_{z \in \Lambda^{-1}\Lambda \cap Z}\int_Z c_\eta(z,e)\overline{\psi}(z') \psi(z'+z)dm_Z(z') \int_Q \overline{\phi(q')}\phi(q'e) dm_Q(q')\\
&=& \eta_e(\psi^* \ast \psi).
\end{eqnarray*}
Since $\eta(f^* \ast f) = \|\mathcal Pf\|^2_{L^2(\Omega_\Lambda, \nu_\Lambda)}$ and $\widehat{\eta_e}(|\widehat{\psi}|^2) = \eta_e(\psi^* \ast \psi)$, the proposition follows.
\end{proof}
This yields a formula for the spectral measure of $\widehat{\eta_e}$, and as in \cite[Prop. 7]{BL04} one deduces:
\begin{corollary}\label{AtomsCentralDiffraction} Let $\xi \in \widehat{Z}$, $\psi \in C_c(Z)$ with $\widehat{\psi}(\xi) \neq 0$ and $\phi \in C_c(Q)$ with support contained in $B_{r/4}(e)$ and normalized to $\|\phi\|_{L^2(Q, dm_Q)} = 1$.  Then
\[
\widehat{\eta_e}(\{\xi\})  = |\widehat{\psi}(\xi)|^{-2} \cdot { \|\pi_\xi(\mathcal P(\psi \otimes \phi))\|^2_{L^2(\Omega_\Lambda, \nu_\Lambda)}}.\qed
\]
\end{corollary}
 We introduce the following notation: Given a function $f \in C_c(G)$ and a character $\xi \in \widehat{Z}$ we abbreviate
\[
f_\xi(q) := \int_Z f(z,q) \xi(z) dm_Z(z).
\]
For example, if $f = \phi \otimes \psi$ is a product with $\phi \in C_c(Z)$ and $\phi \in C_c(Q)$, then $f_\xi(q) = \widehat{\phi}(\xi^{-1}) \cdot \phi$. Now the key step in the proof of Theorem \ref{ThmDiffractionCoefficients} is given by the following proposition.
\begin{proposition}\label{Proposition1} For all $\xi \in \widehat{Z}$ we have
\[
\pi_\xi(\mathcal P f) = \mathcal P_\xi f_\xi.
\]
\end{proposition}
\begin{proof} First we note that the eigenspace $L^2(\Omega_\Lambda, \nu_\Lambda)_\xi$ is just the space of $Z$-fixpoints under the action of $Z$ on $L^2(\Omega_\Lambda, \nu_\Lambda)$ given by $z.f(\Lambda') = \overline{\xi(z)}f(z^{-1}.\Lambda')$. It thus follows from the mean ergodic theorem that the projection $\pi_\xi: L^2(\Omega_\Lambda, \nu_\Lambda) \to L^2(\Omega_\Lambda, \nu_\Lambda)_\xi$  is given by
\[
\pi_\xi(h)(\Lambda') = \lim_{T \to \infty} \frac{1}{m_N(F_T)}\int_{F_T}\overline{\xi(z)}h(z^{-1}\Lambda')dm_Z(z) \quad (h \in L^2(\Omega_\Lambda, \nu_\Lambda)),
\]
where convergence is in $L^2$-norm. Fix $f \in C_c(G)$ and $\Lambda' \in \Omega_1 \cap\Omega_\xi$, and let $K$ be a compact symmetric identity neighbourhood containing ${\rm supp}(f)$. By \eqref{Nice1} we then find $T_K>0$ such that $F_{T} \subset F_TK \subset F_{T+T_K}$. This implies in particular that
\[
 \lim_{T \to \infty}\frac{1}{m_Z(F_{T})} \sum_{t \in \Lambda'_{\delta'} \cap F_{T}K}\overline{\xi(t)} = D_\xi(\Lambda', \delta').
\]
By Lemma \ref{BB+2} applied with $B := F_{T+T_K}$ we then have for all $t \in F_TK \subset F_{T+T_K}$,
\[
\overline{\xi(t)} f_\xi(\delta') = (\overline{\xi} \ast f(\cdot, \delta'))(t) = \int_{F_{T+2T_K}} \overline{\xi(z)} f(z^{-1}t, \delta') dm_Z(z).
\]
We deduce that
\begin{eqnarray*}
&& \left|\int_{F_{T}}\overline{\xi(z)}\mathcal P f(z^{-1}\Lambda')dm_Z(z) -  \sum_{\delta' \in \pi(\Lambda')} f_\xi(\delta') \cdot \left(\sum_{t \in \Lambda'_{\delta'} \cap F_{T}}\overline{\xi(t)}\right)\right|\\ 
&=& \left|\sum_{\delta' \in \pi(\Lambda')} \sum_{t \in \Lambda'_{\delta'}}\int_{F_{T}}\overline{\xi(z)} f(z^{-1}t, \delta')dm_Z(z) -  \sum_{\delta' \in \pi(\Lambda')} \sum_{t \in \Lambda'_{\delta'} \cap F_{T}}\overline{\xi(t)} f_\xi(\delta')\right|\\
&=& \left|\sum_{\delta' \in \pi(\Lambda')} \sum_{t \in \Lambda'_{\delta'} \cap F_TK}\int_{F_{T}}\overline{\xi(z)} f(z^{-1}t, \delta')dm_Z(z) -  \sum_{\delta' \in \pi(\Lambda')} \sum_{t \in \Lambda'_{\delta'} \cap F_{T}}\overline{\xi(t)} f_\xi(\delta')\right|\\
 &\leq& \left|\sum_{\delta' \in \pi(\Lambda')} \sum_{t \in \Lambda'_{\delta'} \cap F_{T}K}\int_{F_{T+2T_K}} \overline{\xi(z)} f(z^{-1}t, \delta')dm_Z(z) - \overline{\xi(t)} f_\xi(\delta')\right|\\
&&+ m_Z(F_{T+2T_K} \setminus F_T)\cdot \|\xi\|_\infty \|f\|_\infty + \left|\sum_{\delta' \in \pi(\Lambda')} f_\xi(\delta')\right| |(F_{T+T_K}\setminus F_T) \cap \Lambda'_{\delta'}|\\
 &=& m_Z(F_{T+2T_K} \setminus F_T)\cdot \|f\|_\infty + |{\rm supp}f_\xi \cap \pi(\Lambda'))| \cdot \|f\|_\infty  |(F_{T+T_K}\setminus F_T) \cap \Lambda'_{\delta'}|.
\end{eqnarray*}
The right hand side divided by $m_Z(F_T)$ converges to $0$ since 
\[
\frac{m_Z(F_{T+2T_K} \setminus F_T)}{m_Z(F_T)} \to 0 \qand  \frac{|(F_{T+T_K}\setminus F_T) \cap \Lambda'_{\delta'}|}{m_Z(F_T)} \to 0,
\]
by exact volume growth and because $\Lambda' \in \Omega_1$. We conclude that
\begin{eqnarray*}
\pi_\xi(\mathcal Pf)(\Lambda') &=& \lim_{T \to \infty}\frac{1}{m_Z(F_{T})}\int_{F_{T}}\overline{\xi(z)}\mathcal P f(z^{-1}\Lambda')dm_Z(z)\\ &=&  \sum_{\delta' \in \pi(\Lambda')} f_\xi(\delta') \cdot \left( \lim_{T \to \infty}\frac{1}{m_Z(F_{T})} \sum_{t \in \Lambda'_{\delta'} \cap F_{T}}\overline{\xi(t)}\right)\\
&=&  \sum_{\delta' \in \pi(\Lambda')} f_\xi(\delta') D_\xi(\Lambda', \delta') \quad = \quad  \mathcal P_\xi f_\xi(\Lambda'),
\end{eqnarray*}
which finishes the proof.
\end{proof}
\begin{proof}[Proof of Theorem \ref{ThmDiffractionCoefficients}] Let $\psi \in C_c(Z)$ and $\phi \in C_c(Q)$ with support contained in $B_{r/4}(e)$. If we abbreviate $f:=(\psi \otimes \phi)$, then by Proposition \ref{Proposition1} and Lemma \ref{TPDCExplicit} we have for all $\xi \in \widehat{Z}$,
\[
\|\pi_\xi(\mathcal P f\|^2_{L^2(\Omega_\Lambda)}  = \|\mathcal P_\xi f_\xi\|^2_{L^2(\Omega_\Lambda)} = c_\xi \|f_\xi\|^2_{L^2(Q, m_Q)}.
\]
Since $f_\xi(q) = \widehat{\psi}(\xi^{-1}) \cdot \phi$, this shows that
\[
\|\pi_\xi(\mathcal P (\psi \otimes \phi))\|^2_{L^2(\Omega_\Lambda)} = c_\xi \cdot |\widehat{\psi}(\xi)|^2 \cdot \|\phi\|^2_{L^2(Q, m_Q)}
\]
In view of Corollary \ref{AtomsCentralDiffraction} this implies that $\widehat{\eta_e}(\{\xi\}) = c_\xi$ and finishes the proof.
\end{proof}

\newpage
\appendix

\section{Background on nilpotent Lie groups}\label{AppNilpotent}

Recall that a Lie algebra $\L g$ is called \emph{$k$-step nilpotent} if its \emph{lower central series} as given by
\[
\L g_1 := \L g \qand \L g_{m+1} := [\L g, \L g_m] \text{ for all }m \geq 1.
\]
satisfies $\L g_{k+1} = \{0\} \neq \L g_k$. We say that $\L g$ is \emph{at least $k$-step nilpotent} (respectively \emph{at most $k$-step nilpotent}) if it is $n$-step nilpotent for some $n \leq k$ (respectively $n \geq k$). 

Given a $k$-step nilpotent Lie algebra $\L g$, the Baker-Campbell-Hausdorff series defines a multiplication $\ast$ on $\L g$ (given by a polynomial of degree $k$). The group $G = (\L g, \ast)$ carries two natural structures: On the one hand, it is a $1$-connected (i.e.\ connected and simply-connected) nilpotent real Lie group, on the other hand it is a unipotent algebraic group over $\R$. Conversely, every $1$-connected real nilpotent Lie group, respectively unipotent algebraic group over $\R$ arises in this way (up to isomorphism). This yields an equivalence of categories between $1$-connected real nilpotent Lie groups and unipotent algebraic groups over $\R$.

In the sequel, when given a subset $A \subset G$, we will refer to the closure of $A$ in the Lie group topology simply as the closure of $A$, and to the \emph{closure} of $A$ in the Zariski topology as the \emph{Zariski closure}. The adjectives \emph{closed, dense, Zariski closed, Zariski dense} are understood accordingly.
 
\begin{remark}[Closed connected subgroups]\label{ClosedSubgpNilpotentGp}
Let $G$ be a $1$-connected nilpotent Lie group.

If $H<G$ be a closed and connected subgroup, then $H$ is a connected nilpotent Lie group with respect to the subspace topology \cite[Prop. 9.3.9]{HilgertNeeb}, in particular, if we denote by $\L h< \L g$ its Lie algebra, then the exponential function $\exp: \L h \to H$ is onto  \cite[Cor. 11.2.7]{HilgertNeeb}, and thus $\L h = \log(H)$. Then $\exp$ and $\log$ restrict to mutually inverse diffeomorphisms between $H$ and $\L h$, and thus $H \cong (\L h, \ast)$ as Lie groups. This implies, firstly, that $H$ is contractible, and in particular arcwise connected  and $1$-connected. It also implies that $H$ is a unipotent linear algebraic group, and a Zariski-closed subgroup of $G$. Thus a connected subgroup of $G$ is closed if and only if it is Zariski-closed. 

Examples of closed connected subgroups of $G$ arise as follows. Firstly, the whole \emph{lower central series} of $G$ as given by $G_1:= G$ and $G_{n+1}:= [G,G_n]$ for $n \geq 1$ consists of closed connected subgroups. For closedness see e.\ g.\ \cite[I.2.4, Corollary 1]{Borel}, and connectedness follows inductively from the general fact that commutators of connected subgroups are connected. In particular, the commutator subgroup $G_2 = [G,G]$ is closed and connected. Secondly, centralizers of elements, hence of arbitrary subgroups are always closed (as preimage of the identity under a suitable commutator map), and centralizers of \emph{analytic} subgroups of $G$ are connected by \cite[Prop.1]{Saito}. In particular, if $H < G$ is closed and connected, then it is analytic (since it is arcwise connected, see \cite[Theorem]{Goto}), hence its centralizer $C_G(H)$ is also closed and connected. This applies in particular to the center $Z(G) = C_G(G)$ of $G$. Finally, we can iterate these constructions, using the fact that a closed connected subgroup of a closed connected subgroup is closed and connected. 

We also observe that if $H < G$ is closed and connected, then $G/H$ is contractible. Indeed, since both $G$ and $H$ are contractible, this follows from the long exact sequence in homotopy of the fibration $H \to G \to G/H$. In particular, if $H$ is moreover normal in $G$, then $G/H$ is again a $1$-connected nilpotent Lie group, hence contractible. In particular we deduce that the extension
\[
H \to G \to G/H
\]
is a trivial fiber bundle, hence admits a continuous section.
\end{remark}

Now let $G$ be a $1$-connected $2$-step nilpotent Lie group with Lie algebra $\L g$. By definition, this means that $\L g$ is non-abelian and that the commutator subalgebra $[\L g, \L g]$ is contained in the center $\L z = \L{z(g)}$ of $\L g$, i.e.\ the quotient $V:= \L{g/z(g)}$ is non-trivial and abelian. We then have a central extension of Lie algebras
\[
0 \to \L z \to \L g \xrightarrow{\pi} V \to 0,
\]
and we pick once and for all a \emph{linear} section $s: V \to \L g$.

We may assume that $G = (\L g, \ast)$, where the Baker--Campbell--Hausdorff product is given by 
\begin{equation}\label{BCHStep2}
X \ast Y = X + Y + \frac 1 2\, [X,Y].
\end{equation}
Let us denote by $Z$ the additive group $(\L z, +)$ and also consider $V$ as an abelian group under addition. We then have a central extension
\[
 Z \to G = (\L g, \ast) \xrightarrow{\pi} V,
\]
of Lie groups, and $s$ defines a continuous section of this extension as well. We claim that if we define $\beta:V \times V \to Z$ by $\beta(v_1, v_2) :=  \frac 1 2[s(v_1), s(v_2)]$, then $G = Z \oplus_\beta V$. Indeed, since $s$ was chosen to be linear, we have for all $Z_1, Z_2$ in $\L z$ and $v_1, v_2 \in V$
\begin{eqnarray*}
(Z_1 + s(v_1)) \ast (Z_2 + s(v_2)) &=&  (Z_1 + s(v_1)) + (Z_2 + s(v_2)) + \frac 1 2 [Z_1 + s(v_1), Z_2 + s(v_2)]\\
 &=& (Z_1 + Z_2 + \frac 1 2 [s(v_1), s(v_2)]) + s(v_1+v_2).
\end{eqnarray*}
We thus see that with respect to the chosen linear structures on the abelian Lie groups $V$ and $Z$, the cocycle $\beta$ is an anti-symmetric bilinear form. This anti-symmetric bilinear form is moreover non-degenerate in the sense that if $\beta(v,w) = 0$ for all $w \in V$, then $v = 0$. Indeed, otherwise $(0,v)$ would be central in $Z \oplus_\beta V$, which is a contradiction. To summarize:
\begin{proposition}\label{NormalFormNilpotent} Every $1$-connected $2$-step nilpotent Lie group $N$ is isomorphic to $Z \oplus_\beta V$ for real vector spaces $V$, $Z$ of positive dimension and a non-degenerate, antisymmetric bilinear map $\beta: V \times V \to Z$.\qed
\end{proposition}
Conversely, if $V$, $Z$ are real vector spaces of positive dimension and  $\beta: V \times V \to Z$ is a non-degenerate antisymmetric bilinear map, then $Z \oplus_\beta V$ is a $2$-step nilpotent Lie group with center $Z$. 
\begin{example}
Let $V := \R^{2n}$, $Z := \R$ and define $\beta: V \times V \to Z$ by
\[
\beta(v,w) = \sum_{i=1}^n v_i w_i -\sum_{j=n+1}^{2n} v_jw_j.
\]
Then $N$ is the $(2n+1)$-dimensional \emph{Heisenberg group}.
\end{example}

\newpage

\section{Construction of exotic universally aligned towers}\label{AppCrazyTower}
In the body of this article we establish the existence of universally aligned abelian characteristic subgroups (Theorem \ref{IntroTower}) by appealing to Machado's embedding theorem. In the first version of this article, when we were still unaware of Machado's result, we provided a different construction, which does not rely on any embedding arguments. We believe that this construction is of independent interest, and hence we record it in this appendix.
\subsection{The projection theorem for $2$-step nilpotent Lie groups}
In this subsection, $G$ denotes a $1$-connected $2$-step nilpotent Lie group. It follows from Machado's embedding theorem that the center $Z(G)$ of $G$ is universally aligned. Here we give an elementary proof of this fact, which does not rely on the embedding theorem.
\begin{proposition}[Projection theorem for $2$-step nilpotent Lie groups]\label{ProjectionTheorem} Let $G$ be a $1$-connected $2$-step nilpotent Lie group and denote by $\pi: G \to G/Z(G)$ the canonical projection. If $\Lambda \subset G$ is Meyerian, then $\Delta := \pi(\Lambda) \subset V$ is uniformly discrete, hence Meyerian (i.e.\ a Meyer set). In particular, every approximate lattice in $G$ is $\pi$-aligned.
\end{proposition}

\begin{proof} By Proposition \ref{NormalFormNilpotent} we may assume that $G = Z \oplus_\beta V$, where $Z \cong Z(G)$ and $V\cong G/Z(G)$ are finite-dimensional real vector spaces and $\beta: V \times V \to Z$ is a non-degenerate, antisymmetric bilinear map, and that $\pi$ is the projection onto the second coordinate. Then by \eqref{CentralInverse} the inverse of an element $(z,v) \in G$ is given by $(z, v)^{-1} = (-z, -v)$, and hence the commutator of elements $(z_1, v_1), (z_2, v_2) \in G$ is given by
\[[(z_1,v_1), (z_2, v_2)] = (2\beta(v_1,v_2), 0).
\]
This implies in particular, that if $A$ and $B$ are arbitrary subsets of $G$, then 
\begin{equation}\label{CommutatorSet}
\beta(\pi(A), \pi(B)) \subset \frac{1}{2}(ABA^{-1}B^{-1})_0,
\end{equation}
where $(ABA^{-1}B^{-1})_0$ denotes the fiber of $(ABA^{-1}B^{-1})$ over $0 \in V$.

Now let $\Lambda \subset G$ be Meyerian. Since $\Lambda$ is relatively dense in $G$ we can choose a compact subset $L \subset G$ such that $G = \Lambda L$. Then $K := \pi(L)$ is compact and $V = \pi(G) = \Delta+K$, which implies that
\begin{equation}\label{DifferenceXiSyndetic}
V = (\Delta - \Delta) + K.
\end{equation}
We are going to establish the following two properties of $\Delta_1 := \Delta - \Delta$:
\begin{enumerate}[(1)]
\item There is a finite subset $A \subset \Delta_1$ such that $A^\perp := \{v \in V\mid \beta(a,v) = 0\} = \{0\}$.
\item $\beta(\Delta_1-\Delta_1, \Delta_1) \subset Z$ is discrete.
\end{enumerate}
Let us show that the proposition follows from (1) and (2) by showing that (1) and (2) imply that $\Delta$ satisfies Property (Me1) of Lemma \ref{MeyerAbelian} which characterizes Meyer sets. To see that $\Delta_1$ is uniformly discrete we choose a sequence $(\delta_n)$ in $\Delta_1- \Delta_1$ which converges to $0$. Since $\beta(\Delta_1-\Delta_1, A) \subset \beta(\Delta_1-\Delta_1, \Delta_1)$ is discrete we find for every $a \in A$ a positive integer $n(a)$ such that
\[
\beta(\delta_n, a) = \lim_{n \to \infty} \beta(\delta_n, a)  = \beta(\lim_{n \to \infty} \delta_n, a)= 0 \quad \text{for all }n \geq n(a).
\]
Thus if $n \geq \max_{a \in A}n(a)$, then $\delta_n \in A^\perp$ and hence $\delta_n = 0$. This shows that $0$ is an isolated point of $\Delta_1-\Delta_1$, hence $\Delta_1 = \Delta -\Delta$ is uniformly discrete. This establishes (Me1) and it thus remains only to establish (1) and (2).

Towards the proof of (1) we first observe that the $\beta$-complement $\Delta_1^\perp \subset V$ of $\Delta$ is trivial. Indeed, assume that $v \in V$ with $\beta(v,\Delta_1) = 0$ and define a linear map $L_v: V \to Z$ by $L_v(p) := \beta(v, p)$. Then $L_v(\Delta-\Delta) = 0$ and hence by \eqref{DifferenceXiSyndetic} the image
\[
L_v(V) = L_v(\Delta-\Delta+K)= L_v(K)
\]
is bounded. This implies $L_v \equiv 0$, whence $v  \in V^\perp$. Since $\beta$ is non-degenerate this implies $v = 0$, and hence 
\[
 \{0\} = \Delta_1^\perp =  \bigcap_{\delta \in \Delta_1} \delta^\perp.
\]
Now if an infinite intersection of subspaces of a finite-dimensional vector space is trivial, then already finitely many of them intersect trivially. We can thus find a \emph{finite} subset  $A  \subset \Delta_1$ such that $A^\perp = \{0\}$. This establishes (1).

Concerning (2), by \eqref{CommutatorSet} we have
\begin{eqnarray*}
\beta(\Delta_1-\Delta_1, \Delta_1) &=& \beta(-\Delta+\Delta - \Delta + \Delta, -\Delta+\Delta)\\
&=&  \beta(\pi((\Lambda^{-1}\Lambda)^2), \pi(\Lambda^{-1}\Lambda))\\
&\subset&  \frac{1}{2}((\Lambda^{-1}\Lambda)^6)_0.
\end{eqnarray*}
Since $\Lambda$ is Meyerian, it follows from Property (M2) of Lemma \ref{MProperties} that $(\Lambda^{-1}\Lambda)^6$ is uniformly discrete. This implies (2) and finishes the proof.
\end{proof}

\subsection{$2$-step nilpotent characteristic subgroups of nilpotent Lie groups}
Let $G$ be a $1$-connected Lie group which is at least $3$-step nilpotent. We are going to consider a special closed characteristic subgroup, which is at most $2$-step nilpotent. We first formulate the corresponding results on the level of Lie algebras
\begin{lemma}\label{CGGGLieAlg}
Let $\L g$ be a nilpotent Lie algebra and let $\L n := \L{c_g([g,g])}$. 
\begin{enumerate}[(i)]
\item $\L n$ is either abelian or $2$-step nilpotent.
\item $\L n$ and its center $\L{z(n)}$ are characteristic ideals in $\L g$.
\item If $\L g$ is at least $3$-step nilpotent, then $\L n$ and $\L{z(n)}$ are non-trivial, i.e. neither $\{0\}$ nor equal to $\L g$.
\end{enumerate}
\end{lemma}
\begin{proof} (i) follows from $[\L n,[\L n, \L n]] \subset [\L n, [\L g, \L g]] = \{0\}$. 

(ii) The commutator ideal of a Lie algebra is always characteristic, as is the centralizer of a characteristic ideal and in particular the center. This shows that $\L n$ is characteristic in $\L g$ and $\L{z(n)}$ is characteristic in $\L n$, hence also $\L{z(n)}$ is characteristic in $\L g$.

(iii) That $\L n \neq \L g$ is immediate from (i), and we have $\{0\} \subsetneq \L{z(g)} \subset \L{z(n)} \subset \L n$. 
\end{proof}
The following examples show that there are cases in which $\L n$ is abelian and cases in which $\L n$ is $2$-step nilpotent; they also show that $\L{z(n)}$ is in general not central in $\L g$.
\begin{example} Let $\L g := \L u_n < \L{gl}_n(\R)$ denote the Lie algebra of strictly upper triangular matrices for some $n \geq 4$. Then $\L{n}$ is a $3$-dimensional abelian Lie algebra generated by the elementary matrices $E_{1,{n-1}}$, $E_{1,n}$ and $E_{2,n}$.
\end{example}
\begin{example} Let $\L g < \L u_5$ be the $7$-dimensional subalgebra generated by the elementary matrices $E_{12}$, \dots, $E_{15}$, $E_{24}$, $E_{25}$, $E_{35}$ and $E_{45}$. Then $\L n$ is the $5$-dimensional $2$-step nilpotent Lie algebra generated by $E_{13}$, $E_{14}$, $E_{15}$, $E_{25}$ and $E_{35}$. Its center is a $3$-dimensional abelian Lie algebra which contains the center of $\L g$ as a codimension $2$ subalgebra.
\end{example}
On the group level we have the following corresponding results.
\begin{lemma}\label{CGGG} Let $G$ be a $1$-connected nilpotent Lie group with Lie algebra $\L g$ and let $N := C_G([G, G])$ and $A := Z(N)$.
\begin{enumerate}[(i)]
\item The Lie algebras of $N$ and $A$ are given by $\L n := \L{c_g([g,g])}$ and $\L a := \L{z(c_g([g,g]))}$
\item $N$ and $A$ are $1$-connected and closed, and hence $N = \exp(\L n)$ and $A = \exp(\L a)$.
\item $N$ and $A$ are characteristic in $G$.
\item $N$ is either abelian or $2$-step nilpotent and $A$ is abelian.
\item If $G$ is $n$-step nilpotent for some $n \geq 3$, then $N$ and $A$ are non-trivial.
\end{enumerate}
\end{lemma}
\begin{proof} (i) This is immediate from the definitions. (ii) We apply Remark \ref{ClosedSubgpNilpotentGp} three times: Firstly, $[G,G]$ is closed and connected (hence $1$-connected and exponential). Secondly, $N$ as the centralizer of a closed and connected subgroup is closed and connected. Finally, $A$ as the center of a closed and connected subgroup is closed and connected. (iii) follows from the fact that commutator subgroups and centralizers are characteristic, and (iv) and (v) follow from (i), (ii) and Lemma \ref{CGGGLieAlg}.
\end{proof}
From the examples after Lemma \ref{CGGGLieAlg} we see that there are cases where $N$ is abelian and cases where $N$ is $2$-step nilpotent, and that $A$ need not be central in $G$.

\subsection{Relatively dense subsets and Zariski dense subsets}
\begin{lemma}\label{PowerZariskiDenseLCS} Let $G$ be a $1$-connected nilpotent Lie group with lower central series $(G_n)_{n \in \mathbb N}$ and let $X \subset G$ be Zariski dense. Then for every $n \in \mathbb N$ there exists $\sigma(n) \in \mathbb N$ such that $X^{\sigma(n)} \cap G_n$ is Zariski dense in $G_n$.
\end{lemma}
\begin{proof} We consider the family of maps $(p_k)_{k \geq 1}$ given by
\[
p_k: G^{\times kn} \to G_n, \quad (g_{11}, \dots, g_{1n}, \dots, g_{k1}, \dots, g_{kn}) \mapsto \prod_{i=1}^k [g_{i1}, [g_{i2}, \dots, [g_{i(n-1)}, g_{in}]\dots]].
\]
By definition, the group $G_n$ is the ascending union
\[
G_n = \bigcup_{k=1}^\infty p_k(G^{\times kn}).
\]
Since it is closed, it contains each of the sets $\overline{p_k(G^{\times kn})}$. In particular,
\[
G_n = \bigcup_{k=1}^\infty p_k(G^{\times kn}) \subset \bigcup_{k=1}^\infty \overline{p_k(G^{\times kn})} \subset G_n \quad \Rightarrow \quad G_n =  \bigcup_{k=1}^\infty \overline{p_k(G^{\times kn})}.
\]
By the Baire category theorem we thus deduce that there exists $k \in \mathbb N$ such that $ \overline{p_k(G^{\times kn})}$ has non-empty interior in $G_n$. This implies that  $ \overline{p_k(G^{\times kn})}$  is Zariski dense in $G_n$, and hence $p_k(G^{\times kn})$ is Zariski dense in $G_n$. On the other hand, since $X$ is Zariski dense in $G$, also $X^{\times kn}$ is Zariski dense in $G^{\times kn}$ and hence $p_k(X^{\times kn})$ is Zariski dense in the Zariski closure of $p_k(G^{\times kn})$, which is $G_n$. Since $p_k(X^{\times kn}) \subset X^{\sigma(n)}$ for some sufficiently large $\sigma(n)$, we deduce that $X^{\sigma(n)}$ is Zariski-dense in $G_n$ as well.
\end{proof}

\begin{lemma}\label{RelDenseZariskiDense} Let $G$ be a connected unipotent algebraic subgroup of ${\rm GL}_n(\R)$ for some $n \geq 2$,
and let $X \subset G$ be relatively dense. Then the subgroup $\langle X \rangle < G$ is Zariski dense.
\end{lemma}
\begin{proof} Denote by $H$ the Zariski closure of $\Gamma := \langle X \rangle$ in $G$. Since $H$ is nilpotent, it is unimodular, and since $H$ contains $X$, it is relatively dense in $G$, hence cocompact. It thus follows that the homogeneous space $G/H$ admits a unique $G$-invariant (hence $G$-ergodic) probability measure $\mu_{G/H}$. On the other hand, by Chevalley's theorem there exists an algebraic representation $\rho: G \to {\rm GL}(W)$ and a line $L \subset W$ such that $H$ is the setwise stabilizer of $L$. We thus obtain an embedding $\phi: G/H \hookrightarrow \mathbb P(W)$, $gH \mapsto gL$, and we denote by $\mu := \phi_*\mu_{G/H}$ the push-forward of $\mu_{G/H}$, which is invariant and ergodic under $\rho(G)$. Now $\rho(G) \subset {\rm GL}(W)$ is again a unipotent group, since $\rho$ is algebraic. It thus follows from \cite[Lemma 3]{Furstenberg} %(see also \cite[Cor. 1.8]{Moskowitz}) 
that $\mu$ is supported on the $\rho(G)$-fixpoints in $\mathbb P(W)$; since it is ergodic it is thus given by a Dirac measure at a single point. We deduce that $G/H$ is a point, i.e. $G = H$, hence $\Gamma$ is Zariski dense in $G$.
\end{proof}
Combining the previous two lemmas (or rather their proofs) one obtains:
\begin{corollary}\label{RelDenseZariskiDense2} Let $G$ be a $1$-connected nilpotent Lie group with lower central series $(G_n)_{n \in \mathbb N}$ and let $X \subset G$ be relatively dense. Then for every $n \in \mathbb N$ there exists $\sigma(n) \in \mathbb N$ such that $X^{\sigma(n)} \cap G_n$ is Zariski dense in $G_n$.
In particular, the subgroup $\langle X \cap G_n \rangle < G_n$ is Zariski dense.
\end{corollary}
\begin{proof} As in the proof of Lemma \ref{PowerZariskiDenseLCS} one show that if we define
\[
p_k: G^{\times kn} \to G_n, \quad (g_{11}, \dots, g_{1n}, \dots, g_{k1}, \dots, g_{kn}) \mapsto \prod_{i=1}^k [g_{i1}, [g_{i2}, \dots, [g_{i(n-1)}, g_{in}]\dots]],
\]
then  $p_k(G^{\times kn})$ is Zariski dense in $G_n$ for some $k$. On the other hand, $X$ is relatively dense in $G$, hence $X^{\times kn}$ is relatively dense in $G^{\times kn}$. It follows from Lemma \ref{RelDenseZariskiDense} that $\langle X \rangle^{\times kn}$ is Zariski dense in $G^{\times kn}$. We deduce that $p_k(\langle X \rangle^{\times kn})$ is Zariski dense in the Zariski closure of $p_k(G^{\times kn})$, which is $G_n$. Since
\[
\langle X \cap G_n \rangle \supset p_k(\langle X \rangle^{\times kn}),
\]
we conclude that $\langle X \cap G_n \rangle$ is Zariski dense in $G_n$ as well.
\end{proof}
\subsection{The first projection theorem for higher step nilpotent Lie groups}
We now provide two projection theorems concerning uniform approximate lattices in $1$-connected Lie groups which are $\geq 3$-step nilpotent. We recall from Lemma \ref{CGGG} that if $G$ is a $1$-connected Lie group, which is at least $3$-step nilpotent, then $N:= C_G([G,G])$ is a  closed $1$-connected characteristic subgroup $G$, which is proper, non-trivial and at most $2$-step nilpotent. Now the first version of our projection theorem can be stated as follows:
\begin{theorem}[First projection theorem] \label{FirstProjectionTheorem} Let $G$ be a $1$-connected Lie group which is at least $3$-step nilpotent and let $\Lambda \subset G$ be Meyerian. Let $N:= C_G([G,G])$ and let $\pi: G \to G/N$ denote the canonical projection. Then $\Delta := \pi(\Lambda)$ is uniformly discrete in $G/N$. In particular, $\pi$ is universally aligned.
\end{theorem}
For the proof we will need the following observation:
\begin{lemma}\label{GammaZDense} With the notation of Theorem \ref{FirstProjectionTheorem} the group $\Gamma := \langle \Lambda^{-1}\Lambda \cap [G,G] \rangle$ is finitely generated and Zariski dense in $[G,G]$.
\end{lemma}
\begin{proof} That $\Gamma$ is Zariski dense in $[G,G]$ is a special case of Corollary \ref{RelDenseZariskiDense2}. Since $G$ is connected, it is compactly generated, and since $\Lambda^{-1}\Lambda$ is an approximate lattice in $G$, we deduce with \cite[Thm. 2.22]{BH1} that $\langle \Lambda^{-1}\Lambda \rangle$ is finitely-generated. Since finitely-generated nilpotent groups are Noetherian, we deduce that every subgroup of $\langle \Lambda^{-1}\Lambda \rangle$ is also finitely-generated. In particular, $\Gamma$ is finitely-generated, and this finishes the proof.
\end{proof}

\begin{proof}[Proof of Theorem \ref{FirstProjectionTheorem}]  For all $a \in [G,G]$, $n \in N$ and $x \in G$ we have $nan^{-1} = a$ and hence
\[
[xn, a]= xn an^{-1}x^{-1}a^{-1} = xax^{-1}a^{-1} =[x,a].
\]
We can this define a continuous map 
\begin{equation}\label{psia}
\psi_a: G/N \to [G,G], \quad \psi_a(xN) := [x,a].
\end{equation}
Now assume for contradiction that $\pi(\Lambda)$ is \emph{not} uniformly discrete. Then there exist $\mu_j \in \Lambda^{-1}\Lambda \setminus (N \cap \Lambda^{-1}\Lambda)$ such that $\pi(\mu_j) \to e$. We may thus choose elements $n_j \in N$ such that $\mu_jn_j \to e$.

Let $a \in [G,G]$ and consider the map $\psi_a$ from \eqref{psia}. Since $\psi_a$ is continuous we deduce that for all $a \in [G,G]$ we have
\[
\lim_{j \to \infty}[\mu_j, a] = \lim_{j \to \infty} \psi_a(\mu_jN) = \lim_{j \to \infty} \psi_a(\mu_jn_j N) = \psi_a(N) = e.
\]
Now assume that $a \in \Gamma$, where $\Gamma$ is as in Corollary \ref{GammaZDense}. Then there exists $k \in \mathbb N$ such that $a \in [G,G] \cap (\Lambda^{-1}\Lambda)^k$, and thus $[\mu_j, a] \in (\Lambda^{-1}\Lambda)^{2k+2}$, which is uniformly discrete. Since $[\mu_j, a] \to e$ we deduce that there exists $j(a) \in \mathbb N$ such that
\begin{equation}
[\mu_j, a] = e \text{ for all }j \geq j(a).
\end{equation}
Now recall from Corollary \ref{GammaZDense} that $\Gamma$ is finitely-generated, and let $F$ be a finite generating set for $\Gamma$. We set
\[
j_0 := \max\{j(a) \mid a \in F\}.
\]
Then for all $j \geq j_0$ and all $a \in F$ we have $[\mu_j, a] = e$, i.e. $a$ centralizes $\mu_j$. Since $\langle F \rangle = \Gamma$ we deduce that 
\[
[\mu_j, a] = e \text{ for all } j \geq j_0 \text{ and }a \in \Gamma.
\]
Since $\Gamma$ is Zariski dense in $[G,G]$ we conclude that even
\[
[\mu_j, a] = e \text{ for all } j \geq j_0 \text{ and }a \in [G,G].
\]
This implies that for $j \geq j_0$ we have $\mu_j \in C_G([G,G]) = N$, contradicting the choice of the $\mu_j$.
\end{proof}
\subsection{The second projection theorem for higher step nilpotent Lie groups}
Combining Theorem \ref{FirstProjectionTheorem} and Proposition \ref{ProjectionTheorem} we end up with the following theorem.
\begin{theorem}[Second projection theorem] \label{SecondProjectionTheorem} Let $G$ be a $1$-connected Lie group which is at least $3$-step nilpotent, let $A:= Z(C_G([G,G]))$ and denote by $\pi_1: G \to G/A$ denote the canonical projection. Then $\pi_1$ is universally aligned.
\end{theorem}
\begin{proof} If $N:= C_G([G,G])$ is already abelian, this follows from the first projection theorem. We may thus assume that $N$ is $2$-step nilpotent. We then introduce the following notation: We denote by $\pi: G \to G/N$, $\pi_2: G/A \to G/N$ and $\phi: N \to N/A$ the canonical projections. By Remark \ref{ClosedSubgpNilpotentGp} we can choose continuous sections $s_1, s_2$ of $\pi_1, \pi_2$ respectively, and we may assume that $s_1(e) = e$. The following diagram illustrates the various maps:
\[\begin{xy}\xymatrix{
G \ar[d]^{\pi_1} \ar@/_3pc/[dd]_{\pi} &&&&&N \ar[dd]^\phi\\
G/A \ar[d]^{\pi_2} \ar@/_2pc/[u]_{s_1}&&&&&\\
G/N \ar@/_2pc/[u]_{s_2}&&&&&N/A
}\end{xy}\]
Now assume that $\Lambda \subset G$ is a uniform approximate lattice. We define $\Delta := \pi_1(\Lambda) \subset G/A$, $\Theta := \pi_2(\Delta) = \pi(\Lambda) \subset G/N$, $\Xi := \Lambda^2 \cap N \subset N$ and $\phi := \phi(\Xi) \subset N/A$. By the first projection theorem, $\Theta$ is an approximate lattice in $G/N$, and $\Xi$ is an approximate lattice in $N$ (since it is the $\pi$-fiber over $e$ of $\Lambda^2$). It then follows from the projection theorem for $2$-step nilpotent groups that $\phi$ is an approximate lattice in $N/A$. 

We need to show that $\Delta$ is uniformly discrete. Thus assume for contradiction that there exist $\delta_j, \eta_j \in \Delta \subset G/A$ such that $\delta_j \neq \eta_j$ for all $j \in \mathbb N$ and $\delta_j^{-1}\eta_j \to e$. Then
$\pi_2(\delta_j), \pi_2(\eta_j) \in \Theta$ and
\[
\pi_2(\delta_j)^{-1}\pi_2(\eta_j) \to e.
\]
Since $\Theta$ is an approximate lattice it follows that for sufficiently large $j$ we have 
\begin{equation}
\pi_2(\delta_j) = \pi_2(\eta_j) %=: \nu_j.
\end{equation}
Now since $\delta_j, \eta_j \in \Delta = \pi_1(\Lambda)$ we can find $\lambda_j, \mu_j \in \Lambda$ such that $\pi_1(\lambda_j) = \delta_j$ and $\pi_1(\mu_j) = \eta_j$. Set
$m_j := \lambda_j s_1(\delta_j)^{-1}$ and $n_j := \mu_j s_1(\eta_j)^{-1}$. Then $m_j, n_j \in \ker(\pi_1) = A$ and
\[
\lambda_j = m_j s_1(\delta_j) \qand \mu_j = n_j s_1(\eta_j).
\]
We deduce that
\[
\pi(\lambda_j^{-1}\mu_j) = \pi_2(\pi_1(s_1(\delta_j)^{-1}m_j^{-1}n_j s_1(\eta_j))) = \pi_2(\delta_j)^{-1}\pi_2(\eta_j) = e,
\]
hence $\lambda_j^{-1}\mu_j \in \Lambda^{-1}\Lambda \cap \ker(\pi) = \Lambda^{-1}\Lambda \cap N = \Xi$, and thus $\phi(\lambda_j^{-1}\mu_j) \in \phi$. On the other hand, since $\phi(m_j) = \phi(n_j) = e$ and $s_1(\delta_j)^{-1}s_1(\eta_j)s_1(\delta_j^{-1}\eta_j)^{-1} \in \ker(\pi_1) = A = \ker(\phi)$ we have
\begin{eqnarray*}
\phi(\lambda_j^{-1}\mu_j) &=& \phi(s_1(\delta_j)^{-1}m_j^{-1}n_j s_1(\eta_j))\quad = \quad \phi(s_1(\delta_j)^{-1}s_1(\eta_j))\\
&=& \phi(s_1(\delta_j)^{-1}s_1(\eta_j)s_1(\delta_j^{-1}\eta_j)^{-1} s_1(\delta_j^{-1}\eta_j))\\
&=& \phi(s_1(\delta_j^{-1}\eta_j)),
\end{eqnarray*}
which converges to $(\phi \circ s_1)(e) = e$. Since $\phi$ is uniformly discrete this is only possible if $\phi(\lambda_j^{-1}\mu_j) = e$ for all sufficiently large $j$. We have thus proved that for all sufficiently large $j$ we have $ \phi(s_1(\delta_j^{-1}\eta_j)) = e$. This implies that for large $j$ we have $s_1(\delta_j^{-1}\eta_j) \in \ker(\phi) = A$, and hence
\[
\delta_j^{-1}\eta_j = \pi_1(s_1(\delta_j^{-1}\eta_j)) \in \pi_1(A) = \{e\},
\]
which leads to the contradiction $\delta_j = \eta_j$. This finishes the proof.
\end{proof}

\newpage
\section{On theorems of Meyer and Dani--Navada}\label{AppendixMeyer}

\subsection{Statements of the original theorems in the abelian setting}
Meyer's original motivation to study Meyer sets came from number theory, and more specifically from his study of Pisot--Salem numbers. 
\begin{definition} A real number $\lambda$ is called a 
\emph{Pisot--Salem number} if it is an algebraic integer, greater than $1$ and if each of its Galois conjugates over $\Q$ has absolute value less or equal to $1$. It is called a \emph{Pisot number} if the Galois conjugates have absolute values strictly less than $1$, otherwise a \emph{Salem number}.
\end{definition}
Meyer has related Pisot--Salem numbers are related to Meyer sets by the following theorem \cite{MeyerBook}. Here, given a real number $t >0$ we denote by $\delta_t: \R \to \R$ the dilation $\delta_t(x) := t x$.
\begin{theorem}[Meyer]\label{MeyerPisot} If $t > 1$ and $\Lambda \subset \R^n$ is a Meyer set such that $\delta_t(\Lambda) \subset \Lambda$, then $t$ is a Pisot--Salem number. Conversely, if $t$ is a Pisot--Salem number, then there exists a Meyer set $\Lambda \subset \R$ such that $\delta_t(\Lambda) \subset \Lambda$.
\end{theorem}
Now let $T$ be an arbitrary continuous group automorphism of $\R^n$ which maps $\Lambda$ into itself, but not necessarily a dilation. One may then ask whether the eigenvalues of $T$ which are $>1$ are Pisot-Salem numbers. If $n=1$ then the answer to this question is positive, since up to a sign $T$ is necessarily a dilation. For $n>1$ the answer is negative: Every matrix $M \in {\rm SL}_n(\Z)$ preserves the Meyer set $\Z^n \subset \R^n$, and the eigenvalues of $M$ can be arbitrary algebraic integers. However, in this example all Galois conjugates of an eigenvalue of $M$ are also eigenvalues of $M$. If one excludes such arithmetic conspiracies, then one can establish a version of Meyer's theorem for automorphisms of $\R^n$ as well. Indeed, the following is a special case of a theorem of Dani and Navada:
\begin{theorem}[Dani--Navada]\label{Dani} Let $T$ be a continuous group automorphism of $\R^n$ which preserves a Meyer set $\Lambda \subset \R^n$. Assume that $T$ has simple spectrum $\sigma(T)$. Then the following hold for every eigenvalue $t \in \sigma(T)$.
\begin{enumerate}[(i)]
\item $t$ is an algebraic integer.
\item If $t'$ is Galois conjugate to $t$ over $\Q$, then either $t' \in \sigma(T)$ or $|t'|\leq 1$.
\end{enumerate}
In particular, if no two eigenvalues of $\sigma(T)$ are Galois conjugate over $\Q$, then every eigenvalue $t>1$ of $T$ is a Pisot--Salem number.
\end{theorem}
For a more general version of the theorem which also applies to the case where the spectrum of $T$ is not simple, we refer the interested reader to the original paper of Dani and Navadi \cite{DaniNavada}.
\subsection{Statements of the theorems in the nilpotent setting}
We now formulate versions of the theorems of Meyer and Dani--Navada in the context of certain classes of nilpotent Lie groups. The natural context of Meyer's theorem is that of \emph{stratified Lie groups}, which we briefly recall.

Let $G$ be a $1$-connected Lie group with Lie algebra $\L g$. A \emph{stratification} of $\L g$ is a vector space decomposition
\[
\L g = V_1 \oplus \dots \oplus V_s
\]
such that  for all $j = \{1, \dots, s\}$ we have $V_{j} \neq \{0\}$ and $[V_1, V_j] = V_{j+1}$ (with the convention $V_{s+1} := \{0\}$). The group $G$ together with a stratification of $\L g$ is then called a \emph{stratified Lie group}; any such group is evidently nilpotent (but not every nilpotent Lie group can be stratified). 

Given a stratified Lie group $G$ and $t > 0$ we define an automorphism $\delta_t \in {\rm Aut}(G)$ by demanding that its differential $(d\delta_t)_e \in {\rm Aut}(\L g)$ maps $X$ to $t^j X$ for all $X \in V_j$. We refer to $\delta_t$ as the \emph{dilation} of $G$ with \emph{dilation factor} $t$. We are going to show:
\begin{theorem}\label{Meyer+} Let $G$ be a $1$-connected stratified Lie group. If $t > 1$ and $\Lambda \subset G$ is a Meyerian set such that $\delta_t(\Lambda) \subset \Lambda$, then the dilation factor $t$ is a Pisot--Salem number.
\end{theorem}
Concerning the theorem of Dani--Navada we have the following generalization to the nilpotent case:
\begin{theorem}\label{Dani+} Let $G$ be a $1$-connected nilpotent Lie group with Lie algebra $\L g$ and let $T$ be a continuous group automorphism of $N$ which preserves a Meyerian set $\Lambda \subset G$. Assume that its linearization $dT_e \in {\rm Aut}(\L g)$ has simple spectrum $\sigma(dT_e)$. Then the following hold for every eigenvalue $t \in \sigma(dT_e)$.
\begin{enumerate}[(i)]
\item $t$ is an algebraic integer.
\item If $t'$ is Galois conjugate to $t$ over $\Q$, then either $t' \in \sigma(dT_e)$ or $|t'|\leq 1$.
\end{enumerate}
In particular, if no two eigenvalues of $\sigma(dT_e)$ are Galois conjugate over $\Q$, then every eigenvalue $t>1$ of $dT_e$ is a Pisot--Salem number.
\end{theorem}

\subsection{Universally aligned towers and automorphisms of nilpotent Lie groups}
For the rest of this appendix let $G$ be a $1$-connected nilpotent Lie group. We recall from Theorem \ref{IntroTower} that $G$ admits a universally aligned characteristic tower. We fix such a tower $G = G_1, \dots, G_{n+1}$ once and for all. 

In order to study the behaviour of a general automorphisms of $G$ along the tower $G_1, \dots, G_{n+1}$ we introduce the following notation: Given $j \in \{1, \dots, n\}$, let \[A_j := \ker{\pi_j^{j_1}: G_j \to G_{j+1}}\] and given $j \in \{1, \dots, n+1\}$ let \[K_j := \ker(\pi_1^{j}: G \to G_{j+1}).\] We also denote by $\L g = \L g_1, \dots, \L g_{n+1}, \L a_1, \dots, \L a_{n}, \L k_1, \dots, \L k_{n+1}$ the respective Lie algebras of the groups $G = G_1$, \dots, $G_{n+1}$, $A_1$, \dots, $A_n$, $K_1$, \dots, $K_{n-1}$. Note that if $j \in \{1, \dots, n\}$ then we have $G_j = G/K_j$ and $G_{j+1} = G/K_{j+1}$ and hence $A_j = K_{j+1}K_j$. We choose a basis
\begin{equation}\label{BasisAppendixMeyer}
(X^{(1)}_1, \dots, X^{(1)}_{\dim \L a_1}, X^{(2)}_1, \dots, X^{(2)}_{\dim \L a_2}, \dots, X^{(n)}_1, \dots, X^{(n)}_{\dim \L a_n}, X^{(n+1)}_1, \dots, X^{(n+1)}_{\dim \L g_{n+1}})
\end{equation}
of $\L g$ such that $(X^{(1)}_1, \dots, X^{(1)}_{\dim \L a_1})$ is a basis of $\L a_1 = \L k_1$,  $(X^{(1)}_1, \dots, X^{(1)}_{\dim \L a_1}, X^{(2)}_1, \dots, X^{(2)}_{\dim \L a_2})$ is a basis of $\L k_2$ etc. Since  $\L a_j = \L k_{j+1}/\L k_{j}$ the images $(\bar X^{(j)}_1, \dots, \bar X^{(j)}_{\dim \L a_1})$ of $(X^{(j)}_1, \dots, X^{(j)}_{\dim \L a_1})$ modulo $\L k_{j+1}$ generate $\L a_j$.

Now let  $T \in {\rm Aut}(G)$. Since $A_1$ is characteristic in $G$, it is invariant under $T$, and in particular, $T$ descends to an automorphism of $G_2 = G/A_1$. Inductively we see that $T$ descends to an automorphism of each $G_j$ and preserves all of the characteristic subgroups $A_j$. It then follows that it also preserves all of the subgroup $K_j$. We will denote by $T\upharpoonright G_j$ and $T\upharpoonright A_j$ the induced automorphisms and by $dT_e: \L g \to \L g$ the differential of the automorphism $T$, which is a Lie algebra automorphism of $\L g$. We observe that the flag $\L k_1 \subset \L k_2 \subset \dots \subset \L k_{n+1}$ is invariant under $dT_e$, hence the matrix $M$ of $dT_e$ with respect to the basis from \eqref{BasisAppendixMeyer} is of upper triangular block form, i.e.
\[
M = \left(\begin{matrix} 
M_1 & * & \hdots & * \\
& M_2 & \ddots & \vdots\\
&&\ddots&*\\
0&&& M_{n+1}
\end{matrix}\right)
\]
Moreover, for $j \in \{1, \dots, n\}$ the matrix $M_j$ represents the automorphism $d(T\upharpoonright A_j)_e \in {\rm Aut}(\L a_j)$ with respect to the basis $(\bar X^{(j)}_1, \dots, \bar X^{(j)}_{\dim \L a_1})$, and similarly $M_{n+1}$ represents $d(T\upharpoonright G_{n+1})_e \in {\rm Aut}(\L g_{n+1})$. We deduce in particular:
\begin{proposition}\label{DaniProposition} For every $T \in {\rm Aut}(G)$ the characteristic polynomial $\chi_{dT_e}$ of the linearization $dT_e \in {\rm Aut}(\L g)$ is given by
%\begin{equation}\label{CharPolFactorization}
\[\chi_{dT_e} = \chi_{d(T\upharpoonright A_1)_e} \cdots \chi_{d(T\upharpoonright A_n)_e} \chi_{d(T\upharpoonright G_{n+1})_e}.\] 
In particular, the spectrum of $dT_e$ is given by
\[
\sigma(dT_e) = \sigma(d(T\upharpoonright A_1)_e) \cup \dots \cup \sigma(d(T\upharpoonright A_n)_e) \cup \sigma(d(T\upharpoonright G_{n+1})_e).\qed
\]
\end{proposition}

\subsection{Reduction to the abelian case}
We can now reduce Theorem \ref{Meyer+} to Theorem \ref{MeyerPisot} and Theorem \ref{Dani+} to Theorem \ref{Dani} by means of the following proposition; here notation is as in the previous subsection.
\begin{proposition}\label{FromMeyerianToMeyerTower} Let $T \in {\rm Aut}(G)$ and let $\Lambda \subset G$ be a Meyerian subset. If $T(\Lambda) \subset \Lambda$, then for every $j\in \{1, \dots, n\}$ the automorphism $T\upharpoonright A_j \in {\rm Aut}(A_j)$ maps a Meyer set $\Xi_j \subset A_j$ into itself, and the automorphism $T\upharpoonright G_{n+1} \in {\rm Aut}(G_{n+1})$ maps a Meyer set $\Xi_{n+1} \subset G_{n+1}$ into itself.
\end{proposition}
\begin{proof} We know from Theorem \ref{UniversallyAlignedTower} and Corollary \ref{CorSquares} that for $j \in \{1, \dots, n\}$ the subset $\Xi_j := \pi_1^j(\Lambda^{-1}\Lambda)^2 \cap A_j$ is a uniform approximate lattice in $A_j$, and that $\Xi_{n+1} :=  \pi_1^j(\Lambda^{-1}\Lambda)$ is a uniform approximate lattice in $G_{n+1}$. Since $T$ is an automorphism which maps $\Lambda$ into itself, it also maps the uniform approximate lattice $\Lambda^{-1}\Lambda$ and its square $\Xi := (\Lambda^{-1}\Lambda)^2$ into themselves. It then follows for every $j\in \{1, \dots, n\}$ the automorphism $T\upharpoonright A_j$ maps $\Xi_j$ into itself, and that the automorphism $T\upharpoonright G_{n+1}$ maps $\Xi_{n+1}$ into itself.
\end{proof}

\begin{proof}[Proof of Theorem \ref{Meyer+}] Let $\delta_\lambda$ be a dilation of $G$. The linearization of $\delta_\lambda$ preserves the characteristic ideal $[\L g, \L g]$ and acts on the abelianization $\L g/[\L g, \L g]$. If $\L g = V_1 \oplus \dots \oplus V_s$ is the given filtration of the stratified Lie algebra $\L g$ of $G$, then by \cite[Lemma 1.16]{LeDonne} we have
\[
[\L g, \L g] = V_2 \oplus \dots \oplus V_s,
\]
hence the action of the linearization of $\delta_\lambda$ on the abelianization is given by multiplication by $\lambda$. Now $\L g_{n+1}$ is an abelian quotient of $\L g$, hence a quotient of $\L g/[\L g, \L g]$. It follows that the linearized action of $\delta_\lambda$ on $\L g_{n+1}$ is given by multiplication by $\lambda$. Now by Proposition \ref{FromMeyerianToMeyerTower} this action maps a Meyer set into itself, hence $\lambda$ is a Pisot--Salem number by Theorem \ref{MeyerPisot}.
\end{proof}

\begin{proof}[Proof of Theorem \ref{Dani+}] Since $dT_e$ has simple spectrum, we deduce from Proposition \ref{DaniProposition} that each of the linearizations $d(T\upharpoonright A_1)_e$, \dots, $d(T\upharpoonright A_n)_e$ and $(d(T\upharpoonright G_{n+1})_e$ has simple spectrum. Moreover, each of these linear operators maps a Meyer set in the corresponding vector space into itself by Proposition \ref{FromMeyerianToMeyerTower}. Now if $t \in \sigma(dT_e)$, then by Proposition \ref{DaniProposition} we have either $t \in \sigma(d(T\upharpoonright A_j)_e)$ for some $j \in \{1, \dots, n\}$ or $t \in \sigma(d(T\upharpoonright G_{n+1})_e)$. Either way we can apply Theorem \ref{Dani} to the corresponding linear operator to conclude that $t$ is algebraic and that all Galois conjugates of $t$ either have absolute value $<1$ or are contained in the spectrum of the respective linear operator, hence in the spectrum of $dT_e$. This finishes the proof. 
\end{proof}

\end{document}